\documentclass[12pt]{article}

\usepackage{amsmath,amssymb,amsthm,booktabs,mathrsfs,enumerate,framed,graphicx,array,color,multirow,amsrefs}

\usepackage[utf8]{inputenc}
\usepackage[colorlinks,citecolor=blue,urlcolor=blue]{hyperref}

\usepackage{bm}
\usepackage{euscript}

\usepackage{multicol}
\usepackage[usenames,dvipsnames,svgnames,table]{xcolor}

\usepackage{a4wide}

\usepackage{mathrsfs}
\usepackage{euscript}

\numberwithin{equation}{section} \theoremstyle{plain}
\newtheorem{theorem}{Theorem}[section]
\newtheorem{lemma}{Lemma}[section]
\newtheorem{corollary}{Corollary}[section]

\newtheorem{definition}{Definition}
\newtheorem{remark}{Remark}[section]

\allowdisplaybreaks

\allowdisplaybreaks[4]
\numberwithin{equation}{section}

\def\bC{\mathbb{C}}
\def\bN{\mathbb{N}}
\def\bR{\mathbb{R}}
\def\bE{\mathbb{E}}
\def\bP{\mathbb{P}}

\def\cH{\mathcal{H}}

\def\cB{\mathcal{B}}

\def\cF{\mathcal{F}}

\def\cE{\mathcal{E}}
\def\cL{\mathcal{L}}

\def\cC{\mathcal{C}}
\def\cI{\mathcal{I}}
\def\rd{\mathrm{diag}}

\def\spec{\mathbf{Spec}}

\def\dimh{\dim_{_{\rm H}}}

\def\1{\mathbf 1}
\def\im{\mathrm{Im}}

\begin{document}

\title{Multiple collisions of eigenvalues and singular values of matrix Gaussian field}

\date{\today}
\author{Wangjun Yuan\thanks{Department of Mathematics, University of Luxembourg. E-mail: ywangjun@connect.hku.hk}}
\maketitle

\begin{abstract}
Let $X^\beta$ be a real symmetric or complex Hermitian matrix whose entries are independent Gaussian random fields. We provide the sufficient and necessary conditions such that multiple collisions of eigenvalue processes of $A^\beta + T_\beta X^\beta T_\beta^*$ occur with positive probability. In addition, for a real or complex rectangular matrix $W^\beta$ with independent Gaussian random field entries, we obtain the sufficient and necessary conditions under which the probability of multiple collisions of non-trivial singular value processes of $B^\beta + T_\beta W^\beta \tilde T_\beta$ is positive. In both cases, the size of the set of collision times is characterized via Hausdorff dimension.
\end{abstract}

\noindent{\bf AMS 2020 subject classifications:} 15B52,~60B20,~60G15,~60G60,~15A18

\medskip 

\noindent{\bf Keywords and phrases:} random matrix; Gaussian random field; Hitting probabilities; singular values; eigenvalues.

\hypersetup{linkcolor=black}
\tableofcontents

\section{Introduction}

\subsection{Background}

The study eigenvalue processes of self-adjoint matrix-valued process can be dated back to at least \cite{Dyson1962}. Let $B(t)$ be an $d \times d$ real or complex matrix Brownian motion, i.e. $B(t)$ is an $d \times d$ matrix whose entries are i.i.d. real or complex standard Brownian motions. We define the self-adjoint matrix Brownian motion by
\begin{align} \label{eq-matrix BM}
    X^\beta(t) = \dfrac{1}{\sqrt{2\beta}} \left( B(t) + B(t)^* \right) + X^\beta(0),
\end{align}
where $\beta=1$ corresponds to the real case and $\beta=2$ is for the complex case.
It was obtained in \cite{Dyson1962} that the set of (ordered) eigenvalue processes $\{\lambda_i^\beta(t):1 \le i \le d\}$ of $X^\beta(t)$ solves the following stochastic differential equation (SDE):
\begin{align} \label{eq-Dyson BM}
    d\lambda_i^\beta(t) = \dfrac{\sqrt{2}}{\sqrt{\beta}} dW_i(t) + \sum_{j:j\neq i} \dfrac{1}{\lambda_i^\beta(t)-\lambda_j^\beta(t)} dt, \quad 1 \le i \le d, t \ge 0,
\end{align}
where the family $\{W_i(t):1 \le i \le d\}$ is a collection of real-valued i.i.d. standard Brownian motions. The solution to SDE \eqref{eq-Dyson BM} is known as Dyson Brownian motion.
Note that the SDE \eqref{eq-Dyson BM} is singular whenever collision of eigenvalue processes happens. Hence, the non-collision property among the eigenvalue processes is required to establish this theory. The key is to obtain the boundedness of the following quantity
\begin{align} \label{eq-ln-eigenvalue difference}
    \bE \left[ - \sum_{i\neq j} \ln \left| \lambda_i(t)-\lambda_j(t) \right| \right]
\end{align}
using It\^{o} calculus and martingate theory. This argument is sometimes known as McKean's argument \cite{McKean1969}. We refer to \cites{anderson2010, Yau2017} for more details. We remark that the solution to SDE \eqref{eq-Dyson BM} can be viewed as a particle system, and is known as Coulomb gas model. See \cites{Mehta2004, Forrester2010}.

In the area of statistics, the singular value processes of matrix-valued Gaussian processes can be used to study the stability in principal component analysis. Let $d_1 \le d_2$ be two integers and let $B(t)$ be a $d_1 \times d_2$ matrix whose entries are independent real standard Brownian motions. Consider
\begin{align} \label{eq-Wishart BM}
    W(t) = W(0) + B(t).
\end{align}
and denote by $\{s_i(t):1 \le i \le d_1\}$ the set of ordered singular value processes of $W(t)$. If we write $\lambda_i(t) = s_i^2(t)$, then the set $\{\lambda_i(t):1 \le i \le d_1\}$ is the ordered eigenvalue processes of $W(t) (W(t))^*$, which is known as Wishart process. It is shown in \cite{Bru1989} that the set of eigenvalue processes $\{\lambda_i(t):1 \le i \le d_1\}$ solves the following SDE:
\begin{align} \label{eq-eigen-SDE}
    d\lambda_i(t) = 2 \sqrt{\lambda_i(t)} dW_i(t) + n dt + \sum_{j:j\neq i} \dfrac{\lambda_i(t)+\lambda_j(t)}{\lambda_i(t)-\lambda_j(t)} dt, 1 \le i \le d, t \ge 0,
\end{align}
where the family $\{W_i(t):1 \le i \le d\}$ is a collection of i.i.d. real standard Brownian motions. Similar to the eigenavlue processes of symmetric matrix Brownian motion case, the SDE \eqref{eq-eigen-SDE} is singular when eigenvalue processes collide. The non-collision of eigenvalue processes was established in \cite{Bru1989} via the McKean's argument. The solution of the SDE \eqref{eq-eigen-SDE} is sometimes called the squared Bessel particle system. We refer the readers to \cite{Graczyk2019} for more studies on the squared Bessel particle system.

Since last decade, there are a lot of works on the eigenvalue processes of self-adjoint matrix-valued Gaussian processes and singular values of matrix-valued Gaussian processes. Let $\{B_{ij}^H(t): 1 \le i \le j \le d\}$ be a family of standard fractional Brownian motions with Hurst parameter $H \in (1/2,1)$. The $d \times d$ symmetric matrix $X^H(t)$, whose entries are given by
\begin{align} \label{eq-matrix fBM}
    X^H_{i,j}(t) =
    \begin{cases}
        B_{ij}^H(t) + X_{ij}(0), & 1 \le i<j \le d, \\
        \sqrt{2} B_{ii}^H(t) + X_{ii}(0), & 1 \le i=j \le d,
    \end{cases}
\end{align}
is called the symmetric matrix fractional Brownian motion. It was obtained in \cite{Nualart2014} that the set of eigenvalue processes $\{\lambda_i^H(t):1 \le i \le d\}$ satisfies the following SDE:
\begin{align} \label{eq-eigenvalue SDE-fBM}
    d\lambda_i^H(t) = dY^H_i(t) + 2H \sum_{j:j\neq i} \dfrac{t^{2H-1}}{\lambda_i^H(t) - \lambda_j^H(t)} dt, \quad 1 \le i \le d, t \ge 0,
\end{align}
where for each $i$, $Y^H_i(t)$ is the sum of Skorohod integrals with respect to all $B_{kh}^H(t)$. The eigenvalue SDE \eqref{eq-eigenvalue SDE-fBM} also has singularity when eigenvalue processes collide. By considering the quantity \eqref{eq-ln-eigenvalue difference} together with fractional calculus, \cite{Nualart2014} established the non-collision of the eigenvalue processes.

Let $B^H(t)$ be a $d_1 \times d_2$ matrix whose entries are i.i.d. real standard fractional Brownian motion with Hurst parameter $H \in (1/2,1)$, and let $W^H(t) = W(0) + B^H(t)$. Then the matrix process $W^H(t) (W^H(t))^*$ is known as the fractional Wishart process. In \cite{Pardo2017}, the set $\{\lambda_i^H(t):1 \le i \le d\}$  of its eigenvalue processes was shown to solve the following singular SDE:
\begin{align} \label{eq-singular value SDE-fBM}
    d\lambda_i^H(t) = dY^H_i(t) + 2Hd_2 dt + 2H \sum_{j:j\neq i} \dfrac{\lambda_i^H(t) + \lambda_j^H(t)}{\lambda_i^H(t) - \lambda_j^H(t)} t^{2H-1} dt, \quad 1 \le i \le d, t \ge 0,
\end{align}
where for each $i$, $Y^H_i(t)$ is the sum of Skorohod integrals. The non-collision of eigenvalue processes of fractional Wishart process was established in \cite{Pardo2017} by a similar startegy as \cite{Nualart2014}.

Very recently, the eigenvalue processes of symmetric matrix Brownian sheet were studied in \cite{Yuan2023}. Let $\{B_{ij}(s,t): 1 \le i \le j \le d\}$ be a family of standard Brownian sheets and let $X(s,t)$ be the $d \times d$ symmetric matrix, whose entries are given by
\begin{align} \label{eq-matrix Brownian sheet}
    X_{i,j}(s,t) =
    \begin{cases}
        B_{ij}(s,t) + X_{ij}(0), & 1 \le i<j \le d, \\
        \sqrt{2} B_{ii}(s,t) + X_{ii}(0,0), & 1 \le i=j \le d.
    \end{cases}
\end{align}
Then $X(s,t)$ is known as the symmetric matrix Brownian sheet. With the help of stochastic calculus on the plane, the stochastic partially differential equation (SPDE) satisfied by the eigenvalue processes of $X(s,t)$ was found in \cite{Yuan2023}. The SPDE is again singular when collision of eigenvalue processes happens. Unfortunately, as a direct consequence of \cite{Jaramillo2018}, the eigenvalue processes of the symmetric matrix Brownian sheet $X(s,t)$ collide with positive probability, which means that the SPDE is only able to describe the eigenvalue processes until the first collision time. 

From the models mentioned above, it is a common phenomenon that the stochastic equations for the eigenvalue processes (singular value processes, resp.) of self-adjoint matrix Gaussian processes (matrix Gaussian processes, resp.) are singular when eigenvalue processes (singular value processes, resp.) collide. So the equations are only able to describe the eigenvalues processes or singular value processes until the first collision time. Thus, it is important to know whether the eigenvalue processes or singular value processes collide or not.

In addition to the McKean's argument, a new method for eigenvalue processes appears recently in \cite{Jaramillo2018}, which is an application of the hitting probability of Gaussian random field.

Let $N \in \bN$ be fixed and consider a real-valued centered continuous Gaussian random field $\xi = \{\xi(t): t \in \bR_+^N \}$ defined on a probability space $(\Omega, \cF, \bP)$ with covariance given by
\begin{align*}
	\bE \left[ \xi(s) \xi(t) \right] = C(s,t),
\end{align*}
for some non-negative definite function $C: \bR_+^N \times \bR_+^N \rightarrow \bR$.
For $a = (a_1, \ldots, a_N), \, b = (b_1, \ldots, b_N) \in \bR_+^N$ satisfying $a_i \le b_i$ for $1 \le i \le N$, let $I = [a,b] $ be the compact interval defined by 
\begin{align} \label{def-interval}
	I = [a,b] = \prod_{j=1}^N[a_j, b_j] \subseteq \bR_+^N.
\end{align}
Throughout this paper, we denote by $\iota = \sqrt{-1}$  the imaginary unit.
For $\beta \in \{1,2\}$, and $d \in \bN$ with $d \ge 2$ fixed, we consider the following $d\times d$ matrix-valued process $X^{\beta} = \{X_{i,j}^{\beta}(t); t \in \bR_+^N, 1 \le i,j \le d\}$ whose entries are given by
\begin{align} \label{def-entries}
	X_{i,j}^{\beta}(t) =
	\begin{cases}
	\xi_{i,j}(t) + \iota \1_{[\beta = 2]} \eta_{i,j}(t), & i < j; \\
	\sqrt{2} \xi_{i,i}(t), & i=j; \\
	\xi_{j,i}(t) - \iota \1_{[\beta = 2]} \eta_{j,i}(t), & i > j.
	\end{cases}
\end{align}
Clearly, for every $t \in \bR_+^N$, $X^{\beta}(t)$ is a real symmetric Gaussian matrix for $\beta=1$ and a complex Hermitian Gaussian matrix for $\beta=2$. Moreover, the matrix Gaussian field $X^\beta$ extends the matrix Brownian motion \eqref{eq-matrix BM}, matrix fractional Brownian motion \eqref{eq-matrix fBM} and matrix Brownian sheet \eqref{eq-matrix Brownian sheet}. In particular, $X^1(t)/\sqrt{C(t,t)}$ belongs to the Gaussian orthogonal ensemble (GOE) and $X^2(t)/\sqrt{2C(t,t)}$ belongs to the Gaussian unitary ensemble (GUE), respectively.
Let $A^1$ be a real symmetric deterministic matrix and $A^2$ be a complex Hermitian deterministic matrix.
Let $\{\lambda_1^{\beta}(t), \cdots, \lambda_d^{\beta}(t)\}$ is the set of eigenvalue processes of 
\begin{align}\label{e:Y}
      Y^{\beta}(t) = A^{\beta} + T_\beta X^{\beta}(t) T_\beta^*
\end{align}
for $\beta=1,2$, where $T_\beta$ is a $d \times d$ real ($\beta=1$) or complex ($\beta=2$) matrix.

To our best of knowledge, the model \eqref{e:Y} is studied in the literature only when $T_\beta=I_d$ and $r=1$. Under certain conditions on the Gaussian random field $\xi$, \cite{Jaramillo2018} studied the probability of collision of at least two eigenvalue processes of $Y^\beta$ using hitting probability of Gaussian random field, and provided a sharp condition under which the collision probability is zero. The results were extended to the collision of at least $k$ eigenvalue processes in \cite{Xiao2021}. Moreover, \cite{Xiao2021} characterized the size of the set
\begin{equation*}
    {\cal C}_k^\beta  = \{ t \in I: \, \lambda_{i_1}^{\beta}(t) = \cdots = \lambda_{i_k}^{\beta}(t) \mathrm{\ for \ some \ } 
    \, 1 \le i_1 < \cdots < i_k \le d\}
\end{equation*}
of $k$-collision times by computing its Hausdorff dimension. However, the two papers does not provide any useful information on collision probability for the critical dimension. The critical dimension case of collision problem was solved in \cite{Lee2023} for $T_\beta$ being identity matrix.

For $n \in \bN$, we denote $[n]$ as the set of integers in $[1,n]$. Let $\{a_j\}_{j \in J}$ be a collection of numbers with index set $J$, and $J'$ is a subset of $J$, then we denote $a_{J'} = \{a_j: j \in J'\}$. For any set $J'$, we write $|J'|$ for the number of distinct elements in $J'$. Then $|a_{J'}|=1$ means that $a_j$ are the same for all $j \in J'$. In other words, all elements in $a_{J'}$ collide. The first goal of the present paper is to investigate the following probability of multiple collisions of eigenvalue processes of 
\begin{align*}
    \bP \left( \exists t \in I, \ \mathrm{s.t. \ } \left| \lambda_{J_j}^{\beta}(t) \right| = 1 \ \mathrm{for \ some \ disjoint \ } J_j \subset [d] \mathrm{\ with \ } |J_j| =  l_j, \forall 1 \le j \le r \right),
\end{align*}
where $2\le l_1,\ldots,l_r\le d$ satisfying $\sum_{j=1}^r l_j \le d$. For general matrix $T_\beta$, we provide the sufficient and necessary conditions such that the probability is zero in Theorem \ref{Thm-hitting prob-real} for the real case ($\beta=1$) and Theorem \ref{Thm-hitting prob-complex} for the complex case ($\beta=2$). When the collision probability is positive, we determine the Hausdorff dimension of the set of collision times
\begin{align} \label{e:h-dim}
    {\cal C}_{l_1,\ldots,l_r}^\beta  = \left\{ t \in I: \left| \lambda_{J_j}^{\beta}(t) \right| = 1 \ \mathrm{for \ some \ disjoint \ } J_j \subset [d] \mathrm{\ with \ } |J_j| =  l_j, \forall 1 \le j \le r \right\}.
\end{align}

The two theorems (Theorem \ref{Thm-hitting prob-real} and Theorem \ref{Thm-hitting prob-complex}) extend previous results by considering a more general multiple collisions problem of eigenvalue processes for the more comprehensive model \eqref{e:Y}. Our results are reduced to \cites{Jaramillo2018, Xiao2021, Lee2023} when $T_\beta$ is identity and $r=1$. The major difficulty is to handle the eigenvectors associated with multiple eigenvalues when computing the dimension of the eigenvector matrix in the spectral decomposition. The argument of \cites{Jaramillo2018, Xiao2021, Lee2023} is based on the fact that, if there is only one multiple eigenvalue, then the eigenspace associated with the multiple eigenvalue is the orthogonal complement of the space spanned by all eigenvectors of single eigenvalues. However, the argument no longer works if $r > 1$, where there are different multiple eigenvalues. See Remark \ref{Rmk-dim-S} for more details. In this paper, we design Lie groups such that the quotient manifold of the Lie groups we designed is exactly the set of eigenvector matrices. This is the main novelty in methodology of our paper.

For $d_1 \le d_2$, consider the following $d_1 \times d_2$ matrix-valued process $W^{\beta} = \{ W_{i,j}^{\beta}(t); t \in \bR_+^N, 1 \le i \le d_1, 1 \le j \le d_2 \}$ with entries $W_{i,j}^{\beta}(t) = \xi_{i,j}(t) + \iota \1_{[\beta = 2]} \eta_{i,j}(t)$.
Obviously, $W^\beta(t)W^\beta(t)^* /(\beta C(t,t))$, and $W^\beta(t)^*W^\beta(t) /(\beta C(t,t))$ are Wishart matrices for $\beta = 1,2$.
Let $B^1$ be a real deterministic $d_1 \times d_2$ matrix and $B^2$ be a complex deterministic $d_1 \times d_2$ matrix. We consider the following matrix
\begin{align} \label{def-Z}
	Z^{\beta}(t) = B^{\beta} + T_\beta W^{\beta}(t) \tilde T_\beta,
\end{align}
where $T_\beta$ and $\tilde T_\beta$ are $d_1 \times d_1$ and $d_2 \times d_2$ real ($\beta=1$) or complex ($\beta=2$) matrices respectively.
Let $\{ s_1^{\beta}(t), \ldots, s_{d_1}^{\beta}(t) \}$ be the set of non-trivial singular value processes of $Z^{\beta}(t)$. The second goal of the present paper is to examine when the probability
\begin{align*}
    \bP \left( \exists t \in I, \ \mathrm{s.t. \ } \left| s_{J_j}^{\beta}(t) \right| = 1 \ \mathrm{for \ some \ disjoint \ } J_j \in [d_1] \mathrm{\ with \ } |J_j| =  l_j, \forall 1 \le j \le r \right)
\end{align*}
is zero. We provide the sufficient and necessary conditions for the probability being zero in Theorem \ref{Thm-hitting prob-real-Wishart} for the real case and in Theorem \ref{Thm-hitting prob-compex-Wishart} for the complex case. We also compute the Hausdorff dimension of the set of collision times
\begin{align} \label{e:h-dim'}
    \widetilde {\cal C}_{l_1,\ldots,l_r}^\beta = \left\{ t \in I: \left| s_{J_j}^{\beta}(t) \right| = 1 \ \mathrm{for \ some \ disjoint \ } J_j \subset [d_1] \mathrm{\ with \ } |J_j| =  l_j, \forall 1 \le j \le r \right\}.
\end{align}
We remark that the set of eigenvalue processes of $Z^{\beta}(t) Z^{\beta}(t)^*$ is $\{ s_1^{\beta}(t)^2, \ldots, s_{d_1}^{\beta}(t)^2 \}$, while the set of eigenvalue processes of $Z^{\beta}(t)^* Z^{\beta}(t)$ is $\{ s_1^{\beta}(t)^2, \ldots, s_{d_1}^{\beta}(t)^2, 0, \ldots, 0\}$. Thus, the knowledge of collision of singular value processes of $Z^\beta(t)$ is equivalent to that of eigenvalue processes of $Z^{\beta}(t) Z^{\beta}(t)^*$ or $Z^{\beta}(t)^* Z^{\beta}(t)$. Thus, our results extend the non-collision property of eigenvalue processes of Wishart process and fractional Wishart process.

To our best knowledge, the two theorems (Theorem \ref{Thm-hitting prob-real-Wishart} and Theorem \ref{Thm-hitting prob-compex-Wishart}) are the first results on the collision probability of singular value processes of matrix Gaussian random field with high dimensional time parameter. To establish the two theorems, we shall consider the singular value decomposition. The main difficulty arises from handling the singular vectors associated with multiple singular values, especially for the trivial singular value 0, noting that there are $d_2-d_1$ right singular vectors that associated with the trivial singular value 0. To address the difficulty of multiple singular values, we identify the matrices of left singular vectors and right singular vectors as the quotient manifold of well-designed Lie groups.

The rest of the paper is organized as follows. We present the assumptions on the Gaussian random field $\xi$ in Subsection \ref{sec:assumptions}, introduce our main theorems in Subsection \ref{sec:results} and fix some notations in Subsection \ref{sec:notations}. We present some preliminaries that will be used in the proofs of main theorems in Section \ref{sec:preliminaries}. Then we develop the proofs of Theorem \ref{Thm-hitting prob-real}, Theorem \ref{Thm-hitting prob-complex}, Theorem \ref{Thm-hitting prob-real-Wishart} and Theorem \ref{Thm-hitting prob-compex-Wishart} in Subsection \ref{sec:real-case}, Subsection \ref{sec:complex-case}, Subsection \ref{sec:real case-singular value} and Subsection \ref{sec:complex case-singular value} respectively.

\subsection{Assumptions} \label{sec:assumptions}

We first recall the following definition of an $\bR^d$-valued Gaussian noise on $\bR_+$:
\begin{definition} \label{Def-noise}
	Let $\nu$ be a Borel measure on $\bR_+$, and let $A \mapsto W(A)$ be a set function defined on $\cB (\bR_+)$ with values in $L^2(\Omega, \cF, \bP; \bR^d)$ such that for each $A$, $W(A)$ is a centered normal random vector with values in $\bR^d$ and covariance matrix $\nu(A) I_d$. Assume that $W(A \cup B) = W(A) + W(B)$ a.s., and $W(A)$ and $W(B)$ are independent whenever $A \cap B = \emptyset$. Then the set function $A \mapsto W(A)$ is called an $\bR^d$-valued Gaussian noise with control measure $\nu$.
\end{definition} 

Recall the compact interval $I$ given by \eqref{def-interval}, we write $I^{(\epsilon)} = \prod_{j=1}^N (a_j-\epsilon,b_j+\epsilon)$ for any $\epsilon>0$.
Then we can impose the assumptions on the Gaussian random field $\xi = \{\xi(t): t \in \bR_+^N \}$. Let $0<H_1\le \ldots \le H_N<1$ be a non-decreasing sequence in $(0,1)$.
We assume that the following conditions hold:
\begin{enumerate}
    \item[{\bf (A1)}] There is a Gaussian random field $\{W(A,t): A \in \cB(\bR_+), t \in \bR^N\}$ and $\epsilon_0>0$ satisfying the following two conditions:
    \begin{enumerate}
        \item[(c1)] For all $t \in I^{(\epsilon_0)}$, $A \mapsto W(A,t)$ is an $\bR$-valued Gaussian noise with a control measure $\nu_t$ such that $W(\bR_+,t) = \xi(t)$ and when $A \cap B = \emptyset$, $W(A,\cdot)$ and $W(B,\cdot)$ are independent.

        \item[(c2)] There exist constants $a_0 \ge 0$, $c_0 > 0$, such that for all $a_0 \le c < d \le +\infty$ and all $s:=(s_1,\dots, s_N),t:=(t_1, \dots, t_N) \in I^{(\epsilon_0)}$,
		\begin{align*}
			\big\| W([c,d),s) - X(s) - W([c,d),t) + X(t) \big\|_{L^2}
			\le c_0 \left[ \sum_{j=1}^N c^{H_j^{-1}-1} |s_j - t_j| + d^{-1} \right],
		\end{align*}
		and
		\begin{align*}
			\big\| W([0,a_0),s) - W([0,a_0),t) \big\|_{L^2} \le c_0 \sum_{j=1}^N |s_j - t_j|.
		\end{align*}
	\end{enumerate}

    \item [{\bf (A2)}] There exists a constant $c_1 > 0$, such that $\|\xi(t)\|_{L^2} \ge c_1$ for all $t \in I^{(\epsilon_0)}$.

    \item [{\bf (A3)}] There exists a constant ${\rho_0} > 0$ with the following property. For $t \in I$, there exist
	$t' = t'(t) \in I^{(\epsilon_0)}$, $\delta_j = \delta_j(t) \in (H_j,1]$ for $1 \le i \le N$, and $c_2 = c_2(t) > 0$, such that
	\begin{align*}
		\Big| \bE \left[ \xi(t') \big( \xi(s) - \xi(\bar{s}) \big) \right] \Big|
		\le c_2 \sum_{j=1}^N |s_j - \bar{s}_j|^{\delta_j},
	\end{align*}
	for all $1 \le i \le N$ and all $s, \bar{s} \in I^{(\epsilon_0)}$ with $\max\{\sum_{j=1}^N |s_j - t_j|^{H_j}, \sum_{j=1}^N |\bar{s}_j - t_j|^{H_j}\} \le 2 {\rho_0}$.
 
	\item[{\bf (A4)}] There exists positive and finite constant $c_3$, such that
	\begin{align*}
		\left\| \xi(s) - \xi(t) \right\|_{L^2} \ge c_3 \sum_{j=1}^N |s_j - t_j|^{H_j}
	\end{align*}
	for all $s = (s_1, \ldots, s_N), t = (t_1, \ldots, t_N) \in I$.
	
	\item[{\bf (A5)}] There exists a positive constant $c_4$ such that for all $s,t\in I$,
	\begin{align*}
		\mathrm{Var} \left[ \xi(t) | \xi(s) \right] \ge c_4 \sum_{j=1}^N |s_j - t_j|^{2H_j},
	\end{align*}
	where $\mathrm{Var} \left[ \xi(t) | \xi(s) \right]$ denotes the conditional variance of $\xi(t)$ given $\xi(s)$.
\end{enumerate}

\begin{remark}
By \cite{Xiao2009}*{page 256}, the conditional variance of centered Gaussian vector has the following formula:
\begin{align*}
    \mathrm{Var} \left( U \big| V \right)
    = \dfrac{\left( \rho^2_{U,V} - \left( \sigma_U-\sigma_V \right)^2 \right) \left( \left( \sigma_U+\sigma_V \right)^2 - \rho^2_{U,V} \right)}{4\sigma_V^2},
\end{align*}
where $\rho^2_{U,V} = \bE [(U-V)^2]$, $\sigma_U^2 = \bE[U^2]$ and $\sigma_V^2 = \bE[V^2]$.
\end{remark}

\begin{remark} \label{Rmk-condition A1}
By \cite{dalang2017polarity}*{Proposition 2.2}, Assumption {\bf (A1)} implies that there exists a constant $C>0$, such that
\begin{align*}
   \left\| \xi(s) - \xi(t) \right\|_{L^2} \le C \sum_{j=1}^N |s_j - t_j|^{H_j}.
\end{align*}
for all $s,t \in I^{(\epsilon_0)}$. See also \cite{Lee2023}*{Lemma 2.2}.
\end{remark}

\begin{remark} \label{Rmk-conditions}
    Assumptions {\bf (A1)} - {\bf (A3)} are slightly stronger than the assumptions in \cite{Lee2023}, and will be used for the critical dimension of hitting probability. By Remark \ref{Rmk-condition A1}, Assumption {\bf (A1)}, {\bf (A2)}, {\bf (A4)} and {\bf (A5)} are slightly stronger than the assumptions in \cite{Jaramillo2018} and \cite{Xiao2021}, and are required for the non-critical dimension of hitting probability.
\end{remark}

\begin{remark} \label{Rmk-conditions-example}
    The Assumptions {\bf (A1)} - {\bf (A5)} include a large class of Gaussian random fields such as multi-parameter fractional Brownian motions, fractional Brownian sheets, solutions to stochastic heat equations. We refer the interested readers to \cites{xiao2009sample, Xiao2021, Lee2023} for more details.
\end{remark}

\subsection{Main results} \label{sec:results}

Let $\xi = \{\xi(t): t \in \bR_+^N \}$ be the Gaussian random field satisfying Assumptions {\bf (A1)} - {\bf (A5)}, and $\{\xi_{i,j}, \eta_{i,j}:i,j \in \bN\}$ be a family of independent copies of $\xi$. We assume that the matrices $T_\beta, \tilde T_\beta$ are invertible.

The following is the first results of this paper, which extends the previous results (\cites{Jaramillo2018, Xiao2021, Lee2023}) to existence of different multiple eigenvalue processes. For the real-valued case, we have the following theorem.

\begin{theorem} \label{Thm-hitting prob-real}
Let $Y^{\beta}$ ($\beta = 1$) be the matrix-valued process defined by \eqref{e:Y} with eigenvalues $\{\lambda_1^\beta(t), \dots, \lambda_d^\beta(t)\}$.
Then for any $l_1, \ldots, l_r \in \{2, \ldots, d\}$ satisfying $\sum_{j=1}^r l_j \le d$, the following statements hold: 
\begin{enumerate}
	\item[(i)] If $\sum_{j=1}^N\frac1{H_j} \le \sum_{j=1}^r (l_j-1)(l_j+2)/2$, then
	\begin{align*}
	   \bP \left( \exists t \in I, \ \mathrm{s.t. \ } \left| \lambda_{J_j}^{\beta}(t) \right| = 1 \ \mathrm{for \ some \ disjoint \ } J_j \in [d] \mathrm{\ with \ } |J_j| =  l_j, \forall 1 \le j \le r \right) = 0.
	\end{align*}
	\item[(ii)] If $\sum_{j=1}^N\frac1{H_j} > \sum_{j=1}^r (l_j-1)(l_j+2)/2$, then
	\begin{align*}
		\bP \left( \exists t \in I, \ \mathrm{s.t. \ } \left| \lambda_{J_j}^{\beta}(t) \right| = 1 \ \mathrm{for \ some \ disjoint \ } J_j \in [d] \mathrm{\ with \ } |J_j| =  l_j, \forall 1 \le j \le r \right) > 0.
	\end{align*}
	\item[(iii)] If $\sum_{j=1}^N\frac1{H_j} > \sum_{j=1}^r (l_j-1)(l_j+2)/2$, then with positive probability, the Hausdorff dimension of the set ${\cal C}_{l_1,\ldots,l_r}^\beta$ of collision times given in \eqref{e:h-dim} is 
	\begin{equation}\label{Eq:dimC}
	   \begin{split}
	       \dimh {\cal C}_{l_1,\ldots,l_r}^\beta
	       &= \min_{1 \le \ell \le N} \bigg\{ \sum_{j=1}^{\ell } \frac{H_\ell}{H_j} + N-\ell - \dfrac{H_\ell}{2} \sum_{j=1}^r (l_j-1)(l_j+2) \bigg\} \\
	       & = \sum_{j=1}^{\ell_0 } \frac{H_{\ell_0}}{H_j} + N-\ell_0 - \dfrac{H_{\ell_0}}{2} \sum_{j=1}^r (l_j-1)(l_j+2),
	   \end{split}
	\end{equation}
	where $\ell_0$ is the smallest $\ell$ such that $\sum_{j=1}^\ell \frac1{H_j} > \sum_{j=1}^r (l_j-1)(l_j+2)/2$, i.e., $\sum_{j=1}^{\ell_0 - 1} \frac{1}{H_j} \le \sum_{j=1}^r (l_j-1)(l_j+2)/2 < \sum_{j=1}^{\ell_0 } \frac{1}{H_j}$ with the convention $\sum_{j=1}^{0} \frac{1}{H_j}:= 0$.
\end{enumerate}
\end{theorem}

To illustrate this theorem, we consider the special case of a symmetric matrix-valued process $X^{\beta}$ associated with fractional Brownian motion $B^H$, noting that $B^H$ satisfies Assumptions {\bf (A1) - (A5)}.
\begin{corollary}
Let $\beta = 1$ and $Y^{\beta}$ be a  matrix-valued process given in \eqref{e:Y} associated with fractional Brownian motion $B^H$. Then for any $l_1, \ldots, l_r \in \{2, \ldots, d\}$ satisfying $\sum_{j=1}^r l_j \le d$, the following hold:

\begin{enumerate}
	\item[(i)] If $N \le H \sum_{j=1}^r (l_j-1)(l_j+2)/2$, then
	\begin{align*}
		\bP \left( \exists t \in I, \ \mathrm{s.t. \ } \left| \lambda_{J_j}^{\beta}(t) \right| = 1 \ \mathrm{for \ some \ disjoint \ } J_j \in [d] \mathrm{\ with \ } |J_j| =  l_j, \forall 1 \le j \le r \right) = 0.
	\end{align*}
	\item[(ii)] If $N > H \sum_{j=1}^r (l_j-1)(l_j+2)/2$, then
	\begin{align*}
		\bP \left( \exists t \in I, \ \mathrm{s.t. \ } \left| \lambda_{J_j}^{\beta}(t) \right| = 1 \ \mathrm{for \ some \ disjoint \ } J_j \in [d] \mathrm{\ with \ } |J_j| =  l_j, \forall 1 \le j \le r \right) > 0.
	\end{align*}
	\item[(iii)] If $N > H \sum_{j=1}^r (l_j-1)(l_j+2)/2$, then with positive probability,
	\begin{equation*}
	   \dimh {\cal C}_{l_1,\ldots,l_r}^\beta =  N - \dfrac{H}{2} \sum_{j=1}^r (l_j-1)(l_j+2).
	\end{equation*}
\end{enumerate}
\end{corollary}

The following theorem is the complex analog of Theorem \ref{Thm-hitting prob-real}.

\begin{theorem} \label{Thm-hitting prob-complex}
Let $Y^{\beta}$ ($\beta = 2$) be the matrix-valued process defined by \eqref{e:Y} with eigenvalues $\{\lambda_1^\beta(t), \dots, \lambda_d^\beta(t)\}$. Then for any $l_1, \ldots, l_r \in \{2, \ldots, d\}$ satisfying $\sum_{j=1}^r l_j \le d$, the following statements hold: 

\begin{enumerate}
	\item[(i)] If $\sum_{j=1}^N\frac1{H_j} \le \sum_{j=1}^r (l_j+1)(l_j-1)$, then
	\begin{align*}
		\bP \left( \exists t \in I, \ \mathrm{s.t. \ } \left| \lambda_{J_j}^{\beta}(t) \right| = 1 \ \mathrm{for \ some \ disjoint \ } J_j \in [d] \mathrm{\ with \ } |J_j| =  l_j, \forall 1 \le j \le r \right) = 0.
	\end{align*}
	\item[(ii)] If $\sum_{j=1}^N\frac1{H_j} > \sum_{j=1}^r (l_j+1)(l_j-1)$, then
	\begin{align*}
		\bP \left( \exists t \in I, \ \mathrm{s.t. \ } \left| \lambda_{J_j}^{\beta}(t) \right| = 1 \ \mathrm{for \ some \ disjoint \ } J_j \in [d] \mathrm{\ with \ } |J_j| =  l_j, \forall 1 \le j \le r \right) > 0.
	\end{align*}
	\item[(iii)] If $\sum_{j=1}^N\frac1{H_j} > \sum_{j=1}^r (l_j+1)(l_j-1)$, then with positive probability, the Hausdorff dimension of the set ${\cal C}_k^\beta$ of collision times given in \eqref{e:h-dim} is 
	\begin{equation*}
	   \begin{split}
            \dimh {\cal C}_{l_1,\ldots,l_r}^\beta
            &= \min_{1 \le \ell \le N}\bigg\{\sum_{j=1}^{\ell } \frac{H_\ell}{H_j} + N-\ell - H_\ell \sum_{j=1}^r (l_j+1)(l_j-1) \bigg\}\\
            & = \sum_{j=1}^{\ell_0 } \frac{H_{\ell_0}}{H_j} + N-\ell_0 - H_{\ell_0} \sum_{j=1}^r (l_j+1)(l_j-1),
        \end{split}
    \end{equation*}
    where $\ell_0$ is the smallest $\ell$ such that $\sum_{j=1}^{\ell} \frac1{H_j} > \sum_{j=1}^r (l_j+1)(l_j-1)$. 
\end{enumerate}
\end{theorem}

Similarly, we have the following corollary where the Gaussian random field $\xi$ is reduced to the fractional Brownian motion $B^H$.

\begin{corollary}
Let $\beta = 2$ and $Y^{\beta}$ be a  matrix-valued process given in \eqref{e:Y} associated with fractional Brownian motion $B^H$. Then for any $l_1, \ldots, l_r \in \{2, \ldots, d\}$ satisfying $\sum_{j=1}^r l_j \le d$, the following hold:

\begin{enumerate}
	\item[(i)] If $N \le H\sum_{j=1}^r (l_j+1)(l_j-1)$, then
	\begin{align*}
		\bP \left( \exists t \in I, \ \mathrm{s.t. \ } \left| \lambda_{J_j}^{\beta}(t) \right| = 1 \ \mathrm{for \ some \ disjoint \ } J_j \in [d] \mathrm{\ with \ } |J_j| =  l_j, \forall 1 \le j \le r \right) = 0.
	\end{align*}
	\item[(ii)] If $N > H\sum_{j=1}^r (l_j+1)(l_j-1)$, then
	\begin{align*}
		\bP \left( \exists t \in I, \ \mathrm{s.t. \ } \left| \lambda_{J_j}^{\beta}(t) \right| = 1 \ \mathrm{for \ some \ disjoint \ } J_j \in [d] \mathrm{\ with \ } |J_j| =  l_j, \forall 1 \le j \le r \right) > 0.
	\end{align*}
	\item[(iii)] If $N > H\sum_{j=1}^r (l_j+1)(l_j-1)$, then with positive probability,
	\begin{equation*}
	\dimh {\cal C}_{l_1,\ldots,l_r}^\beta =  N - H \sum_{j=1}^r (l_j+1)(l_j-1).
	\end{equation*}
\end{enumerate}
\end{corollary}

Next, we will turn to the collision of singular value processes of $Z^\beta(t)$. Recall that $Z^\beta(t)$ is a $d_1 \times d_2$ matrix given by \eqref{def-Z} with $d_1 \le d_2$. Let's start with the real case.

\begin{theorem} \label{Thm-hitting prob-real-Wishart}
Let $Z^{\beta}$ ($\beta = 1$) be the matrix-valued process defined by \eqref{def-Z} with non-trivial singular values $\left\{ s_1^\beta(t), \dots, s_{d_1}^\beta(t) \right\}$. Then for any $l_1, \ldots, l_r \in \{2, \ldots, d_1\}$ satisfying $\sum_{j=1}^r l_j \le d_1$, the following statements hold: 
\begin{enumerate}
	\item[(i)] If $\sum_{j=1}^N\frac1{H_j} \le \frac{1}{2} \sum_{j=1}^r (l_j-1)(l_j+2)$, then
	\begin{align*}
		\bP \left( \exists t \in I, \ \mathrm{s.t. \ } \left| s_{J_j}^{\beta}(t) \right| = 1 \ \mathrm{for \ some \ disjoint \ } J_j \in [d_1] \mathrm{\ with \ } |J_j| =  l_j, \forall 1 \le j \le r \right) = 0.
	\end{align*}
	\item[(ii)] If $\sum_{j=1}^N\frac1{H_j} > \frac{1}{2} \sum_{j=1}^r (l_j-1)(l_j+2)$, then
	\begin{align*}
		\bP \left( \exists t \in I, \ \mathrm{s.t. \ } \left| s_{J_j}^{\beta}(t) \right| = 1 \ \mathrm{for \ some \ disjoint \ } J_j \in [d_1] \mathrm{\ with \ } |J_j| =  l_j, \forall 1 \le j \le r \right) > 0.
	\end{align*}
	\item[(iii)] If~ $\sum_{j=1}^N\frac1{H_j} > \frac{1}{2} \sum_{j=1}^r (l_j-1)(l_j+2)$, then with positive probability, the Hausdorff dimension of the set $\widetilde {\cal C}_{l_1,\ldots,l_r}^\beta$ of collision times given in \eqref{e:h-dim} is 
	\begin{equation}\label{Eq:dimC-Wishart}
		\begin{split}
			\dimh \widetilde {\cal C}_{l_1,\ldots,l_r}^\beta
			&= \min_{1 \le \ell \le N} \bigg\{ \sum_{j=1}^{\ell } \frac{H_\ell}{H_j} + N-\ell - \frac{H_\ell}{2} \sum_{j=1}^r (l_j-1)(l_j+2) \bigg\} \\
			& =  \sum_{j=1}^{\ell_0 } \dfrac{H_{\ell_0}}{H_j} + N-\ell_0 -  \dfrac{H_{\ell_0}}{2} \sum_{j=1}^r (l_j-1)(l_j+2),
		\end{split}
	\end{equation}
	where $\ell_0$ is the smallest $\ell$ such that $\sum_{j=1}^\ell \frac1{H_j} > \frac{1}{2} \sum_{j=1}^r (l_j-1)(l_j+2)$.
\end{enumerate}
\end{theorem}

If the Gaussian random field $\xi$ is reduced to the fractional Brownian motion, we have the following corollary.

\begin{corollary}
Let $\beta = 1$ and $Z^{\beta}$ be a matrix-valued process given in \eqref{def-Z} associated with fractional Brownian motion $B^H$. Then for any $l_1, \ldots, l_r \in \{2, \ldots, d_1\}$ satisfying $\sum_{j=1}^r l_j \le d_1$, the following hold:
	
\begin{enumerate}
	\item[(i)] If $N < \frac{H}{2} \sum_{j=1}^r (l_j-1)(l_j+2)$, then
	\begin{align*}
		\bP \left( \exists t \in I, \ \mathrm{s.t. \ } \left| \lambda_{J_j}^{\beta}(t) \right| = 1 \ \mathrm{for \ some \ disjoint \ } J_j \in [d] \mathrm{\ with \ } |J_j| =  l_j, \forall 1 \le j \le r \right) = 0.
		\end{align*}
	\item[(ii)] If $N > \frac{H}{2} \sum_{j=1}^r (l_j-1)(l_j+2)$, then
	\begin{align*}
		\bP \left( \exists t \in I, \ \mathrm{s.t. \ } \left| \lambda_{J_j}^{\beta}(t) \right| = 1 \ \mathrm{for \ some \ disjoint \ } J_j \in [d] \mathrm{\ with \ } |J_j| =  l_j, \forall 1 \le j \le r \right) > 0.
	\end{align*}
	\item[(iii)] If $N > \frac{H}{2} \sum_{j=1}^r (l_j-1)(l_j+2)$, then with positive probability,
	\begin{equation*}
		\dimh \widetilde {\cal C}_{l_1,\ldots,l_r}^\beta =  N - \dfrac{H}{2} \sum_{j=1}^r (l_j-1)(l_j+2).
	\end{equation*}
\end{enumerate}
\end{corollary}

The next theorem is the complex version of Theorem \ref{Thm-hitting prob-real-Wishart}.

\begin{theorem} \label{Thm-hitting prob-compex-Wishart}
Let $Z^{\beta}$ ($\beta = 2$) be the matrix-valued process defined by \eqref{def-Z} with non-trivial singular values $\left\{ s_1^\beta(t), \dots, s_{d_1}^\beta(t) \right\}$. Then for any $l_1, \ldots, l_r \in \{2, \ldots, d_1\}$ satisfying $\sum_{j=1}^r l_j \le d_1$, the following statements hold: 
\begin{enumerate}
	\item[(i)] If $\sum_{j=1}^N\frac1{H_j} \le \sum_{j=1}^r (l_j-1)(l_j+1)$, then
	\begin{align*}
		\bP \left( \exists t \in I, \ \mathrm{s.t. \ } \left| s_{J_j}^{\beta}(t) \right| = 1 \ \mathrm{for \ some \ disjoint \ } J_j \in [d_1] \mathrm{\ with \ } |J_j| =  l_j, \forall 1 \le j \le r \right) = 0.
	\end{align*}
	\item[(ii)] If $\sum_{j=1}^N\frac1{H_j} > \sum_{j=1}^r (l_j-1)(l_j+1)$, then
	\begin{align*}
		\bP \left( \exists t \in I, \ \mathrm{s.t. \ } \left| s_{J_j}^{\beta}(t) \right| = 1 \ \mathrm{for \ some \ disjoint \ } J_j \in [d_1] \mathrm{\ with \ } |J_j| =  l_j, \forall 1 \le j \le r \right) > 0.
	\end{align*}
	\item[(iii)] If $\sum_{j=1}^N\frac1{H_j} > \sum_{j=1}^r (l_j-1)(l_j+1)$, then with positive probability, the Hausdorff dimension of the set $\widetilde {\cal C}_{l_1,\ldots,l_r}^\beta$ of collision times given in \eqref{e:h-dim} is 
	\begin{equation*}
		\begin{split}
			\dimh \widetilde {\cal C}_{l_1,\ldots,l_r}^\beta
			&= \min_{1 \le \ell \le N} \left\{ \sum_{j=1}^{\ell } \frac{H_\ell}{H_j} + N-\ell - H_\ell \sum_{j=1}^r (l_j-1)(l_j+1) \right\} \\
			& =  \sum_{j=1}^{\ell_0 } \frac{H_{\ell_0}}{H_j} + N-\ell_0 - H_{\ell_0} \sum_{j=1}^r (l_j-1)(l_j+1),
		\end{split}
	\end{equation*}
	where $\ell_0$ is the smallest $\ell$ such that $\sum_{j=1}^\ell \frac1{H_j} > \sum_{j=1}^r (l_j-1)(l_j+1)$.
\end{enumerate}
\end{theorem}

If the associated Gaussian random field $\xi$ is a fractional Brownian motion $B^H$, then Theorem \ref{Thm-hitting prob-compex-Wishart} reduced to the following corollary.

\begin{corollary}
Let $\beta = 2$ and $Z^{\beta}$ be a matrix-valued process given in \eqref{def-Z} associated with fractional Brownian motion $B^H$. Then for any $l_1, \ldots, l_r \in \{2, \ldots, d_1\}$ satisfying $\sum_{j=1}^r l_j \le d_1$, the following hold:
	
\begin{enumerate}
	\item[(i)] If $N \le H\sum_{j=1}^r (l_j-1)(l_j+1)$, then
	\begin{align*}
		\bP \left( \exists t \in I, \ \mathrm{s.t. \ } \left| \lambda_{J_j}^{\beta}(t) \right| = 1 \ \mathrm{for \ some \ disjoint \ } J_j \in [d] \mathrm{\ with \ } |J_j| =  l_j, \forall 1 \le j \le r \right) = 0.
	\end{align*}
	\item[(ii)] If $N > H\sum_{j=1}^r (l_j-1)(l_j+1)$, then
	\begin{align*}
		\bP \left( \exists t \in I, \ \mathrm{s.t. \ } \left| \lambda_{J_j}^{\beta}(t) \right| = 1 \ \mathrm{for \ some \ disjoint \ } J_j \in [d] \mathrm{\ with \ } |J_j| =  l_j, \forall 1 \le j \le r \right) > 0.
	\end{align*}
	\item[(iii)] If $N > H\sum_{j=1}^r (l_j-1)(l_j+1)$, then with positive probability,
	\begin{equation*}
		\dimh \widetilde {\cal C}_{l_1,\ldots,l_r}^\beta =  N - H\sum_{j=1}^r (l_j-1)(l_j+1).
	\end{equation*}
\end{enumerate}
\end{corollary}

\subsection{Notations} \label{sec:notations}

In this subsection, we introduce some notations on matrices that will be used in the proofs. 

For $d,d_1,d_2 \in \bN$, we denote by ${\mathbf D}(d_1 \times d_2)$ the set of diagonal real $d_1 \times d_2$ matrices (the set of matrices whose $(i,j)$ entry vanishes for $i \neq j$) and ${\mathbf D}(d) = {\mathbf D}(d \times d)$. For a matrix $A$, denote by $A^*$ the conjugate transpose of $A$. We also use the notation $A_{*,j}$ for the $j$th column of $A$. If $A$ is a square matrix, then we denote by $\spec(A)$ the spectrum of $A$, i.e. the set of eigenvalues of $A$. We write $\cE^A_{\lambda}$ for the eigenspace associated with $\lambda \in \spec(A)$.

We denote by $\bR^{d_1 \times d_2}$ the set of real $d_1 \times d_2$ matrix, and identify its element
\begin{align*}
	\tilde{x} = \left(
	\begin{matrix}
	\tilde{x}_{11} & \tilde{x}_{12} & \cdots & \tilde{x}_{1d_2} \\
	\tilde{x}_{21} & \tilde{x}_{22} & \cdots & \tilde{x}_{2d_2} \\
	\ldots & \ldots &\ldots & \ldots \\
	\tilde{x}_{d_11} & \tilde{x}_{d_12} & \cdots & \tilde{x}_{d_1d_2}
	\end{matrix}
	\right)
\end{align*}
as the row vector
\begin{align*}
    \left( \tilde{x}_{11}, \ldots, \tilde{x}_{1d_2}, \tilde{x}_{21}, \ldots, \tilde{x}_{2d_2}, \ldots, \tilde{x}_{d_11}, \ldots, \tilde{x}_{d_1d_2} \right)
\end{align*}
in $\bR^{d_1d_2}$. Similarly, we use $\bC^{d_1 \times d_2}$ for the set of complex $d_1 \times d_2$ matrix, and its elements can be identify as row vectors in $\bR^{2d_1d_2}$ by considering the real part and imaginary part.

We denote by $\mathbf S(d)$ and $\mathbf H(d)$ the set of real symmetric $d\times d$ matrices and the set of complex Hermitian $d\times d$ matrices, respectively. A symmetric matrix
\begin{align*}
	\tilde{x} = \left(
	\begin{matrix}
	\tilde{x}_{11} & \tilde{x}_{12} & \cdots & \tilde{x}_{1d} \\
	\tilde{x}_{12} & \tilde{x}_{22} & \cdots & \tilde{x}_{2d} \\
	\ldots & \ldots &\ldots & \ldots \\
	\tilde{x}_{1d} & \tilde{x}_{2d} & \cdots & \tilde{x}_{dd}
	\end{matrix}
	\right)
\end{align*}
can be viewed as a row vector
\begin{align*}
    \left( \tilde{x}_{11}, \ldots, \tilde{x}_{1d}, \tilde{x}_{22}, \ldots, \tilde{x}_{2d}, \ldots, \tilde{x}_{dd} \right)
\end{align*}
in $\bR^{d(d+1)/2}$. In a similar way, a Hermitian matrix
\begin{align*}
	\tilde{x} = \left(
	\begin{matrix}
	\tilde{x}_{11} & \tilde{x}_{12} & \cdots & \tilde{x}_{1d} \\
	\overline{\tilde{x}_{12}} & \tilde{x}_{22} & \cdots & \tilde{x}_{2d} \\
	\ldots & \ldots &\ldots & \ldots \\
	\overline{\tilde{x}_{1d}} & \overline{\tilde{x}_{2d}} & \cdots & \tilde{x}_{dd}
	\end{matrix}
	\right) 
\end{align*}
can be understood as a row vector 
\begin{align*}
    &\Big(\tilde{x}_{11}, \mathrm{Re}(\tilde{x}_{12}), \ldots, \mathrm{Re}(\tilde{x}_{1d}), \tilde{x}_{22}, \mathrm{Re}(\tilde{x}_{23}), \ldots, \mathrm{Re}(\tilde{x}_{2d}), \ldots, \tilde{x}_{dd}, \\
    &~~~~~ \qquad \qquad \qquad \quad \mathrm{Im}(\tilde{x}_{12}), \ldots, \mathrm{Im}(\tilde{x}_{1d}), \mathrm{Im}(\tilde{x}_{23}), 
    \ldots, \mathrm{Im}(\tilde{x}_{(d-1)d})\Big).
\end{align*}
in $\bR^{d^2}$.

By the identification above, we abuse the notation, namely, for a vector $x$ in $\bR^{d(d+1)/2}$, $\bR^{d^2}$, $\bR^{d_1d_2}$ or $\bR^{2d_1d_2}$, it also means the corresponding matrix in $\mathbf S(d)$, $\mathbf H(d)$, $\bR^{d_1\times d_2}$ or $\bC^{d_1\times d_2}$ respectively.

For a vector $x$ in $\bR^{d(d+1)/2}$ or $\bR^{d^2}$, let $E_i(x)$ be the $i$th smallest eigenvalue of $x$. Then $E_i(x)$ is a continuous function of $x$ for each $i\in\{1,\dots, d\}$, noting that $(E_1(x), \dots, E_d(x))$ are ordered roots of the characteristic polynomial of $x$. For a vector $x$ in $\bR^{d_1d_2}$ or $\bR^{2d_1d_2}$, let $S_i(x)$ be the $i$th smallest singular value of $x$. Then $S_i(x)$ is a continuous function of $x$ for each $i \in \{1,\dots, d_1\}$, noting that $(S_1(x)^2, \dots, S_{d_1}(x)^2)$ are ordered eigenvalues of $x x^*$.

For $r \in \bN$, $l_1, \ldots, l_r \in \{2, \ldots, d\}$ satisfying $\sum_{j=1}^r l_j \le d$, we define
\begin{align} \label{def-degenerate-real}
	{\mathbf S}(d;l_1, \ldots, l_r) = \{x \in {\mathbf S}(d): \ |E_{J_j}| = 1, \mathrm{\ for \ some \ disjoint \ } J_j \subset [d] \mathrm{\ with \ } |J_j| = l_j, 1 \le j \le r \}
\end{align}
and
\begin{align} \label{def-degenerate-complex}
	{\mathbf H}(d;l_1, \ldots, l_r) = \{x \in {\mathbf H}(d): \ |E_{J_j}| = 1, \mathrm{\ for \ some \ disjoint \ } J_j \subset [d] \mathrm{\ with \ } |J_j| = l_j, 1 \le j \le r \}.
\end{align}
For $d_1 \le d_2$, if $l_1, \ldots, l_r \in \{2, \ldots, d_1\}$ satisfying $\sum_{j=1}^r l_j \le d_1$, we define
\begin{align} \label{def-degenerate-Wishart-real}
	{\mathbf M}_{\bR}(d_1 \times d_2;l_1, \ldots, l_r) =& \{x \in \bR^{d_1 \times d_2}: \ |S_{J_j}| = 1, \nonumber \\
    & \qquad\qquad \mathrm{for \ some \ disjoint \ } J_j \subset [d_1] \mathrm{\ with \ } |J_j| = l_j, 1 \le j \le r \}
\end{align}
and
\begin{align} \label{def-degenerate-Wishart-complex}
	{\mathbf M}_{\bC}(d_1 \times d_2;l_1, \ldots, l_r) =& \{x \in \bC^{d_1 \times d_2}: \ |S_{J_j}| = 1, \nonumber \\
    & \qquad\qquad \mathrm{for \ some \ disjoint \ } J_j \subset [d_1] \mathrm{\ with \ } |J_j| = l_j, 1 \le j \le r \}.
\end{align}

For a vector space $\bR^d$ or $\bC^d$, let $\|\cdot\|$ be the Euclidean norm and $\langle\cdot,\cdot\rangle$ be the corresponding inner product. For a metric space $X$, we denote by $\mathfrak B_r(x)$ the open ball centered at $x\in X$ with radius $r$.
For $A\in \bC^{m\times n}$, we define the Frobenius norm,
\begin{equation}\label{eq:F-norm}
    \|A\|=\bigg(\sum_{i=1}^m\sum_{j=1}^n |A_{ij}|^2 \bigg)^{1/2}.
\end{equation}
Thus $\|A\|$ is just the Euclidean norm of $A$,  if we consider $A$ as a vector of size $mn$.

\section{Preliminaries} \label{sec:preliminaries}

\subsection{Hausdorff dimension and capacity}

Firstly, we briefly recall the definitions of Hausdorff measure, Hausdorff dimension, and Bessel-Riesz capacity as well as their relationship that are used in the present paper. We refer to \cites{Falconer2014,mattila1999} for more details.

Let $n \in \bN$ be fixed and $q > 0$ be a constant.  For any subset $A \subseteq \bR^n$,  the $q$-dimensional Hausdorff measure of $A$ is defined by
\begin{equation*}
    \mathcal{H}_{q} (A)
    = \lim_{\varepsilon \to 0} \inf \bigg\{ \sum_i (2 r_i)^q:\, A \subseteq \bigcup_{i=1}^{\infty} \mathfrak B_{r_i}(x_i), \  r_i < \varepsilon \bigg\}.
\end{equation*}
We use the convention that $\cH_q(A)=1$ when $q \le 0$.
It is known that $\mathcal{H}_{q} (\cdot)$ is a metric outer measure and every Borel set in $\bR^n$ is $\mathcal{H}_{q}$-measurable. The Hausdorff dimension $\dimh (A)$ defined by
\begin{align*}
    \dimh (A) = \inf \{q>0 \ | \ \cH_q(A) = 0\}
	= \sup \{q>0 \ | \ \cH_q(A) > 0\}.
\end{align*}

The Bessel-Riesz capacity of order $q$ of $A$ is defined by
\begin{align*}
	\cC_{q} (A)
	= \left( \inf_{\mu \in \mathcal{P}(A)}  \iint_{\bR^n \times \bR^n} f_q(\|x - y\|) \mu(dx) \mu(dy) \right)^{-1},
\end{align*}
where $\mathcal{P}(A)$ is the family of probability measures supported in $A$, and the function 
$f_q: \bR_+ \rightarrow \bR_+$ is given by
\begin{align*}
	f_q(r) =
	\begin{cases}
	r^{-q}, & q > 0; \\
	\ln \left( \dfrac{e}{r \wedge 1} \right), & q= 0; \\
	1, & q < 0.
	\end{cases}
\end{align*}
Hausdorff dimension and Bessel-Riesz capacity are related by the following Frostman theorem:
\begin{align*}
	\dimh (A) = \inf \{q>0 \ | \ \cC_q(A) = 0\}
	= \sup \{q>0 \ | \ \cC_q(A) > 0\}.
\end{align*}

The following lemma gives the estimate of Hausdorff measure of the image of a mapping.
\begin{lemma} \label{Lem-trans}
Let $\cL$ be a bijective mapping from $\bR^n$ to $\bR^n$. Assume that there exists a positive constant $c_\cL$, such that
\begin{align} \label{ineq-cL}
    \cL^{-1} \left( \cB_r(x) \right) \subseteq \cB_{c_\cL r} (\cL^{-1}(x)), \quad \forall x \in \bR^n, r \in \bR_+.
\end{align}
Then for any $A \subseteq \bR^n$, we have
\begin{align*}
    \cH_q \left( \cL(A) \right)
    \le c_\cL^{-q} \cH_q \left( A \right), \quad \forall q>0.
\end{align*}
\end{lemma}

\begin{proof}
By definition, for $q>0$,
\begin{align*}
    \cH_q \left( \cL(A) \right)
    &= \lim_{\varepsilon \to 0} \inf \bigg\{ \sum_i (2 r_i)^q:\, \cL(A) \subseteq \bigcup_{i=1}^{\infty} \mathfrak B_{r_i}(x_i), \  r_i < \varepsilon \bigg\} \\
    &\le \lim_{\varepsilon \to 0} \inf \bigg\{ \sum_i (2 r_i)^q:\, A \subseteq \bigcup_{i=1}^{\infty} \mathfrak B_{c_\cL r_i} \left( \cL^{-1}(x_i) \right), \  r_i < \varepsilon \bigg\} \\
    &= c_\cL^{-q} \lim_{\epsilon \to 0} \inf \bigg\{ \sum_i (2 s_i)^q:\, A \subseteq \bigcup_{i=1}^{\infty} \mathfrak B_{s_i} \left( y_i \right), \  s_i < \epsilon \bigg\} \\
    &= c_\cL^{-q} \cH_q \left( A \right),
\end{align*}
where we change of variables $y_i = \cL^{-1}(x_i)$, $s_i=c_\cL r_i$ and $\epsilon = c_\cL \varepsilon$.
\end{proof}

As a consequence, we have the following relationship of Hausdorff dimension of any set $A$ and its image under $\cL$.

\begin{corollary} \label{Coro-trans}
Let $\cL$ be a bijective mapping from $\bR^n$ to $\bR^n$ satisfying \eqref{ineq-cL}, then for any $A \subseteq \bR^n$,
\begin{align} \label{ineq-dimh}
    \dimh(A) \ge \dimh \left( \cL(A) \right).
\end{align}
Moreover, if both $\cL$ and $\cL^{-1}$ satisfy \eqref{ineq-cL}, then
\begin{align} \label{eq-dimh}
    \dimh(A) = \dimh \left( \cL(A) \right).
\end{align}
\end{corollary}

\begin{proof}
For any $q>\dimh(A)$, we have $\cH_q(A) = 0$. Thus, by Lemma \ref{Lem-trans}, $\cH_q(\cL(A)) = 0$. Thus, by definition, we get $\dimh(\cL(A)) \le q$. Then the proof of \eqref{ineq-dimh} is concluded by setting $q \to \dimh(A)$.

Applying \eqref{ineq-dimh} for the mapping $\cL^{-1}$ and the set $\cL(A)$, we obtain
\begin{align*}
    \dimh(\cL(A)) \ge \dimh \left( \cL^{-1}(\cL(A)) \right) = \dimh(A).
\end{align*}
\end{proof}

The following lemma helps to verify the assumption in Corollary \ref{Coro-trans} for a class of matrix linear transformation $\cL$.

\begin{lemma} \label{Lem-T-Wigner}
Let $T, \tilde T$ be real or complex invertible matrices of dimension $d_1$ and $d_2$ respectively.
Then the mapping $\cL: \bR^{d_1 \times d_2} \to \bR^{d_1 \times d_2}$ or $\cL: \bC^{d_1 \times d_2} \to \bC^{d_1 \times d_2}$ given by $\cL(X) = T X \tilde T$ is well-defined, bijective with inverse $\cL^{-1} = T^{-1} X \tilde T^{-1}$. Moreover, both $\cL$ and $\cL^{-1}$ satisfy \eqref{ineq-cL}.
\end{lemma}

\begin{proof}
By direct computation, $\cL$ is well-defined and invertible. In the following, we will prove that $\cL$ and $\cL^{-1}$ satisfy \eqref{ineq-cL} with $c_\cL = c_{\cL^{-1}} = C_T^2 d_1d_2$, where
\begin{align*}
    C_T = \max \left\{ \left| T_{i,j} \right| + \left| \left( T^{-1} \right)_{i,j} \right| + \left| \tilde T_{i',j'} \right| + \left| \left( \tilde T^{-1} \right)_{i',j'} \right| : i,j \in [d_1], i',j' \in [d_2] \right\} < \infty.
\end{align*}

Note that for any $d_1 \times d_2$ matrix $X$, we have
\begin{align*}
    \left| \left( T X \tilde T \right)_{i,j} \right|
    = \left| \sum_{k=1}^{d_1} \sum_{l=1}^{d_2} T_{i,k} X_{k,l} \tilde T_{l,j} \right|
    \le C_T^2 \sum_{k=1}^{d_1} \sum_{l=1}^{d_2} \left| X_{k,l} \right|,
    \quad \forall i \in [d_1], j \in [d_2].
\end{align*}
Hence, by Cauchy-Schwarz inequality, for any $d_1 \times d_2$ matrices $X,Y$,
\begin{align} \label{ineq-cL'}
    &\left\| \cL(Y) - \cL(X) \right\|^2
    = \sum_{i=1}^{d_1} \sum_{j=1}^{d_2} \left| \left( T (X-Y) \tilde T \right)_{i,j} \right|^2
    \le C_T^4 d_1d_2 \left( \sum_{k=1}^{d_1} \sum_{l=1}^{d_2} \left| (X-Y)_{k,l} \right| \right)^2 \nonumber \\
    &\le C_T^4 d_1^2d_2^2 \sum_{k=1}^{d_1} \sum_{l=1}^{d_2} \left| (X-Y)_{k,l} \right|^2
    = C_T^4 d_1^2d_2^2 \left\| X-Y \right\|^2.
\end{align}
Similarly,
\begin{align} \label{ineq-cL-1}
    \left\| \cL^{-1}(Y) - \cL^{-1}(X) \right\|
    \le C_T^2 d_1d_2 \left\| X-Y \right\|.
\end{align}
The proof is concluded by \eqref{ineq-cL'} and \eqref{ineq-cL-1}.
\end{proof}

\subsection{Hitting probabilities}

In this subsection, we collect some results from \cite{Xiao2009} and \cite{Lee2023} on the hitting probabilities of Gaussian random fields that satisfy conditions {\bf (A1)} - {\bf (A5)}.

Let $n \in \bN$, let $W = \{(W_1(t), \ldots, W_n(t)): t \in \bR_+^N\}$ be an $n$-dimensional Gaussian random field, whose entries are independent copies of $\{\xi(t): t \in \bR_+^N\}$.
In view of Lemma \ref{Rmk-condition A1}, the following lemma is a direct consequence of \cite{Xiao2009}*{Theorem 2.1, Theorem 2.3 and Theorem 2.5}.

\begin{lemma} \cite{Xiao2009}*{Theorem 2.1, Theorem 2.3 and Theorem 2.5} \label{Lemma-Xiao-Hitting prob}
Consider the interval $I$ of the form \eqref{def-interval}. Suppose that the assumptions {\bf (A1)}, {\bf (A4)} and {\bf (A5)} hold. Then we have the following:
\begin{enumerate}
    \item Let $B \subseteq \bR^n$ be a Borel set, then there exist positive constants $c_5, c_6$ that depend only on $I$, $B$, $H$, such that
    \begin{align*}
    	c_5 \cC_{n-Q} (B)
    	\le \bP \left( W^{-1}(B) \cap I \not= \emptyset \right)
    	\le c_6 {\mathbf H}_{n-Q} (B),
    \end{align*}
    where $Q = \sum_{j=1}^N \frac{1}{H_j}$.
    \item Let $B \subseteq \bR^n$ be a Borel set such that $\dimh (B) \ge n - Q$. Then the following statements hold:
    \begin{enumerate}
	   \item[(a)] Almost surely,
    	\begin{align*}
    		\dimh (W^{-1}(B) \cap I)
    		\le \min_{1 \le i \le N} \bigg\{ \sum_{j=1}^i \dfrac{H_i}{H_j} + N - i - H_i (n - \dim_H(B)) \bigg\}.
    	\end{align*}
    	\item[(b)] Assume that $\dimh (B) > n - Q$ and there is a finite constant $c_{7} \ge 1$ with the following property: 
        For every $\eta \in (0, \, \dimh (B))$ there is a finite Borel measure $\mu_{\eta}$ with compact support in $B$ such that
        \begin{equation}\label{Eq:mu0}
            \mu_{\eta}\big(\mathfrak B_\rho(x)\big) \le c_{7}\, \rho^{\eta} \qquad \hbox{ for all } x \in \bR^n  \hbox{ and } \rho > 0. 
        \end{equation}
        Then with positive probability,
	   \begin{align*}
	       \dimh (W^{-1}(B) \cap I)
	       \ge \min_{1 \le i \le N} \bigg\{ \sum_{j=1}^i \dfrac{H_i}{H_j} + N - i - H_i (n - \dim_H(B)) \bigg\}. 
	   \end{align*}
    \end{enumerate}
\end{enumerate}
\end{lemma}

The key feature of Condition (\ref{Eq:mu0}) is that the constant $c_{7}$ is independent of $\eta$, even though the probability measure $\mu_{\eta}$ may depend on $\eta$. For the proofs of Theorems \ref{Thm-hitting prob-real}, \ref{Thm-hitting prob-complex}, \ref{Thm-hitting prob-complex} and \ref{Thm-hitting prob-compex-Wishart}, we take  $\mu_{\eta}$ as the restriction of the Lebesgue measure on the manifolds ${\mathbf S}(d;l_1,\ldots,l_r) \cap \mathfrak B_{\delta_0}(x_0)$, ${\mathbf H}(d;l_1,\ldots,l_r) \cap \mathfrak B_{\delta_0}(x_0)$, ${\mathbf M}_\bR(d_1 \times d_2;l_1,\ldots,l_r) \cap \mathfrak B_{\delta_0}(x_0)$, ${\mathbf M}_\bC(d_1 \times d_2;l_1,\ldots,l_r) \cap \mathfrak B_{\delta_0}(x_0)$ respectively. 
Both of these measures are independent of $\eta$ and they satisfy (\ref{Eq:mu0}). 

Lemma \ref{Lemma-Xiao-Hitting prob} has less useful information when $\dimh (B) = d-Q$. This case is referred to the critical dimension of hitting probability, and is studied in \cite{Lee2023}. The following lemma is a case of \cite{Lee2023}*{Corollary 2.4}.

\begin{lemma} \cite{Lee2023}*{Corollary 2.4} \label{Lemma-Hitting prb}
Let Assumptions {\bf (A1)}-{\bf (A3)} hold and let $B\subset \bR^n$ be a bounded set. Assume that $n > Q$ and $\dimh (B) = n-Q$, such that
\begin{align} \label{con-lemma-lambda}
    \lambda_n \left( B^{(r)} \right) \le C_B r^Q \left( \ln \ln (1/r) \right)^\kappa
\end{align}
for a constant $\kappa<1$, where $\lambda_n$ is the $n$-dimensional Lebesgue measure and
\begin{align*}
    B^{(r)} = \left\{ x \in \bR^n: \inf_{y \in F} |x-y| \le r \right\}
\end{align*}
is the $r$-neighborhood of $F$. Then $X^{-1}(F)\cap I=\emptyset$ almost surely. 
\end{lemma}

The next lemma, which is \cite{Lee2023}*{Lemma 4.1}, helps to verify the condition \eqref{con-lemma-lambda} of Lemma \ref{Lemma-Hitting prb}.

\begin{lemma} \label{Lemma-condition-verify}
Let $J \subset \bR^m$ be a compact interval and $G: J \to \bR^n$ be a Lipschitz function, where $1\le m \le n$. Let $F = G(J)$, then, we have
\begin{align*}
    \lambda_n(F^{(r)}) \le C r^{n-m}
\end{align*}
for all $r > 0$ small, where $C$ is a positive constant only depending on $J, m,n$ and the Lipschitz constant of the function $G$.
\end{lemma}

The following corollary says that the matrix linear transformation $\cL$ preserve the condition \eqref{con-lemma-lambda}. The proof is a direct application of \eqref{ineq-cL'} and \eqref{ineq-cL-1}, and is omitted.

\begin{corollary} \label{Coro-trans-critical}
Let $T, \tilde T$ and $\cL$ be as in Lemma \ref{Lem-T-Wigner}. If \eqref{con-lemma-lambda} holds for $B$, then it also holds for $\cL(B)$.
\end{corollary}

\subsection{Manifolds and Lie group}

In this subsection, we present some results in manifolds and Lie groups from \cite{Lee2013}. We start with the some conceptions in manifolds.

\begin{definition}
Suppose $M$ and $N$ are smooth manifolds with or without boundary. A smooth map $F:M \to N$ is called a smooth submersion if its differential is surjective at each point. It is called a smooth immersion if its differential is injective at each point.
\end{definition}

The following lemma is part of \cite{Lee2013}*{Proposition 4.28}, which says that a smooth submersion is always open.

\begin{lemma} {\cite{Lee2013}*{Proposition 4.28}} \label{Lemma-open}
Let $M$ and $N$ be smooth manifolds, and suppose $\pi: M \to N$ is a smooth submersion. Then $\pi$ is an open map.
\end{lemma}

\begin{definition}
Suppose $M$ and $N$ are smooth manifolds with or without boundary. If $\pi: M \to N$ is any continuous map, a section of $\pi$ is a continuous right inverse for $\pi$, i.e., a continuous map $\sigma: N \to M$ such that $\pi \circ \sigma = \mathrm{Id}_N$. A local section of $\pi$ is a continuous map $\sigma: U \to M$ defined on some open subset $U \subseteq N$ and satisfying the analogous relation $\pi \circ \sigma = \mathrm{Id}_U$.
\end{definition}

\begin{lemma} {\cite{Lee2013}*{Theorem 4.26}} \label{Lemma:local section}
Suppose $M$ and $N$ are smooth manifolds and $\pi: M \to N$ is a smooth map. Then $\pi$ is a smooth submersion if and only if every point of $M$ is in the image of a smooth local section of $\pi$.
\end{lemma}

In the following, we turn to the conceptions in Lie groups. We begin with the definition of Lie group.

\begin{definition}
A Lie group is a smooth manifold $G$ (without boundary) that is also a group in the algebraic sense, such that the mapping of group multiplication and group inverse are both smooth.
\end{definition}

Next, we define the group action, which plays an important role in defining quotient map.

\begin{definition}
Let $G$ be a group and $M$ be a set.
\begin{itemize}
    \item A left action of $G$ on $M$ is a map $G \times M \to M$ with the notation $(g,p) \mapsto g \cdot p$ such that
    \begin{align*}
        g_1 \cdot (g_2 \cdot p) = (g_1g_2) \cdot p \ \mathrm{and} \
        e \cdot p = p, \quad \forall g_1,g_2 \in G, p \in M,
    \end{align*}
    where $e$ is the identity element of the group $G$.
    
    \item A right action of $G$ on $M$ is a map $M \times G \to M$ with the notation $(p,g) \mapsto p \cdot g$ such that
    \begin{align*}
        (p \cdot g_1) \cdot g_2 = p \cdot (g_1g_2) \ \mathrm{and} \
        p \cdot e = p, \quad \forall g_1,g_2 \in G, p \in M.
    \end{align*}

    \item If in addition $M$ is a smooth manifold with or without boundary, $G$ is a Lie group, and the defining map $G \times M \to M$ ($M \times G \to M$ resp.) is smooth, then the left (right resp.) action is said to be a smooth action.

    \item The action is said to be free if the only element of $G$ that fixes any element of $M$ is the identity element.

    \item For each $p \in M$, we denote by $G \cdot p = \{g \cdot p: g \in G\}$ the orbit of $p$. The set of orbits is denoted by $M/G$, and is called the orbit space of the action if it equips the quotient topology.
\end{itemize}
\end{definition}

\begin{remark} \label{Rmk-Lie subgroup}
The set of $d$-dimensional orthogonal matrices ${\mathbf O}(d)$ and the set of $d$-dimensional unitary matrices ${\mathbf U}(d)$ are compact Lie groups. One can show that ${\mathbf O}(d)$ is a smooth submanifold of $\bR^{d^2}$ of dimension $d(d-1)/2$ and ${\mathbf U}(d)$ is a smooth submanifold of $\bC^{d^2} \cong \bR^{2d^2}$ of dimension $d^2$. See, e.g., \cite{Jaramillo2018}*{page 7}.
\end{remark}

Now we define the proper action of Lie group.

\begin{definition}
A continuous left action of a Lie group $G$ on a manifold $M$ is said to be a proper left action if the map $G \times M \to M \times M$ given by $(g,p) \mapsto (g \cdot p, p)$ is a proper map. The proper right action can be defined similarly.
\end{definition}

The following lemma can be used to verify that a action is proper.

\begin{lemma} {\cite{Lee2013}*{Corollary 21.6}} \label{Lemma-proper}
Every continuous action by a compact Lie group on a manifold is
proper.
\end{lemma}

Now we are able to introduce the following results, which is known as the Quotient Manifold Theorem.

\begin{lemma}{\cite{Lee2013}*{Theorem 21.10}} \label{Lemma-Quotient}
Suppose $G$ is a Lie group acting smoothly, freely, and properly on a smooth manifold $M$. Then the orbit space $M/G$ is a topological manifold of dimension equal to $\dim M - \dim G$, and has a unique smooth structure with the property that the quotient map $\pi:M \to M/G$ is a smooth submersion.
\end{lemma}

\section{Multiple collisions of eigenvalue processes}

\subsection{The real case: proof of Theorem \ref{Thm-hitting prob-real}}\label{sec:real-case}

In this subsection, we deal with the matrix \eqref{def-entries} for $\beta = 1$ and prove Theorem \ref{Thm-hitting prob-real}.

Fix $r \in \bN$, $l_1, \ldots, l_r \in \{1, \ldots, d\}$ satisfying $\sum_{j=1}^r l_j \le d$. For simplicity, we denote $l = d - \sum_{j=1}^r (l_j-1)$ be the number of distinct eigenvalues, and $L_i = d - \sum_{j=i}^r l_j + 1$. Then one can check that $L_1 = l-r+1$.

Firstly, we provide an upper bound for the dimension of ${\mathbf S}(d;l_1, \ldots, l_r)$ given in \eqref{def-degenerate-real}. For the dimension of ${\mathbf S}(d;l_1, \ldots, l_r)$, the upper bound $\frac12[d(d+1) - \sum_{j=1}^r (l_j-1)(l_j+2)]$ is a direct consequence of Lemma \ref{Lem:dim-S} below.

\begin{lemma} \label{Lem:dim-S}
Let $\Delta: \bR^l \rightarrow {\mathbf D}(d)$ be a function that maps any vector $u= (u_1, \ldots, u_l) \in \bR^l$ to a $d\times d$ diagonal matrix $\Delta(u)$ with entries given by
\begin{align} \label{def-Lamda mapping}
	\Delta_{i,i}(u) =
	\begin{cases}
	u_i, & 1 \le i \le L_1-1; \\
	u_{l-r+j}, & L_j \le i \le L_j+l_j-1, j = 1, \ldots, r.
	\end{cases}
\end{align}
Then there exists a compactly supported smooth function $\Gamma: \bR^{\frac12[d(d-1) - \sum_{j=1}^r l_j(l_j-1)]} \rightarrow \mathbf O(d)$, such that the mapping
\[G: \bR^l \times \bR^{\frac12[d(d-1) - \sum_{j=1}^r l_j(l_j-1)]} \rightarrow {\mathbf S}(d)
\]
given by
\begin{align} \label{def-F function}
	G(u, v) = \Gamma(v) \Delta(u) \Gamma(v)^*,~~ (u,v) \in \bR^l \times \bR^{\frac12[d(d-1) - \sum_{j=1}^r l_j(l_j-1)]},
\end{align}
satisfies
\begin{align} \label{eq-Lem:dim-S}
	{\mathbf S}(d;l_1, \ldots, l_r) \subseteq \left\{ x \in \bR^{d(d+1)/2}: \ x \in \mathrm{Im}(G) \right\},
\end{align}
where $\im(G)$ is the image of the mapping $G$.
\end{lemma}

\begin{proof}
We denote the set of matrices
\begin{align} \label{eq-O}
	&{\mathbf O}(d;l_1, \ldots, l_r)
	= \big\{ B = \rd (b_1, \ldots, b_{l-r}, P_1, \ldots, P_r) \in {\mathbf O}(d) : \nonumber \\
    & \quad\quad \quad\quad \quad\quad \quad\quad \quad\quad \quad\quad b_{i} \in {\mathbf O}(1), P_j \in {\mathbf O}(l_j), \forall 1 \le j \le r, 1 \le i \le l-r \big\}.
\end{align}
With the help of Remark \ref{Rmk-Lie subgroup}, one can easily check that ${\mathbf O}(d;l_1, \ldots, l_r)$ equipped with the matrix multiplication is a Lie subgroup of ${\mathbf O}(d)$ with dimension $\sum_{j=1}^r l_j(l_j-1)/2$. Thus, we can define the right action of the Lie group ${\mathbf O}(d;l_1, \ldots, l_r)$ on the manifold ${\mathbf O}(d)$ as the matrix multiplication. Note that ${\mathbf O}(d)$ is also a Lie group, which means that the group multiplication map and the group inverse map are smooth. Thus, the right action is smooth. Moreover, by the compactness of ${\mathbf O}(d)$ and Lemma \ref{Lemma-proper}, the right action is proper. Furthermore, by checking matrix multiplication, we can see that the right action is free. Therefore, by Lemma \ref{Lemma-Quotient}, the orbit space ${\mathbf O}(d) / {\mathbf O}(d;l_1, \ldots, l_r)$ is a smooth manifold of dimension $\frac12[d(d-1) - \sum_{j=1}^r l_j(l_j-1)]$, and the quotient map
\begin{align*}
    \pi: {\mathbf O}(d) \longrightarrow {\mathbf O}(d) / {\mathbf O}(d;l_1, \ldots, l_r)
\end{align*}
is a smooth submersion. Furthermore, by Lemma \ref{Lemma-open}, $\pi$ is an open map.

For $\epsilon > 0$, let $I_{\epsilon}$ be the following open interval
\begin{equation}\label{e:Je}
	I_{\epsilon} = (-\epsilon, \epsilon)^{\frac12[d(d-1) - \sum_{j=1}^r l_j(l_j-1)]}.
\end{equation}

For any $A \in {\mathbf O}(d)$, by the definition of chart (see e.g. \cite{Tu2011}*{Definition 5.1}), there exists a neighbourhood of $\pi(A) \in {\mathbf O}(d) / {\mathbf O}(d;l_1, \ldots, l_r)$ that is smoothly diffeomorphic to $I_{\epsilon}$. Moreover, since $\pi$ is a submersion, by Lemma \ref{Lemma:local section}, there exists a smooth local section $\sigma_{\pi}$ from an open neighbourhood of $\pi(A)$ to ${\mathbf O}(d)$, such that $\sigma_{\pi}$ maps $\pi(A)$ to $A$ and $\pi \circ \sigma_{\pi}$ is identity. Note that $\pi$ is open. Thus, we can choose $\delta$ small enough, such that there exist a positive number $\epsilon$ (which may depend on $A$), an open neighbourhood $U(\pi(A))$ of $\pi(A)$ in the domain of $\sigma_{\pi}$ and a smooth diffeomorphism
\begin{align} \label{def-family of varphi}
	\phi: I_{\epsilon} \rightarrow U(\pi(A)) \subset \left( {\mathbf O}(d) / {\mathbf O}(d;l_1, \ldots, l_r) \right) \cap \mathfrak \pi(\mathfrak B_{\delta}(A))
\end{align}
satisfying $\phi(0) = \pi(A)$. We write $U_{\phi}(\pi(A))$ for $U(\pi(A))$ in the following to emphasize that the neighbourhood is associated with $\phi$. Here $\mathfrak B_\delta(A)$ is the open ball with radius $\delta$ centered at $A$ in the space $\mathbf O(d)$ under the Frobenius norm $\|\cdot\|$ given by \eqref{eq:F-norm}.

Therefore, we can construct a smooth function
\begin{align*}
    \Gamma: \bR^{\frac12[d(d-1) - \sum_{j=1}^r l_j(l_j-1)]} \rightarrow {\mathbf O}(d)
\end{align*}
supported on $I_{2\epsilon}$, such that
\begin{align} \label{def-mapping Pi}
	\Gamma(v) = \sigma_{\pi} \circ \phi(v), \forall v\in I_{\epsilon}.
\end{align}

Since $\phi$ is a diffeomorphism, the set $U_{\phi}(\pi(A)) = \phi(I_{\epsilon})$ is a subset of ${\mathbf O}(d) / {\mathbf O}(d;l_1, \ldots, l_r)$ that is open and contains $\pi(A)$. Hence, the collection of the sets $\{U_{\phi}(\pi(A)): \pi(A) \in {\mathbf O}(d) / {\mathbf O}(d;l_1, \ldots, l_r)\}$ is an open cover for ${\mathbf O}(d) / {\mathbf O}(d;l_1, \ldots, l_r)$. By the compactness of ${\mathbf O}(d)$ and the continuity of $\pi$, ${\mathbf O}(d) / {\mathbf O}(d;l_1, \ldots, l_r)$ is compact. Then one can find a finite subcover $\{U_{\phi^{(i)}}(\pi(A_i)), i=1, \dots, M\}$ for some $M \in \bN$, such that
\begin{align} \label{eq-finite open cover}
	{\mathbf O}(d) / {\mathbf O}(d;l_1, \ldots, l_r) = \bigcup_{i=1}^M U_{\phi^{(i)}}(\pi(A_i)),
\end{align}
where $A_1, \dots, A_M$ are distinct matrices in $\mathbf O(d)$, $\phi^{(1)}, \ldots, \phi^{(M)}$ are smooth mappings of the form \eqref{def-family of varphi}, and $U_{\phi^{(i)}}(\pi(A_i)) = \phi^{(i)}(I_{\epsilon_i})$.

For $1 \le i\le M$, we define the mapping $G^{(i)}: \bR^l \times I_{\epsilon_i} \rightarrow {\mathbf S}(d)$ by
\begin{align} \label{eq-def-Gi}
	G^{(i)}(u,v) = \Gamma^{(i)}(v) \Delta(u) \Gamma^{(i)}(v)^*,
\end{align}
for $v \in I_{\epsilon_i}$ and $u \in \bR^l$.

For an arbitrary fixed $x \in {\mathbf S}(d;l_1, \ldots, l_r)$, we have the spectral decomposition $x = Q D Q^*$ for some $Q \in {\mathbf O}(d)$ and $D \in {\mathbf D}(d)$. By rearranging the diagonal entries of $D$ and the columns of $Q$, we can assume that $D_{L_j, L_j} = \cdots = D_{L_j+l_j-1, L_j+l_j-1}$ for $1 \le j \le r$. By \eqref{eq-finite open cover}, there exists $ i_0 \in\{1, \dots,  M\}$ such that $\pi(Q) \in U_{\phi^{(i_0)}}(\pi(A_{i_0}))$, and hence one can find $v \in I_{\epsilon_{i_0}}$ such that $\pi(Q) = \phi^{(i_0)}(v)$. By the relationship \eqref{def-mapping Pi}, we can write $\Gamma^{(i_0)}(v) = \sigma_{\pi} \circ \pi (Q)$. Note that $\sigma_{\pi} \circ \pi (Q)$ and $Q$ are in the same orbit space of $\pi$, we have $\sigma_{\pi} \circ \pi (Q) = Q \cdot g$, where $g$ is an element in the Lie group ${\mathbf O}(d;l_1, \ldots, l_r)$. Let $u = (D_{1,1}, D_{2,2}, \ldots, D_{L_1-1,L_1-1}, D_{L_1, L_1}, D_{L_2,L_2}, \ldots, D_{L_r,L_r}) \in \bR^l$ then $D = \Delta(u)$. Thus, by the relationship \eqref{def-mapping Pi}, we have
\begin{align*}
    x
	= QDQ^*
    = Q (gDg^*) Q^*
	= (Q \cdot g) D (Q \cdot g)^*
	= \Gamma^{(i_0)}(v) \Delta(u) \Gamma^{(i_0)}(v)^*
	= G^{(i_0)} (u,v).
\end{align*}
Since $x\in {\mathbf S}(d;l_1, \ldots, l_r)$ is arbitrarily chosen, we conclude that
\begin{align*}
	{\mathbf S}(d;l_1, \ldots, l_r) \subseteq \left\{ x \in \bR^{d(d+1)/2}: x \in \bigcup_{i=1}^M \mathrm{Im}(G^{(i)}) \right\}.
\end{align*}

Finally, let $\epsilon = \max \{\epsilon_1, \ldots, \epsilon_M\}$, then for any smooth function $\Gamma$ supported on $I_{4M\epsilon}$ satisfying
\begin{align*}
	\Gamma(y) = \Gamma^{(j+1)} \left( y - (4j\epsilon, 0, \dots, 0)\right), 
\end{align*}
for $y \in I_{2\epsilon} + \left\{ \left( 4j \epsilon, 0, \ldots, 0 \right) \right\} \subseteq \bR^{\frac12[d(d-1) - \sum_{j=1}^r l_j(l_j-1)]}$ for some  $j\in\{0,1,\dots, M-1\}$, the mapping $G$ defined by \eqref{def-F function} satisfies \eqref{eq-Lem:dim-S}.
\end{proof}

\begin{remark} \label{Rmk-dim-S}
We would like to remark that Lemma \ref{Lem:dim-S} was obtained in \cite{Jaramillo2018}*{Proposition 4.5} for the case $r=1, l_1 = 2$ and in \cite{Xiao2021} for the case $r=1$. In the two references, the mapping $\Gamma$ is constructed firstly by finding all the eigenvectors associated to single eigenvalues, and then by finding eigenvectors associated with multiple eigenvalues via orthogonal extension and Gram-Schmidt orthonormalization. However, the argument in \cites{Jaramillo2018, Xiao2021} can not be extended to the case $r \ge 2$, since the orthogonal complement of the space spanned by eigenvectors associated to single eigenvalues is a direct sum of eigenspaces of multiple eigenvalues.
\end{remark}

The following lemma is about the continuity of the spectral decomposition on ${\mathbf S}(d;l_1, \ldots, l_r)$, and it is an extension of \cite{Jaramillo2018}*{Lemma 4.3}. 

\begin{lemma} \label{Lem:approx-S}
Let $A \in {\mathbf S}(d;l_1, \ldots, l_r)$ be a symmetric matrix with spectral decomposition $A = PDP^*$ for some $P \in {\mathbf O}(d)$ and $D \in {\mathbf D}(d)$, such that $|\spec(A)| = l$. Then for any $\epsilon > 0$, there exists $\delta > 0$, such that for all $B \in {\mathbf S}(d;l_1, \ldots, l_r)$ satisfying
\begin{align*}
	\max_{1 \le i, j \le d} |A_{i,j} - B_{i,j}| < \delta,
\end{align*}
we have $|\spec(B)|=l$, and there exists a spectral decomposition $B = Q FQ^*$, where $Q \in {\mathbf O}(d)$ and $F \in {\mathbf D}(d)$ satisfy
\begin{align} \label{eq:eigen-continuity}
    \max_{1 \le i \le d} |D_{i,i} - F_{i,i}| < \epsilon, \quad 	\max_{1 \le i, j \le d} |Q_{i,j} - P_{i,j}| < \epsilon.
\end{align}
\end{lemma}

\begin{proof}
We follow the idea of the proof of \cite{Jaramillo2018}*{Lemma 4.3}.
The first inequality of \eqref{eq:eigen-continuity}, which describes the continuity of the eigenvalues, follows directly from the continuity of the functions $E_1, \ldots, E_d$ which are introduced in Section \ref{sec:notations}. The second inequality of \eqref{eq:eigen-continuity} claims that eigenvectors, considered as functions of matrices in ${\mathbf S}(d;l_1, \ldots, l_r)$, are continuous at $A\in {\mathbf S}(d;l_1, \ldots, l_r)$ with $\spec(A)=l$. The key idea to prove this is to represent eigenprojections as matrix-valued Cauchy integrals. 

Noting that $A\in {\mathbf S}(d;l_1, \ldots, l_r)$ and $|\spec(A)| = l$, without loss of generality we assume 
\begin{align} \label{eq-order of eigenvalue}
	D_{1,1} < \cdots < D_{L_1,L_1} = \cdots = D_{L_2-1,L_2-1} < D_{L_2,L_2} = \cdots = D_{L_3-1,L_3-1} < \cdots < D_{L_r,L_r} = \cdots = D_{d,d}.
\end{align}
Denote by $\tilde{I} = [L_1-1] \cup \{L_j:1 \le j \le r\}$ the indexes of distinct eigenvalues. For $i \in \tilde{I}$, let $\cC_i^A \subseteq \bC \setminus \spec(A)$ be any smooth closed curve around $D_{i,i}$ and denote by $\cI_i^A$ the closure of the interior of $\cC_i^A$. We can choose the curves $\{\cC_i^A: i \in \tilde{I}\}$ with sufficiently small diameters so that $\{\cI_i^A: i \in \tilde{I}\}$ are disjoint. For simplicity, let $\cC_i^A = \cC_{L_j}^A$ and $\cI_i^A = \cI_{L_j}^A$ for $L_j < i \le L_j+l_j-1$, $1 \le j \le r$. 

By the continuity of the functions $E_1, E_2, \ldots, E_{L_1}, E_{L_2}, \ldots, E_{L_r}$ and \eqref{eq-order of eigenvalue}, there exists $\delta > 0$, such that for all $B$ in the $\delta$-neighbourhood in ${\mathbf S}(d;l_1, \ldots, l_r)$ of $A$
\begin{align*}
	U_{\delta} = \left\{ B \in {\mathbf S}(d;l_1, \ldots, l_r): \max_{1 \le i, j \le d} |A_{i,j} - B_{i,j}| < \delta \right\},
\end{align*}
we have $E_1(B) < \cdots < E_{L_1}(B) < E_{L_2(B)} < \ldots < E_{L_r} (B)$ and $E_i(B) \in \cI_i^A$ for $i \in \tilde{I}$. Noting that  $U_{\delta} \subseteq {\mathbf S}(d;l_1, \ldots, l_r)$, for all $B \in U_{\delta}$, we have
\begin{align} \label{eq-strictly order of eigenvalue}
	E_1(B) < \cdots < E_{L_1}(B) = \cdots = E_{L_2-1}(B) < E_{L_2}(B) = \cdots = E_{L_3-1}(B) < \ldots < E_{L_r}(B) = \cdots = E_{d} (B),
\end{align}
and
\begin{align} \label{eq-region of eigenvalue}
	E_i(B) \in \cI_i^A, \quad 1 \le i \le d.
\end{align}
Thus, for any $i \in [d]$, for any $\zeta \in \cC_i^A$, \eqref{eq-region of eigenvalue} implies that the matrix $\zeta I_d - B$ is invertible. Hence, we can define the mappings $\Theta_i^A: U_{\delta} \rightarrow {\mathbf S}(d)$ by the following matrix-valued Cauchy integrals:
\begin{align*}
	\Theta_i^A(B) = \dfrac{1}{2\pi\iota} \oint_{\cC_i^A} (\zeta I_d - B)^{-1} d\zeta\,.
\end{align*}
For $B_1,B_2 \in U_\delta$, we have
\begin{align*}
    \left( \zeta I_d - B_1 \right)^{-1} - \left( \zeta I_d - B_2 \right)^{-1}
    = \left( \zeta I_d - B_1 \right)^{-1} \left( B_1 - B_2 \right) \left( \zeta I_d - B_2 \right)^{-1},
\end{align*}
which implies that the mappings $\Theta_i^A$ are continuous with respect to $B$ for $B\in U_\delta$. By \cite{Lax2002}*{page 200, Theorem 6}, the matrix $\Theta_i^A(B)$ is a projection over the sum of the eigenspaces associated with eigenvalues of $B$ that are inside $\cI_i^A$. Hence, by \eqref{eq-strictly order of eigenvalue} and \eqref{eq-region of eigenvalue}, $\Theta_i^A(B)$ is a projection over the eigenspace $\cE_{E_i(B)}^B$ for $1 \le i \le d$, noting that $\{\cI_i^A: i \in \tilde{I}\}$ are disjoint.

For $1 \le j \le L_1-1$, we define
\begin{align} \label{eq-def of w1}
	w^j = \dfrac{\Theta_j^A(B) P_{*,j}}{\left\| \Theta_j^A(B) P_{*,j} \right\|},
\end{align}
which are unit eigenvectors of $E_j(B)$ for $j=1,\dots, L_1-1$. Recall that $\Theta_j^A(B)$ is a continuous function of $B$ for $B\in U_\delta$ with $\delta$ sufficiently small and that $\Theta_j^A(A) P_{*,j} = P_{*,j}$,  for all $1 \le j \le d$. Hence, for all $1 \le i \le r$, for $L_i \le j \le L_i+l_i-1$, noting that $\mathcal E_{E_{L_i}(B)}^B= \cdots =\mathcal E_{E_{L_i+l_i-1}(B)}^B$ is a $l_i$-dimensional vector space and the vectors $P_{*,L_i}, \ldots, P_{*,L_i+l_i-1}$ are orthogonal, we can see that $\Big\{ \Theta_{L_i}^A(B) P_{*,L_i}, \cdots, \Theta_{L_i+l_i-1}^A(B) P_{*,L_i+l_i-1} \Big\}$ is the set of linearly independent eigenvectors associated with the same eigenvalues $E_{L_i}$ for $\delta$ being suggiciently small. Thus, we define iteratively for $B\in U_\delta$ with sufficiently small $\delta$, by applying the Gram-Schmidt orthonormalizing process to obtain the following unit eigenvectors of $E_{L_i}(B)$
\begin{equation} \label{eq-def of w2}
	\begin{cases}
		w^{L_i}(B) = \dfrac{\Theta_{L_i}^A(B) P_{*,L_i}}{\left\| \Theta_{L_i}^A(B) P_{*,L_i} \right\|},\\
		w^j(B) = \dfrac{\dfrac{\Theta_j^A(B) P_{*,j}} {\big\| \Theta_j^A(B) P_{*,j} \big\|} - \sum_{j'=L_i}^{j-1} \Big\langle \dfrac{\Theta_j^A(B) P_{*,j}} {\| \Theta_j^A(B) P_{*,j} \|}, w^{j'}(B) \Big\rangle w^{j'}(B)}
        {\bigg\| \dfrac{\Theta_j^A(B) P_{*,j}} {\| \Theta_j^A(B) P_{*,j} \|} - \sum_{j'=L_i}^{j-1} \Big\langle \dfrac{\Theta_j^A(B) P_{*,j}} {\| \Theta_j^A(B) P_{*,j} \|}, w^{j'}(B) \Big\rangle w^{j'}(B) \bigg\|}, \quad  L_i+1 \le j \le L_i+l_i-1.
	 \end{cases}
\end{equation}
Also note that the inner product $\langle u, v\rangle$ is a continuous function of $(u, v)$, and hence $\omega^j(B), j=1,\dots, d$, defined by \eqref{eq-def of w1} and \eqref{eq-def of w2} are continuous functions of $B$ in a sufficiently small neighborhood $U_\delta$ of $A$ with $\omega^j(A) = P_{*,j}$.  Thus,  for any $\epsilon>0$, one can find a sufficiently small positive constant $\delta$, such that for all $B\in U_\delta$, 
\begin{align*}
	\max_{1 \le j \le d} \| P_{*,j} - w^j(B)\| < \epsilon.
\end{align*}
Thus if we denote the matrix $Q=(\omega^1(B),\dots, \omega^d(B))$, then $B=QFQ^*$ and the second inequality of \eqref{eq:eigen-continuity} is satisfied.  The proof is concluded.  
\end{proof}

The following result is an extension of \cite{Jaramillo2018}*{Proposition 4.7}, and it shows that the optimal upper bound for the dimension of ${\mathbf S}(d;l_1, \ldots, l_r)$ is $\frac12[d(d+1) - \sum_{j=1}^r(l_j-1)(l_j+2)]$.

\begin{lemma} \label{Lem:dim-bound}
For $x_0 \in {\mathbf S}(d;l_1, \ldots, l_r)$ with $|\spec(x_0)| = l$, there exists $\delta_0 > 0$ such that ${\mathbf S}(d;l_1, \ldots, l_r) \cap \mathfrak B_{\delta_0}(x_0)$ is a $(\frac12[d(d+1) - \sum_{j=1}^r (l_j-1)(l_j+2)])$-dimensional manifold.
In particular, ${\mathbf S}(d;l_1, \ldots, l_r) \cap \mathfrak B_{\delta_0}(x_0)$ has positive $(\frac12[d(d+1) - \sum_{j=1}^r (l_j-1)(l_j+2)])$-dimensional Lebesgue measure.
\end{lemma}

\begin{proof}
By Lemma \ref{Lem:dim-S}, there exist $u_0 \in \bR^l$, $v_0 \in \bR^{\frac12[d(d-1) - \sum_{j=1}^r l_j(l_j-1)]}$, such that $x_0 = G(u_0,v_0)$, where the function $G$ is given by \eqref{def-F function}. Moreover, without loss of generality, we may assume that $x_0 = G^{(1)} (u_0,v_0) = \Gamma^{(1)}(v_0) \Delta(u_0) \Gamma^{(1)}(v_0)^*$, $v_0 \in I_{\epsilon_1}$ and the components of $u_0$ are increasing. To avoid confusing, we rewrite \eqref{def-mapping Pi} as
\begin{align} \label{eq-Gamma (1)}
	\Gamma^{(1)}(v) = \sigma_{\pi}^{(1)} \circ \phi^{(1)}(v), \forall v \in I_{\epsilon_1}.
\end{align}

To show that the manifold ${\mathbf S}(d;l_1, \ldots, l_r) \cap \mathfrak B_{\delta_0}(x_0)$ has dimension $\frac12[d(d+1) - \sum_{j=1}^r (l_j-1)(l_j+2)]$ and has positive $(\frac12[d(d+1) - \sum_{j=1}^r (l_j-1)(l_j+2)])$-dimensional Lebesgue measure, it is sufficient to show that there exist open sets $U \subseteq \bR^l \times I_{\epsilon_1}$ and $V \subseteq {\mathbf S}(d;l_1, \ldots, l_r) \cap \mathfrak B_{\delta_0}(x_0)$, such that the map $G^{(1)}|_U: U \rightarrow V$ is a homeomorphism.

Let 
\begin{align} \label{e:r0}
	r_0 = \dfrac{1}{2} \min_{\substack{\mu , \lambda\in \spec(x_0)\\ \mu\not= \lambda}} |\mu - \lambda|.
\end{align}
For $\gamma_0 \in (0, r_0)$ that is small enough, by Lemma \ref{Lem:approx-S}, there exists $\delta_0$ such that for each $x \in {\mathbf S}(d;l_1, \ldots, l_r) \cap \mathfrak B_{\delta_0}(x_0)$, it has the spectral decomposition $x = Q E Q^*$ with $Q \in {\mathbf O}(d) \cap \mathfrak B_{\gamma_0} \left( \Gamma^{(1)}(v_0) \right)$ and $E \in {\mathbf D}(d) \cap \mathfrak B_{\gamma_0}(\Delta(u_0))$, such that $|\spec(x)|=l$.
Then by the definition \eqref{e:r0} of $r_0$, we have
\begin{align*}
	E_{1,1} < E_{2,2} < \cdots < E_{L_1-1,L_1-1} < E_{L_1,L_1} = \cdots = E_{L_2-1,L_2-1} < E_{L_2,L_2} = \cdots < E_{L_r,L_r} = \cdots = E_{d,d}.
\end{align*}
Denote $u = (E_{1,1}, \dots, E_{L_1-1,L_1-1}, E_{L_1,L_1}, E_{L_2,L_2}, \ldots, E_{L_r,L_r}) \in \bR^l$, then $E = \Delta(u)$.
Recall that $\pi$ is continuous and $\phi^{(1)}$ is a diffeomorphism. We choose $\gamma_0$ to be sufficiently small, then for $Q \in {\mathbf O}(d) \cap \mathfrak B_{\gamma_0} \left( \Gamma^{(1)}(v_0) \right)$, we have $\pi(Q) \in \phi^{(1)}(I_{(\epsilon_1+\|v_0\|)/2})$, noting that $\pi \left( \Gamma^{(1)} (v_0) \right) = \phi^{(1)} (v_0)$. Thus, we choose $v = \left( \phi^{(1)} \right)^{-1} (\pi(Q)) \in I_{(\epsilon_1+\|v_0\|)/2}$, then $\Gamma^{(1)} (v) = \sigma_{\pi}^{(1)} \left( \pi(Q) \right)$ is an orthogonal matrix in the same orbit space as $Q$. Hence, there exists $g \in {\mathbf O}(d;l_1, \ldots, l_r)$ such that $\sigma_{\pi}^{(1)} \circ \pi (Q) = Q \cdot g$. Therefore,
\begin{align}\label{e:hatx}
	x = QEQ^*
	= Q (g E g^*) Q^*
	= (Q \cdot g) E (Q \cdot g)^*
	= \Gamma^{(1)}(v) \Delta (u) \Gamma^{(1)}(v)^*
	= G^{(1)}(u,v).
\end{align}

Now we choose $V = {\mathbf S}(d;l_1, \ldots, l_r) \cap \mathfrak B_{\delta_0}(x_0)$, then by \eqref{e:hatx}, we have $V \subset G^{(1)} (\bR^l \times I_{(\epsilon_1 + \|v_0\|)/2})$. We choose $U = (G^{(1)})^{-1}(V)$, then $\left. G^{(1)} \right|_U$ is surjective. The continuity of the mapping $G^{(1)}$ implies that $U$ is open in $\bR^l \times  I_{(\epsilon_1 + \|v_0\|)/2}$. To show $\left. G^{(1)} \right|_U$ is a homeomorphism, it suffices to show the following conditions are satisfied:
\begin{enumerate}
	\item[(a1)] $\left. G^{(1)} \right|_U$ is injective;
	\item[(a2)] $\left( \left. G^{(1)} \right|_U \right)^{-1}$ is continuous over $V$.
\end{enumerate}

Firstly, we show that $\left. G^{(1)} \right|_U: U \to {\mathbf S}(d;l_1, \ldots, l_r) \cap \mathfrak B_{\delta_0}(x_0) $ is injective. Suppose that for $x \in {\mathbf S}(d;l_1, \ldots, l_r) \cap \mathfrak B_{\delta_0}(x_0)$, it has the following spectral decompositions,
\begin{align*}
	x = \Gamma^{(1)}(v) \Delta(u) \Gamma^{(1)}(v)^* = \Gamma^{(1)}(v') \Delta(u') \Gamma^{(1)}(v')^*.
\end{align*}
We aim to show that $(u,v)=(u',v').$

Denote $u' = (u'_1, \ldots, u'_{L_1}, \ldots, u_{L_r}')$. Note that the set of eigenvalues is the set of zeros of the characteristic polynomial, which is uniquely determined by $x$. Thus, if $u_i \neq u'_i$ for some $i \in \tilde{I}$, then the corresponding unit eigenvectors $\Gamma^{(1)}_{*,i}(v)$ and $\Gamma^{(1)}_{*,i}(v')$ belong to different eigenspaces and hence are orthogonal. On the other hand, by \eqref{def-family of varphi}, there exist $A_1, A_2$ in the open ball with radius $\delta^{(1)}$ in the space $\mathbf O(d)$, such that
\begin{align*}
	\pi \left( \Gamma^{(1)}(v) \right)
	= \phi^{(1)} (v)
	= \pi(A_1), \
	\pi \left( \Gamma^{(1)}(v') \right)
	= \phi^{(1)} (v')
	= \pi(A_2).
\end{align*}
Thus, $\Gamma^{(1)}(v)$ and $\Gamma^{(1)}(v')$ are in the same orbit space as $A_1$ and $A_2$ respectively. Hence, there exist $g_1, g_2 \in {\mathbf O}(d;l_1, \ldots, l_r)$ such that $\Gamma^{(1)}(v) \cdot g_1 = A_1$ and $\Gamma^{(1)}(v') \cdot g_2 = A_2$. Note that $\Gamma^{(1)}(v) \cdot g_1$ is a matrix obtained from $\Gamma^{(1)}(v)$ by applying orthogonal transformation to the set of eigenvectors associate to the same eigenvalue. So is $\Gamma^{(1)}(v') \cdot g_2$. Since the $i$th columns of $\Gamma^{(1)}(v)$ and $\Gamma^{(1)}(v')$ belong to different eigenspaces, we can see that the $i$th columns of $A_1$ and $A_2$ also belong to different eigenspaces. Hence, $\|A_1 - A_2\| \ge \| (A_1)_{*,i} - (A_2)_{*,i} \| \ge \sqrt{2}$, which is a contradiction if we choose $\delta^{(1)} < \sqrt{2}/2$. Therefore, we have $u = u'$.

Note that the components of $u$ are distinct, and the $i$th columns of $\Gamma^{(1)}(v)$ and $\Gamma^{(1)}(v')$ are the eigenvectors associated to the same eigenvalue for all $1 \le i \le d$. Thus, by the definition of $\pi$, we have $\pi \left( \Gamma^{(1)}(v) \right) = \pi \left( \Gamma^{(1)}(v') \right)$. By \eqref{eq-Gamma (1)}, we have $\phi^{(1)}(v) = \phi^{(1)}(v')$. Therefore, $v=v'$ since $\phi^{(1)}$ is a diffeomorphism.

Now we show that the condition (a2) is satisfied.  Consider any sequence $\{x_n\}_{n \in \bN} \subseteq V$, such that $\lim_{n \rightarrow \infty} x_n = x \in V$. Let $(u_n, v_n) = \left( \left. G^{(1)} \right|_U \right)^{-1}(x_n) \in \bR^l \times  I_{(\epsilon_1+\|v_0\|)/2}$ and $(u, v) = \left( \left. G^{(1)} \right|_U \right)^{-1}(x) \in \bR^l \times I_{(\epsilon_1+\|v_0\|)/2}$, then
\begin{align*}
	x_n = \Gamma^{(1)}(v_n) \Delta(u_n) \Gamma^{(1)}(v_n)^*, \quad
	x = \Gamma^{(1)}(v) \Delta (u) \Gamma^{(1)}(v)^*.
\end{align*}
Note that the components of $u_0$ are distinct. By Lemma \ref{Lem:approx-S} and the continuity of the functions $E_1, \ldots, E_d$, it is clear that
\begin{align} \label{eq-continuity of beta}
	\lim_{n \rightarrow \infty} u_n = u.
\end{align}

Note that for all $n \in \bN$, $v_n \in \overline{I_{(\epsilon_1+\|v_0\|)/2}}$, which is compact. Let $v' \in \overline{I_{(\epsilon_1+\|v_0\|)/2}} \subset I_{\epsilon_1}$ be a limit point of the sequence $\{v_n\}_{n \in \bN}$, then there exists a subsequence $\{v_{m_n}\}_{n \in \bN}$, such that $\lim_{n \rightarrow \infty} v_{m_n} = v'$. By the continuity of the mappings $\Gamma^{(1)}$ and $\Delta$, we have
\begin{align*}
	x = \lim_{n \rightarrow \infty} x_{m_n}
	= \lim_{n \rightarrow \infty} \Gamma^{(1)}(v_{m_n}) \Delta(u_{m_n}) \Gamma^{(1)}(v_{m_n})^*
	= \Gamma^{(1)}(v') \Delta(u) \Gamma^{(1)}(v')^*.
\end{align*}
Thus, the columns of $\Gamma^{(1)}(v)$ and $\Gamma^{(1)}(v')$ are the eigenvectors of the same eigenvalues, which implies that $\pi \left( \Gamma^{(1)}(v) \right) = \pi \left( \Gamma^{(1)}(v') \right)$. Since $\pi \circ \sigma_{\pi}^{(1)}$ is identity and $\phi^{(1)}$ is diffeomorphism on $I_{\epsilon_1}$, we have $v=v'$, which together with the compactness of $\overline{I_{(\epsilon_1+\|v_0\|)/2}}$ implies that
\begin{equation}\label{eq-continuity of alpha}
	\lim_{n\to\infty}  v_n =v,
\end{equation}
since any limit point of the sequence $\{v_n\}_{n \in \bN}$ has to be $v$.

By \eqref{eq-continuity of beta} and \eqref{eq-continuity of alpha}, we obtain the continuity of the map $\left(G^{(1)}\right)^{-1}$ and (a2) is proved. The proof is concluded. 
\end{proof}

Now we are ready to prove Theorem \ref{Thm-hitting prob-real}.
\begin{proof}[Proof of Theorem \ref{Thm-hitting prob-real}]
We first prove Part (i), recalling that $\beta=1$.
By Lemma \ref{Lem:dim-S}, there exists a smooth map
\begin{align*}
    G: \bR^{\frac12[d(d+1) - \sum_{j=1}^r (l_j-1)(l_j+2)]} \rightarrow {\mathbf S}(d)
\end{align*}
such that ${\mathbf S}(d;l_1, \ldots, l_r) \subseteq \mathrm{Im}(G)$. Moreover, by Lemma \ref{Lem:dim-bound},
\begin{align} \label{eq-dimh Im G}
    \dimh \left( \mathrm{Im}(G) \right)
    = \dfrac{1}{2} \left( d(d+1) - \sum_{j=1}^r (l_j-1)(l_j+2) \right).
\end{align}
Besides, we have
\begin{equation}\label{Eq:Hit-up}
\begin{split}
	& \bP \left( \exists t \in I, \ \mathrm{s.t. \ } \left| \lambda_{J_j}^{\beta}(t) \right| = 1 \ \mathrm{for \ some \ disjoint \ } J_j \in [d] \mathrm{\ with \ } |J_j| =  l_j, \forall 1 \le j \le r \right) \\
	&= \bP \left( Y^{\beta}(t) \in {\mathbf S}(d;l_1, \ldots, l_r) \mathrm{\ for \ some \ } t \in I \right) \\
	&= \bP \left( X^{\beta}(t) \in T_\beta^{-1} \left( {\mathbf S}(d;l_1, \ldots, l_r) - A^{\beta} \right) \left( T_\beta^* \right)^{-1} \mathrm{\ for \ some \ } t \in I \right) \\
	&\le \bP \left( X^{\beta}(t) \in T_\beta^{-1} \left( \mathrm{Im}(G) - A^{\beta} \right) \left( T_\beta^* \right)^{-1} \mathrm{\ for \ some \ } t \in I \right) \\
	&= \bP \left( X^{\beta}(I) \cap T_\beta^{-1} \left( \mathrm{Im}(G) - A^{\beta} \right) \left( T_\beta^* \right)^{-1} \not= \emptyset \right).
\end{split}
\end{equation}

On one hand, in the case that
\begin{align} \label{ineq-cond-1}
    \sum_{j=1}^N \frac{1}{H_j} < \dfrac{1}{2} \sum_{j=1}^r (l_j-1)(l_j+2),
\end{align}
we apply Lemma \ref{Lemma-Xiao-Hitting prob} to the last term in (\ref{Eq:Hit-up}) to obtain 
\begin{align*}
	& \bP \left( \exists t \in I, \ \mathrm{s.t. \ } \left| \lambda_{J_j}^{\beta}(t) \right| = 1 \ \mathrm{for \ some \ disjoint \ } J_j \in [d] \mathrm{\ with \ } |J_j| =  l_j, \forall 1 \le j \le r \right) \\
	&\le c_6 {\mathbf H}_{\frac{d(d+1)}{2} - \sum_{j=1}^N \frac{1}{H_j}} \left( T_\beta^{-1} \left( \mathrm{Im}(G) - A^{\beta} \right) \left( T_\beta^* \right)^{-1} \right)
    = c_6 {\mathbf H}_{\frac{d(d+1)}{2} - \sum_{j=1}^N \frac{1}{H_j}} \left( \mathrm{Im}(G) - A^{\beta} \right) \\
	&= c_6 {\mathbf H}_{\frac{d(d+1)}{2} - \sum_{j=1}^N \frac{1}{H_j}} \left( \mathrm{Im}(G) \right)
	= 0.
\end{align*}
In the above, the first equality follows from Corollary \ref{Coro-trans} and Lemma \ref{Lem-T-Wigner}, the second equality follows from the translation invariance of Hausdorff measure, and the third equality follows from the inequality $\dimh (\mathrm{Im}(G)) < \frac{d(d+1)}{2} - \sum_{j=1}^N \frac{1}{H_j}$ that is a combination of \eqref{eq-dimh Im G} and \eqref{ineq-cond-1}.

On the other hand, for $K \ge 1$ large, we choose the interval
\begin{align*}
    F_K = [-K,K]^l \times F'
    \subseteq \bR^{\frac12[d(d+1) - \sum_{j=1}^r (l_j-1)(l_j+2)]},
\end{align*}
where $F'$ is a compact interval containing the support of $\Gamma$ defined in Lemma \ref{Lem:dim-S}. Then $F_K$ is a compact interval. We write $G|_{F_K}$ as a smooth restriction of $G$ on the neighbourhood of $F_K$ in the sense that $G|_{F_K} = G$ on $F_K$, and $G|_{F_K}=0$ on $F_{K+1}^\complement$.
If
\begin{align*}
    \sum_{j=1}^N \frac{1}{H_j} = \dfrac{1}{2} \sum_{j=1}^r (l_j-1)(l_j+2),
\end{align*}
then Lemma \ref{Lemma-condition-verify} yields
\begin{align*}
    \lambda_{\frac{d(d+1)}{2}} \left( \left( \mathrm{Im}(G|_{F_K}) - A^\beta \right)^{(r)} \right)
    \le C r^{\frac{1}{2} \sum_{j=1}^r (l_j-1)(l_j+2)}
    = C r^{\sum_{j=1}^N \frac{1}{H_j}}
\end{align*}
for all $r>0$ small, and $C$ is a constant only depending on $G,d,H_1,\ldots,H_N$.
By the definition \eqref{def-F function} and a similar argument as \eqref{Eq:Hit-up}, for $K$ large enough, we have
\begin{align*}
    & \bP \bigg( \left\{ \exists t \in I, \ \mathrm{s.t. \ } \left| \lambda_{J_j}^{\beta}(t) \right| = 1 \ \mathrm{for \ some \ disjoint \ } J_j \in [d] \mathrm{\ with \ } |J_j| =  l_j, \forall 1 \le j \le r \right\} \\
    & \quad\quad \quad\quad \quad\quad \cap \{ \spec(Y^\beta) \subseteq [-K,K] \} \bigg) \\
    & \le \bP \left( \left\{ X^{\beta}(t) \in T_\beta^{-1} \left( \mathrm{Im}(G) - A^{\beta} \right) \left( T_\beta^* \right)^{-1} \mathrm{\ for \ some \ } t \in I \right\} \cap \left\{ \spec \left( Y^\beta \right) \subseteq [-K,K] \right\} \right) \\
    & \le \bP \left( X^{\beta}(I) \cap T_\beta^{-1} \left( \mathrm{Im}(G|_{F_K}) - A^{\beta} \right) \left(T_\beta^*\right)^{-1} \not= \emptyset \right)
\end{align*}
Hence, by Lemma \ref{Lemma-Hitting prb} and Corollary \ref{Coro-trans-critical}, we obtain
\begin{align*}
    & \bP \bigg( \left\{ \exists t \in I, \ \mathrm{s.t. \ } \left| \lambda_{J_j}^{\beta}(t) \right| = 1 \ \mathrm{for \ some \ disjoint \ } J_j \in [d] \mathrm{\ with \ } |J_j| =  l_j, \forall 1 \le j \le r \right\} \\
    & \quad\quad \quad\quad \quad\quad \cap \{ \spec(Y^\beta) \subseteq [-K,K] \} \bigg) = 0.
\end{align*}
Letting $K \to \infty$, we have
\begin{align*}
    & \bP \bigg( \exists t \in I, \ \mathrm{s.t. \ } \left| \lambda_{J_j}^{\beta}(t) \right| = 1 \ \mathrm{for \ some \ disjoint \ } J_j \in [d] \mathrm{\ with \ } |J_j| =  l_j, \forall 1 \le j \le r \bigg) = 0.
\end{align*}
The proof of Part (i) of Theorem \ref{Thm-hitting prob-real} is concluded.

Next, we prove Part (ii). We choose $x_0 \in {\mathbf S}(d;l_1, \ldots, l_r)$ satisfying $\spec(x_0) = l$. By Lemma \ref{Lem:dim-bound}, there exists $\delta_0 > 0$, such that ${\mathbf S}(d;l_1, \ldots, l_r)  \cap \mathfrak B_{\delta_0}(x_0)$ is a $(\frac12[d(d+1) - \sum_{j=1}^r (l_j-1)(l_j+2)])$-dimensional manifold. Then, similarly to (\ref{Eq:Hit-up}), by using lower bound on hitting probability in Lemma \ref{Lemma-Xiao-Hitting prob}, we have
\begin{align*}
	& \bP \left( \exists t \in I, \ \mathrm{s.t. \ } \left| \lambda_{J_j}^{\beta}(t) \right| = 1 \ \mathrm{for \ some \ disjoint \ } J_j \in [d] \mathrm{\ with \ } |J_j| =  l_j, \forall 1 \le j \le r \right) \\
	&= \bP \left( X^{\beta}(t) \in T_\beta^{-1} \left( {\mathbf S}(d;l_1, \ldots, l_r) - A^{\beta} \right) \left(T_\beta^*\right)^{-1} \mathrm{\ for \ some \ } t \in I \right) \\
	&\ge \bP \left( X^{\beta}(t) \in T_\beta^{-1} \left( {\mathbf S}(d;l_1, \ldots, l_r) \cap \mathfrak B_{\delta_0}(x_0) - A^{\beta} \right) \left(T_\beta^*\right)^{-1} \mathrm{\ for \ some \ } t \in I \right) \\
	&= \bP \left( X^{\beta}(I) \cap T_\beta^{-1} \left( {\mathbf S}(d;l_1, \ldots, l_r) \cap \mathfrak B_{\delta_0}(x_0) - A^{\beta} \right) \left(T_\beta^*\right)^{-1} \not= \emptyset \right) \\
	&\ge c_5 \cC_{\frac{d(d+1)}{2} - \sum_{j=1}^N \frac{1}{H_j}} \left( T_\beta^{-1} \left( {\mathbf S}(d;l_1, \ldots, l_r) \cap \mathfrak B_{\delta_0}(x_0) - A^{\beta} \right) \left(T_\beta^*\right)^{-1} \right)
	> 0.
\end{align*}
In the above, the last step follows from
\begin{align*}
    & \dimh \left( T_\beta^{-1} \left( {\mathbf S}(d;l_1, \ldots, l_r) \cap \mathfrak B_{\delta_0}(x_0) - A^\beta \right) \left(T_\beta^*\right)^{-1} \right) \\
    & = \dimh \left( {\mathbf S}(d;l_1, \ldots, l_r)  \cap \mathfrak B_{\delta_0}(x_0) \right)
    > \frac{d(d+1)}{2} - \sum_{j=1}^N \frac{1}{H_j},
\end{align*}
due to Corollary \ref{Coro-trans}, Lemma \ref{Lem-T-Wigner} and the translation invariance of Hausdorff measure.

Finally, we prove Part (iii) by applying the second part of Lemma \ref{Lemma-Xiao-Hitting prob}. Notice that
\begin{align*}
    {\cal C}_{l_1,\ldots,l_r}^\beta
    = \left( Y^{\beta} \right)^{-1} \left( {\mathbf S}(d;l_1, \ldots, l_r) \right)
    \subseteq \left( X^{\beta} \right)^{-1} \left( T_\beta^{-1} \left(\mathrm{Im}(G) - A^{\beta}\right) \left(T^*_\beta\right)^{-1} \right),
\end{align*}
where $(Y^{\beta})^{-1}(B) = \{t: \, Y^{\beta}(t) \in B\}$ is the pre-image of $B$ under the mapping $Y^{\beta}$.
We obtain from Lemma \ref{Lemma-Xiao-Hitting prob}, \eqref{eq-dimh Im G}, Corollary \ref{Coro-trans} and Lemma \ref{Lem-T-Wigner} that almost surely,
\begin{align} \label{Eq:dimC1}
    \dimh {\cal C}_{l_1,\ldots,l_r}^\beta
    \le& \dimh \left( \left( X^{\beta} \right)^{-1} \left( T_\beta^{-1} \left(\mathrm{Im}(G) - A^{\beta}\right) \left(T^*_\beta\right)^{-1} \right) \cap I \right) \nonumber \\
    \le& \min_{1 \le \ell \le N}\bigg\{\sum_{j=1}^{\ell } \frac{H_\ell}{H_j} + N-\ell - \dfrac{H_\ell}{2} \sum_{j=1}^r (l_j-1)(l_j+2) \bigg\}.
\end{align}
In addition, for any $x_0 \in {\mathbf S}(d;l_1, \ldots, l_r) $ with with $|\spec(x_0)| = l$, let  $\mathfrak B_{\delta_0}(x_0)$ be the open ball in Lemma \ref{Lem:dim-bound}. Since
\begin{align*}
    {\cal C}_{l_1,\ldots,l_r}^\beta
    &\supseteq \left( Y^{\beta} \right)^{-1} \left( {\mathbf S}(d;l_1, \ldots, l_r) \cap \mathfrak B_{\delta_0}(x_0) \right) \\
    &= \left( X^{\beta} \right)^{-1} \left( T_\beta^{-1} \left( {\mathbf S}(d;l_1, \ldots, l_r) \cap \mathfrak B_{\delta_0}(x_0) - A^{\beta} \right) \left( T_\beta^* \right)^{-1} \right),
\end{align*}
and the Lebesgue measure on ${\mathbf S}(d;l_1, \ldots, l_r) \cap \mathfrak B_{\delta_0}(x_0) - A^{\beta}$ satisfies condition (\ref{Eq:mu0}).
It follows from Lemma \ref{Lemma-Xiao-Hitting prob} that with positive probability,
\begin{align}\label{Eq:dimC2}
    \dimh {\cal C}_{l_1,\ldots,l_r}^\beta
    \ge& \dimh \left( \left( X^{\beta} \right)^{-1} \left( T_\beta^{-1} \left( {\mathbf S}(d;l_1, \ldots, l_r) \cap \mathfrak B_{\delta_0}(x_0) - A^{\beta} \right) \left( T_\beta^*\right)^{-1} \right) \right) \nonumber \\
	\ge& \min_{1 \le \ell \le N}\bigg\{\sum_{j=1}^{\ell } \frac{H_\ell}{H_j} + N- \ell - \dfrac{H_\ell}{2} \sum_{j=1}^r (l_j-1)(l_j+2) \bigg\}.
 \end{align}
Thus, the first equality in (\ref{Eq:dimC}) follows from \eqref{Eq:dimC1} and \eqref{Eq:dimC2}. The second equality in (\ref{Eq:dimC}) is 
elementary and can be verified directly. This finishes the proof of Theorem \ref{Thm-hitting prob-real}.
\end{proof}

\subsection{The complex case: proof of Theorem \ref{Thm-hitting prob-complex}} \label{sec:complex-case}

In this subsection, we consider the matrix \eqref{def-entries} for the case $\beta = 2$. The idea of the proof is similar to that of Section \ref{sec:real-case}, and is sketched below.
The following lemma is the complex version of Lemma \ref{Lem:dim-S}, which indicates that $d^2 - \sum_{j=1}^r (l_j+1)(l_j-1)$ is an upper bound for the dimension of ${\mathbf H}(d;l_1, \ldots, l_r)$.

\begin{lemma} \label{Lemma-Prop 4.6}
Let $\Delta: \bR^l \rightarrow {\mathbf D}(d)$ be the function that maps the vector $u = (u_1, \ldots, u_l) \in \bR^l$ to the diagonal matrix $\Delta(u) = \{\Delta_{i,j}(u):1 \le i, j \le d\}$ given by \eqref{def-Lamda mapping}. Then there exists a compactly supported smooth function $\widehat{\Gamma}: \bR^{d(d-1) - \sum_{j=1}^r l_j(l_j-1)} \rightarrow \mathbf U(d)$ such that the mapping $\widehat{G}: \bR^l \times \bR^{d(d-1) - \sum_{j=1}^r l_j(l_j-1)} \rightarrow {\mathbf H}(d)$ given by
\begin{align} \label{def-F tilde function}
	\widehat{G}(u,v) = \widehat{\Gamma}(v) \Delta(u) \widehat{\Gamma}(v)^*, \quad (u,v) \in 
	\bR^l \times \bR^{d(d-1) - \sum_{j=1}^r l_j(l_j-1)}
\end{align}
satisfies
\begin{align} \label{eq-Lemma-Prop 4.6}
	{\mathbf H}(d;l_1, \ldots, l_r) \subseteq \left\{ x \in \bR^{d^2}: \ x \in \mathrm{Im}(\widehat{G}) \right\}.
\end{align}
\end{lemma}

\begin{proof}
The proof is similar to the proof for the real case, and is sketched below.
We consider the set of matrices
\begin{align*}
	&{\mathbf U}(d;l_1, \ldots, l_r) \\
	&= \left\{ B = \rd (b_1, \ldots, b_{l-r}, P_1, \ldots, P_r) \in {\mathbf U}(d) : b_{i} \in {\mathbf U}(1), P_j \in {\mathbf U}(l_j), \forall 1 \le j \le r, 1 \le i \le l-r \right\}.
\end{align*}
equipped with the matrix multiplication forms a Lie subgroup of ${\mathbf U}(d)$ with dimension $l-r + \sum_{j=1}^r l_j^2 = d + \sum_{j=1}^r l_j(l_j-1)$. Thus, as the real case, we define the smooth and free right action of the Lie group ${\mathbf U}(d;l_1, \ldots, l_r)$ on the manifold ${\mathbf U}(d)$ as matrix multiplication. Moreover, the right action is proper by Lemma \ref{Lemma-proper}. Therefore, by Lemma \ref{Lemma-Quotient} and Lemma \ref{Lemma-open}, the quotient space ${\mathbf U}(d) / {\mathbf U}(d;l_1, \ldots, l_r)$ is a smooth manifold of dimension $d(d-1) - \sum_{j=1}^r l_j(l_j-1)$, and the quotient map $\pi: {\mathbf U}(d) \longrightarrow {\mathbf U}(d) / {\mathbf U}(d;l_1, \ldots, l_r)$ is a smooth submersion and is an open map.

For $\epsilon>0$, let $I_\epsilon$ be the open interval
\begin{align*}
    I_\epsilon = (-\epsilon,\epsilon)^{d(d-1) - \sum_{j=1}^r l_j(l_j-1)}.
\end{align*}
Then for any $A \in {\mathbf U}(d)$, one can define the a smooth diffeomorphism $\phi$ from $I_\epsilon$ to a neighbourhood of $\pi(A)$ such that $\phi(0) = \pi(A)$. Then we can define $\Gamma^{(i)}$ on $I_\epsilon$ as in \eqref{def-mapping Pi}, and $G^{(i)}$ as in \eqref{eq-def-Gi}. By the compactness of ${\mathbf U}(d) / {\mathbf U}(d;l_1, \ldots, l_r)$, we can show that ${\mathbf U}(d;l_1, \ldots, l_r)$ is covered by finite many $\mathrm{Im} (G^{(i)})$. Thus, we can construct $\Gamma$ as the end of the proof of Lemma \ref{Lem:dim-S}.
\end{proof}

The following result is the complex version of Lemma \ref{Lem:approx-S}. The proof is similar to that of Lemma \ref{Lem:approx-S} 
and hence is omitted.

\begin{lemma} \label{Lemma-Prop 4.4}
Let $A \in {\mathbf H}(d;l_1, \ldots, l_r)$ be a Hermitian matrix with spectral decomposition $A = P D P^*$ for some $P \in {\mathbf U}(d)$ and $D \in {\mathbf D}(d)$ such that $|\spec(A)| = l$. Then for every $\epsilon > 0$, there exists $\delta > 0$, such that for all $B \in {\mathbf H}(d;l_1, \ldots, l_r)$ satisfying
\begin{align*}
	\max_{1 \le i, j \le d} |A_{i,j} - B_{i,j}| < \delta,
\end{align*}
we have $|\spec(B)|=l$, and there exists a spectral decomposition of the form $B = Q F Q^*$, 
where $Q \in {\mathbf U}(d)$ and $F \in {\mathbf D}(d)$ satisfy
\begin{align*}
	\max_{1 \le i \le d} |D_{i,i} - F_{i,i}| < \epsilon, \quad \max_{1 \le i, j \le d} |Q_{i,j} - P_{i,j}| < \epsilon.	
\end{align*}
\end{lemma}

The next lemma is the complex version of Lemma \ref{Lem:dim-bound}, which indicates that $d^2 - \sum_{j=1}^r (l_j+1)(l_j-1)$ is an optimal upper bound for the dimension of ${\mathbf H}(d;l_1, \ldots, l_r)$. The proof is similar to the proof of Lemma \ref{Lem:dim-bound}, and hence is omitted.

\begin{lemma} \label{Lemma-Prop 4.8}
For $x_0 \in {\mathbf H}(d;l_1, \ldots, l_r)$ with $|\spec(x_0)| = l$, there exists $\delta_0 > 0$, such that ${\mathbf H}(d;l_1, \ldots, l_r) \cap \mathfrak B_{\delta_0}(x_0)$ is a $(d^2 - \sum_{j=1}^r (l_j+1)(l_j-1))$-dimensional manifold with positive $(d^2 - \sum_{j=1}^r (l_j+1)(l_j-1))$-dimensional Lebesgue measure.
\end{lemma}

Now we are ready to prove Theorem \ref{Thm-hitting prob-complex}. 
 
\begin{proof}[Proof of Theorem \ref{Thm-hitting prob-complex}]
The proof is similar to that of Theorem \ref{Thm-hitting prob-real}, which is sketched below.

We first prove Part (i). By Lemma \ref{Lemma-Prop 4.6}, there exists a smooth map
\begin{align*}
    \widehat{G}: \bR^l \times \bR^{d(d-1) - \sum_{j=1}^r l_j(l_j-1)} \rightarrow {\mathbf H}(d),
\end{align*}
such that ${\mathbf H}(d;l_1, \ldots, l_r) \subseteq \mathrm{Im}(\widehat{G})$. Together with Lemma \ref{Lemma-Prop 4.8}, we have
\begin{align} \label{eq-dimh-hat G}
    \dimh \left( \mathrm{Im} \left( \widehat{G} \right) \right) = d^2 - \sum_{j=1}^r (l_j+1)(l_j-1).
\end{align}

If
\begin{align*}
    \sum_{j=1}^N \dfrac{1}{H_j} < \sum_{j=1}^r (l_j+1)(l_j-1),
\end{align*}
then by Lemma \ref{Lemma-Xiao-Hitting prob}, the translation invariance of Hausdorff measure, \eqref{eq-dimh-hat G}, Corollary \ref{Coro-trans} and Lemma \ref{Lem-T-Wigner}, we derive that
\begin{align*}
	&\bP \left( \exists t \in I, \ \mathrm{s.t. \ } \left| \lambda_{J_j}^{\beta}(t) \right| = 1 \ \mathrm{for \ some \ disjoint \ } J_j \in [d] \mathrm{\ with \ } |J_j| =  l_j, \forall 1 \le j \le r \right) \\
	&= \bP \left( Y^{\beta}(t) \in {\mathbf H}(d;l_1, \ldots, l_r) \mathrm{\ for \ some \ } t \in I \right) \\
	&\le \bP \left( X^{\beta}(I) \cap T_\beta^{-1} \left( \mathrm{Im} \left( \widehat G \right) - A^{\beta} \right) \left( T_\beta^* \right)^{-1} \not= \emptyset \right) \\
	&\le c_6 {\mathbf H}_{d^2-\sum_{j=1}^N \frac{1}{H_j}} \left( T_\beta^{-1} \left( \mathrm{Im} \left( \widehat G \right) - A^{\beta} \right) \left( T_\beta^* \right)^{-1} \right) = 0.
\end{align*}
If 
\begin{align*}
    \sum_{j=1}^N \dfrac{1}{H_j} = \sum_{j=1}^r (l_j+1)(l_j-1),
\end{align*}
then we can compute the Lebesgue measure via Lemma \ref{Lemma-condition-verify}
\begin{align*}
    \lambda_{d^2} \left( \left( \mathrm{Im} \left( \widehat G \big|_{F_K} \right) - A^{\beta} \right)^{(r)} \right)
    \le C r^{\sum_{j=1}^r (l_j+1)(l_j-1)}
    = C r^{\sum_{j=1}^N \frac{1}{H_j}}
\end{align*}
for all $r>0$ small. Here, $F_K$ is a compact interval of $\bR^{d^2 - \sum_{j=1}^r (l_j+1)(l_j-1)}$ given by
\begin{align*}
    F_K = [-K,K]^l \times F',
\end{align*}
where $F'$ is a compact interval containing the support of $\widehat \Gamma$ given by Lemma \ref{Lemma-Prop 4.6}. Then by Lemma \ref{Lemma-Hitting prb} and Corollary \ref{Coro-trans-critical}, for $K$ large, we have
\begin{align*}
    &\bP \bigg( \left\{ \exists t \in I, \ \mathrm{s.t. \ } \left| \lambda_{J_j}^{\beta}(t) \right| = 1 \ \mathrm{for \ some \ disjoint \ } J_j \in [d] \mathrm{\ with \ } |J_j| =  l_j, \forall 1 \le j \le r \right\} \\
    & \quad\quad \quad\quad \quad\quad \cap \left\{ \spec \left( Y^\beta \right) \subseteq [-K,K] \right\}\bigg) \\
    & \le \bP \left( X^\beta(I) \cap T_\beta^{-1} \left( \mathrm{Im} \left( \widehat G\big|_{F_K} \right) - A^{\beta} \right) \left( T_\beta^* \right)^{-1} \not= \emptyset \right) = 0.
\end{align*}
Letting $K \to \infty$, we obtain
\begin{align*}
    \bP \left( \exists t \in I, \ \mathrm{s.t. \ } \left| \lambda_{J_j}^{\beta}(t) \right| = 1 \ \mathrm{for \ some \ disjoint \ } J_j \in [d] \mathrm{\ with \ } |J_j| =  l_j, \forall 1 \le j \le r \right)
    =0.
\end{align*}

Next, we prove (ii). We choose $x_0 \in {\mathbf H}(d;l_1, \ldots, l_r)$ satisfying $\spec(x_0) = l$. By Lemma \ref{Lemma-Prop 4.8}, there exists $\delta_0 > 0$, such that ${\mathbf H}(d;l_1, \ldots, l_r) \cap \mathfrak B_{\delta_0}(x_0)$ is an $(d^2 - \sum_{j=1}^r (l_j+1)(l_j-1))$-dimensional manifold. Thus, by Lemma \ref{Lemma-Xiao-Hitting prob}, when
\begin{align*}
    \sum_{j=1}^N \dfrac{1}{H_j} > \sum_{j=1}^r (l_j+1)(l_j-1),
\end{align*}
we have
\begin{align*}
	&\bP \left( \exists t \in I, \ \mathrm{s.t. \ } \left| \lambda_{J_j}^{\beta}(t) \right| = 1 \ \mathrm{for \ some \ disjoint \ } J_j \in [d] \mathrm{\ with \ } |J_j| =  l_j, \forall 1 \le j \le r \right) \\
	&= \bP \left( X^{\beta}(t) \in T_\beta^{-1} \left( {\mathbf H}(d;l_1, \ldots, l_r) - A^{\beta} \right) \left( T_\beta^* \right)^{-1} \mathrm{\ for \ some \ } t \in I \right) \\
	&\ge \bP \left( X^{\beta}(I) \in T_\beta^{-1} \left( {\mathbf H}(d;l_1, \ldots, l_r) \cap\mathfrak B_{\delta_0}(x_0) - A^{\beta} \right) \left( T_\beta^* \right)^{-1} \not= \emptyset \right) \\
	&\ge c_5 \cC_{d^2-\sum_{j=1}^N \frac{1}{H_j}} \left( T_\beta^{-1} \left( {\mathbf H}(d;l_1, \ldots, l_r) \cap\mathfrak B_{\delta_0}(x_0) - A^{\beta} \right) \left( T_\beta^* \right)^{-1} \right) > 0,
\end{align*}
where we use Corollary \ref{Coro-trans} and Lemma \ref{Lem-T-Wigner}.

The proof of (iii), which is an application of Lemma \ref{Lemma-Xiao-Hitting prob}, is similar to the proof of part (iii) of Theorem \ref{Thm-hitting prob-real} and is omitted. This finishes the proof.
\end{proof}

\section{Multiple collisions of singular value processes}

\subsection{The real case: proof of Theorem \ref{Thm-hitting prob-real-Wishart}} \label{sec:real case-singular value}

Fix $r \in \bN$, $l_1, \ldots, l_r \in \{1, \ldots, d_1\}$ satisfying $\sum_{j=1}^r l_j \le d_1$. For simplicity, we denote $l = d_1 - \sum_{j=1}^r (l_j-1)$ be the number of distinct singular values, and $L_i = d_1 - \sum_{j=i}^r l_j + 1$. Then one can check that $L_1 = l-r+1$. We also denote $l_0 = d_2 - d_1$ be the number of trivial singular values.

We start with an upper bound for the dimension of ${\mathbf M}_{\bR}(d_1 \times d_2; l_1, \ldots, l_r)$ given in \eqref{def-degenerate-Wishart-real}. Lemma \ref{Lem:dim-M_R} below implies that $d_1d_2 - \frac{1}{2} \sum_{j=1}^r (l_j-1)(l_j+2)$ is an upper bound of the dimension of ${\mathbf M}_{\bR}(d_1 \times d_2; l_1, \ldots, l_r)$.

\begin{lemma} \label{Lem:dim-M_R}
Let $\Delta: \bR_{\ge 0}^l \rightarrow {\mathbf D}(d_1 \times d_2)$ be a function that maps each vector $u= (u_1, \ldots, u_l) \in \bR_{\ge 0}^l$ to a $d_1 \times d_2$ diagonal matrix $\Delta(u)$ with entries given by
\begin{align*}
	\Delta_{i,i}(u) =
	\begin{cases}
		u_i, & 1 \le i \le L_1-1; \\
		u_{l-r+j}, & L_j \le i \le L_j+l_j-1.
	\end{cases}
\end{align*}
Then there exist compactly supported smooth functions $\Gamma_L: \bR^{\frac12[d_2(d_2-1) + d_1(d_1-1) - \sum_{j=0}^r l_j(l_j-1)]} \rightarrow \mathbf O(d_1)$ and $\Gamma_R: \bR^{\frac12[d_2(d_2-1) + d_1(d_1-1) - \sum_{j=0}^r l_j(l_j-1)]} \rightarrow \mathbf O(d_2)$, such that the mapping
\begin{align*}
	G: \bR_{\ge 0}^l \times \bR^{\frac12[d_2(d_2-1) + d_1(d_1-1) - \sum_{j=0}^r l_j(l_j-1)]} \rightarrow \bR^{d_1 \times d_2}
\end{align*}
given by
\begin{align} \label{eq-def-G-singular}
	G(u, v) = \Gamma_L(v) \Delta(u) \Gamma_R(v)^*,~
	(u,v) \in \bR_{\ge 0}^l \times \bR^{\frac12[d_2(d_2-1) + d_1(d_1-1) - \sum_{j=0}^r l_j(l_j-1)]},
\end{align}
satisfies
\begin{align*}
	{\mathbf M}_{\bR}(d_1 \times d_2; l_1, \ldots, l_r)
	\subseteq \left\{ x \in \bR^{d_1 \times d_2}: \ x \in \mathrm{Im}(G) \right\},
\end{align*}
where $\im(G)$ is the image of the mapping $G$.
\end{lemma}

\begin{proof}
Recall that $\mathbf O(d_1)$ and $\mathbf O(d_2)$ are Lie groups of dimension $d_1(d_1-1)/2$ and $d_2(d_2-1)/2$ respectively. Thus, by \cite{Lee2013}*{Example 7.3 (k)}, their direct product $\mathbf O(d_1) \times \mathbf O(d_2)$ is again a Lie group with dimension $d_1(d_1-1)/2 + d_2(d_2-1)/2$. Recall the definition \eqref{eq-O} of the Lie group $\mathbf O(d_2;l_1, \ldots, l_r, l_0)$, then for any $g \in \mathbf O(d_2;l_1, \ldots, l_r, l_0)$, one can write $g = \rd (g_{d_1}, g_{l_0})$ with $g_{d_1} \in \mathbf O(d_1;l_1, \ldots, l_r)$ and $g_{l_0} \in \mathbf O(l_0)$. If $l_0 = 0$, then we use the convention that $\mathbf O(d_2;l_1, \ldots, l_r, l_0) = \mathbf O(d_1;l_1, \ldots, l_r)$ and in this case we write $g = g_{d_1}$. Then one can verify that the map
\begin{align*}
    \left( \mathbf O(d_1) \times \mathbf O(d_2) \right) \times \mathbf O(d_2;l_1, \ldots, l_r, l_0) \to \mathbf O(d_1) \times \mathbf O(d_2)
\end{align*}
given by
\begin{align*}
	((p_1,p_2),g) \mapsto (p_1,p_2) \cdot g
	= (p_1 g_{d_1}, p_2 g).
\end{align*}
is the right action of the Lie group $\mathbf O(d_2;l_1, \ldots, l_r, l_0)$ on the manifold $\mathbf O(d_1) \times \mathbf O(d_2)$. One can deduce the smoothness of the right action from the smoothness of the matrix multiplication. Besides, the right action is free since identity matrix is the only multiplication identity element. In addition, by Lemma \ref{Lemma-proper} and the compactness of $\mathbf O(d_2;l_1, \ldots, l_r, l_0)$, the right action is proper. Therefore, by Lemma \ref{Lemma-Quotient}, the orbit space $\left( \mathbf O(d_1) \times \mathbf O(d_2) \right) / \mathbf O(d_2;l_1, \ldots, l_r, l_0)$ is a smooth manifold of dimension $\frac12[d_2(d_2-1) + d_1(d_1-1) - \sum_{j=0}^r l_j(l_j-1)]$, and the quotient map $\pi: \mathbf O(d_1) \times \mathbf O(d_2) \longrightarrow \left( \mathbf O(d_1) \times \mathbf O(d_2) \right) / \mathbf O(d_2;l_1, \ldots, l_r, l_0)$ is a smooth submersion. Furthermore, $\pi$ is an open map by Lemma \ref{Lemma-open}.

For any $p \in \mathbf O(d_1) \times \mathbf O(d_2)$, since $\pi$ is a submersion, by Lemma \ref{Lemma:local section}, there exists a smooth local section $\sigma_{\pi}$ from an open neighbourhood of $\pi(p)$ to $\mathbf O(d_1) \times \mathbf O(d_2)$ that maps $\pi(p)$ to $p$ and that $\pi \circ \sigma_{\pi}$ is identity. Moreover, by the definition of chart (see e.g. \cite{Tu2011}*{Definition 5.1}), there exists an open neighbourhood of $\pi(p)$ in $\left( \mathbf O(d_1) \times \mathbf O(d_2) \right) / \mathbf O(d_2;l_1, \ldots, l_r, l_0)$ that is smoothly diffeomorphic to $I_{\epsilon}$, which is defined by
\begin{align*}
	I_{\epsilon} = (-\epsilon, \epsilon)^{\frac12[d_2(d_2-1) + d_1(d_1-1) - \sum_{j=0}^r l_j(l_j-1)]}.
\end{align*}
That is, there exist $\delta > 0$ and a diffeomorphism
\begin{align*}
	\phi: I_{\epsilon} \to \phi(I_{\epsilon}) \subset \left( \mathbf O(d_1) \times \mathbf O(d_2) \right) / \mathbf O(d_2;l_1, \ldots, l_r, l_0) \cap \pi(\mathfrak B_{\delta}(p))
\end{align*}
satisfying $\phi(0) = \pi(p)$. Then we construct smooth function
\begin{align*}
    &\Gamma = (\Gamma_L,\Gamma_R): \bR^{\frac12[d_2(d_2-1) + d_1(d_1-1) - \sum_{j=0}^r l_j(l_j-1)]} \rightarrow {\mathbf O}(d_1) \times {\mathbf O}(d_2)
\end{align*}
that is supported on $I_{2\epsilon}$ and satisfies
\begin{align*}
	(\Gamma_L(v), \Gamma_R(v)) = \Gamma(v)
	= \sigma_{\pi} \circ \phi(v), \forall v \in I_{\epsilon}.
\end{align*}

Since the Lie groups $\mathbf O(d_1)$ and $\mathbf O(d_2)$ are compact, by Tychonoff's theorem, the product space $\mathbf O(d_1) \times \mathbf O(d_2)$ is also compact, which yields the compactness of the quotient space $\left( \mathbf O(d_1) \times \mathbf O(d_2) \right) / \mathbf O(d_2;l_1, \ldots, l_r, l_0)$.
As $\phi$ is a diffeomorphism, $\phi(I_{\epsilon})$ is an open subset of the quotient space $\left( \mathbf O(d_1) \times \mathbf O(d_2) \right) / \mathbf O(d_2;l_1, \ldots, l_r, l_0)$ containing $\pi(p)$. Hence, when $p$ runs over $\mathbf O(d_1) \times \mathbf O(d_2)$, the collection of the sets $\phi(I_{\epsilon})$ forms an open cover of $\left( \mathbf O(d_1) \times \mathbf O(d_2) \right) / \mathbf O(d_2;l_1, \ldots, l_r, l_0)$. By the compactness, there exists a finite subcover $\left\{ \phi^{(i)}(I_{\epsilon_i}): i = 1, \ldots, M \right\}$ for some $M \in \bN$.

For $1 \le i \le M$, we define mappings $G^{(i)}: \bR_{\ge 0}^l \times I_{\epsilon_i} \rightarrow \bR^{d_1 \times d_2}$ by
\begin{align*}
	G^{(i)}(u,v) = \Gamma_L^{(i)}(v) \Delta(u) \Gamma_R^{(i)}(v)^*,
\end{align*}
for $u \in \bR_{\ge 0}^l$, $v \in I_{\epsilon_i}$.

For an arbitrary fixed $x \in {\mathbf M}_{\bR}(d_1 \times d_2; l_1, \ldots, l_r)$, we have the singular value decomposition $x = P_L D P_R^*$ for some $P_L \in {\mathbf O}(d_1)$, $P_R \in {\mathbf O}(d_2)$ and $D \in {\mathbf D}(d_1 \times d_2)$.  We assume that $D_{L_j, L_j} = \cdots = D_{L_j+l_j-1, L_j+l_j-1}$ for $1 \le j \le r$, by rearranging the diagonal of $D$, the columns of $P_L$ and $P_R$ if necessary. There exists $i_0 \in\{1, \dots,  M\}$ such that $\pi(P_L, P_R) \in \phi^{(i_0)} \left( I_{\epsilon_{i_0}} \right)$, and hence one can find $v \in I_{\epsilon_{i_0}}$ such that $\pi(P_L,P_R) = \phi^{(i_0)}(v)$. Thus, $\left( \Gamma_L^{(i_0)}(v), \Gamma_R^{(i_0)}(v) \right) = \sigma_{\pi} \circ \pi (P_L, P_R)$. Note that $\sigma_{\pi} \circ \pi (P_L, P_R)$ and $(P_L,P_R)$ are in the same orbit space under the right action of the Lie group ${\mathbf O}(d_1;l_1, \ldots, l_r,l_0)$, we have $\sigma_{\pi} \circ \pi (P_L,P_R) = (P_L,P_R) \cdot g = (P_Lg_{d_1}, P_Rg)$, where $g = \rd(g_{d_1}, g_{l_0})$ is an element in the Lie group ${\mathbf O}(d_1;l_1, \ldots, l_r, l_0)$. Let $u = (D_{1,1}, D_{2,2}, \ldots, D_{L_1-1,L_1-1}, D_{L_1, L_1}, D_{L_2,L_2}, \ldots, D_{L_r,L_r}) \in \bR_{\ge 0}^l$ then $D = \Delta(u)$. Thus, $x$ has the decomposition
\begin{align*}
	x
	= P_LDP_R^*
	= P_L \left( g_{d_1} D g^* \right) P_R^*
	= (P_L \cdot g_{d_1}) D (P_R \cdot g)^*
	= \Gamma_L^{(i_0)}(v) \Delta(u) \Gamma_R^{(i_0)}(v)^*
	= G^{(i_0)} (u,v).
\end{align*}
Since $x\in {\mathbf M}_{\bR}(d_1 \times d_2; l_1, \ldots, l_r)$ is arbitrarily chosen, we conclude that
\begin{align*}
	{\mathbf M}_{\bR}(d_1 \times d_2; l_1, \ldots, l_r)
	\subseteq \left\{ x \in \bR^{d_1 \times d_2}: x \in \bigcup_{i=1}^M \mathrm{Im}(G^{(i)}) \right\}.
\end{align*}

Finally, let $\epsilon = \max \left\{ \epsilon_1, \ldots, \epsilon_M \right\}$, then we can construct smooth functions $\Gamma_L$ and $\Gamma_R$ with support $I_{4M\epsilon}$ as what we do at the end of the proof of Lemma \ref{Lem:dim-S}.
\end{proof}

The following lemma is about the continuity of the singular value decomposition on ${\mathbf M}_{\bR}(d_1 \times d_2; l_1, \ldots, l_r)$.

\begin{lemma} \label{Lem:approx-singular}
Let $A \in {\mathbf M}_{\bR}(d_1 \times d_2; l_1, \ldots, l_r)$ with singular value decomposition $A = P_L D P_R^*$ for some $P_L \in {\mathbf O}(d_1)$, $P_R \in {\mathbf O}(d_2)$ and $D \in {\mathbf D}(d_1 \times d_2)$, such that $0 \notin \spec(AA^*)$ and $|\spec(AA^*)| = l$. Then for any $\epsilon > 0$, there exists $\delta > 0$, such that for all $B \in {\mathbf M}_{\bR}(d_1 \times d_2; l_1, \ldots, l_r)$ satisfying
\begin{align*}
	\max_{1 \le i \le d_1, 1 \le j \le d_2} |A_{i,j} - B_{i,j}| < \delta,
\end{align*}
we have $0 \notin \spec(BB^*)$ and $|\spec(BB^*)|=l$. Moreover, there exists a singular value decomposition  $B = Q_L FQ_R^*$, where $Q_L \in {\mathbf O}(d_1)$, $Q_R \in {\mathbf O}(d_2)$ and $F \in {\mathbf D}(d_1 \times d_2)$ satisfy
\begin{align} \label{eq-continuous of svd}
	\max_{1 \le i \le d_1} |D_{i,i} - F_{i,i}| < \epsilon, \quad
	\max_{1 \le i, j \le d_1} \left| \left( Q_L \right)_{i,j} - \left( P_L \right)_{i,j} \right| < \epsilon, \quad
	\max_{1 \le i, j \le d_2} \left| \left( Q_R \right)_{i,j} - \left( P_R \right)_{i,j} \right| < \epsilon.
\end{align}
\end{lemma}

\begin{proof}
Noting that $A\in {\mathbf M}_{\bR}(d_1 \times d_2; l_1, \ldots, l_r)$ and $|\spec(AA^*)| = l$, without loss of generality we  assume 
\begin{align} \label{eq-order of D}
	0 < D_{1,1} < \cdots < D_{L_1,L_1} = \cdots = D_{L_2-1,L_2-1} < D_{L_2,L_2} = \cdots = D_{L_3-1,L_3-1} < \cdots < D_{L_r,L_r} = \cdots = D_{d,d}.
\end{align}

As $0 \notin \spec(AA^*)$, we have $A^*A \in {\mathbf S}(d_2;l_1, \ldots, l_r, l_0)$ with decomposition $A^*A = P_RD^*DP_R^*$. Moreover, for $B \in {\mathbf M}_{\bR}(d_1 \times d_2; l_1, \ldots, l_r)$ satisfying
\begin{align*}
	\max_{1 \le i \le d_1, 1 \le j \le d_2} |A_{i,j} - B_{i,j}| < \delta,
\end{align*}
we have $B^*B \in {\mathbf S}(d_2;l_1, \ldots, l_r, l_0)$ and
\begin{align*}
	\max_{1 \le i, j \le d_2} \left| (A^*A)_{i,j} - (B^*B)_{i,j} \right|
	&= \max_{1 \le i, j \le d_2} \left| \sum_{k=1}^{d_1} A_{i,k}^* A_{k,j} - \sum_{k=1}^{d_1} B_{i,k}^* B_{k,j} \right| \\
	&= \max_{1 \le i, j \le d_2} \left| \sum_{k=1}^{d_1} (A_{i,k}^* - B_{i,k}^*) A_{k,j} + \sum_{k=1}^{d_1} B_{i,k}^* (A_{k,j}- B_{k,j}) \right| \\
	&< \delta \left( d_1 \max_{1 \le i \le d_1, 1 \le j \le d_2} \left| A_{i,j} \right| + d_1 \max_{1 \le i \le d_1, 1 \le j \le d_2} \left| B_{i,j} \right| \right) \\
	&< \delta \left( 2d_1 \max_{1 \le i \le d_1, 1 \le j \le d_2} \left| A_{i,j} \right| + d_1 \delta \right).
\end{align*}
Thus, by Lemma \ref{Lem:approx-S}, for any $\epsilon > 0$, we can choose $\delta$ small enough, so that $|\spec(B^*B)|=l+1$ if $l_0>0$ and $|\spec(B^*B)|=l$ if $l_0 = 0$, and there exists a spectral decomposition $B^*B = Q_R F_R Q_R^*$, where $Q_R \in \mathbf O(d_2)$ and $F_R \in \mathbf D(d_2)$ satisfy
\begin{align} \label{eq-approximate left matrix}
	\max_{1 \le i \le d_2} \left| \left( D^*D \right)_{i,i} - \left( F_R \right)_{i,i} \right| < \epsilon, \quad
	\max_{1 \le i, j \le d_2} \left| \left( Q_R \right)_{i,j} - \left( P_R \right)_{i,j} \right| < \epsilon.
\end{align}
We choose $\epsilon$ to be small enough, then by \eqref{eq-order of D}, we have
\begin{align*}
    0<(F_R)_{1,1} < \cdots < (F_R)_{L_1,L_1} < (F_R)_{L_2,L_2} < \cdots < (F_R)_{L_r,L_r}, \ \left| (F_R)_{d_1+1,d_1+1} \right|, \ldots, \left| (F_R)_{d_2,d_2} \right| < \epsilon.
\end{align*}
Since $B^*B \in {\mathbf S}(d_2;l_1, \ldots, l_r, l_0)$, we have $(F_R)_{i,i} = (F_R)_{L_j,L_j}$ for $L_j \le i \le L_j + l_j-1$, $1 \le j \le r$, and $(F_R)_{i,i} = 0$ for $d_1+1 \le i \le d_2$.

Now consider the singular value decomposition $B = Q_L' F Q_R'^*$ for some $Q_L' \in {\mathbf O}(d_1)$, $Q_R' \in {\mathbf O}(d_2)$ and $F \in {\mathbf D}(d_1 \times d_2)$. By rearranging the diagonal of $F$, the columns of $Q_L'$ and $Q_R'$ if necessary, we can assume that the diagonal entries of $F$ are in a non-decreasing order. Then we have the spectral decomposition $B^*B = Q_R' F^*F Q_R'^*$. Note that both $F^*F$ and $F_R$ are diagonal matrices of eigenvalues of $B^*B$ whose last $d_2-d_1$ diagonal entries are zero and whose first $d_1$ entries are in a non-decreasing order, we have $F^*F = F_R$. Hence, we have
\begin{align} \label{eq-order of F}
	0<F_{1,1} < \cdots < F_{L_1,L_1} = \cdots = F_{L_2-1,L_2-1} < F_{L_2,L_2} = \cdots = F_{L_3-1,L_3-1} < \cdots < F_{L_r,L_r} = \cdots = F_{d_1,d_1}.
\end{align}
For $1 \le i \le d_2$, the $i$th columns of $Q_R$ and $Q_R'$ are the eigenvectors of $B^*B$ associated with the eigenvalue $\left( F_R \right)_{i,i}$. Therefore, there exists $g \in \mathbf O(d_2;l_1, \ldots, l_r, l_0)$, such that $Q_R' = Q_R \cdot g$. Moreover, the $d_2 \times d_2$ matrix $g$ has the form $g = \rd \left( g_{d_1}, g_{l_0} \right)$, where $g_{d_1} \in \mathbf O(d_1;l_1, \ldots, l_r)$ and $g_{l_0} \in \mathbf O(l_0)$. One can check that $g_{d_1} F g^* = F$. Let $Q_L = Q_L' g_{d_1}^*$ then $Q_L \in \mathbf O(d_1)$. Therefore, we have
\begin{align} \label{eq-4.5-decomposition}
	B = Q_L' F Q_R'^*
	= Q_L' \left( g_{d_1}^* g_{d_1} \right) F \left( g^*g \right) Q_R'^*
	= \left( Q_L' g_{d_1}^* \right) \left( g_{d_1} F g^* \right) \left( Q_R' g^* \right)^*
	= Q_L F Q_R^*.
\end{align}

It remains to verify \eqref{eq-continuous of svd} for the singular value decomposition \eqref{eq-4.5-decomposition}. The last inequality in \eqref{eq-continuous of svd} follows directly from \eqref{eq-approximate left matrix}.

Let $K$ be a large number such that
\begin{align*}
	\max_{i \in [d_1], j \in [d_2]} |A_{i,j}| < K, \
	D_{1,1} > 2/K.
\end{align*}
By triangle inequality, one can easily deduce that $\left| B_{i,j} \right| < K + \delta$ for all $i \in [d_1], j \in [d_2]$. Without loss of generality, we may assume that $\epsilon \ll 1/K^2$, then by \eqref{eq-approximate left matrix}, we have
\begin{align*}
	F_{i,i}^2
	\ge D_{i,i}^2 - \left| F_{i,i}^2 - D_{i,i}^2 \right|
	\ge D_{1,1}^2 - \epsilon
	> 1/K^2.
\end{align*}
Hence, by \eqref{eq-approximate left matrix}, we have
\begin{align} \label{eq-approximate of D}
	\max_{1 \le i \le d_1} |D_{i,i} - F_{i,i}|
	= \max_{1 \le i \le d_1} \dfrac{\left| D_{i,i}^2 - F_{i,i}^2 \right|}{D_{i,i} + F_{i,i}}
	< \dfrac{K}{3} \max_{1 \le i \le d_1} \left| D_{i,i}^2 - F_{i,i}^2 \right|
	< \dfrac{K\epsilon}{3}.
\end{align}

Note that $P_R$ and $Q_R$ are orthogonal matrices, by triangle inequality and \eqref{eq-approximate left matrix}, we have 
\begin{align} \label{eq-4.8}
	& \max_{i \in [d_1], j \in [d_2]} \left| \left(Q_L F\right)_{i,j} - \left(P_L D\right)_{i,j} \right|
	= \max_{i \in [d_1], j \in [d_2]} \left| \left( B Q_R\right)_{i,j} - \left( A P_R\right)_{i,j} \right| \nonumber \\
	&= \max_{i \in [d_1], j \in [d_2]} \left| \sum_{k=1}^{d_2} B_{i,k} \left(Q_R\right)_{k,j} - \sum_{k=1}^{d_2} A_{i,k} \left(P_R\right)_{k,j} \right| \nonumber \\
	&= \max_{i \in [d_1], j \in [d_2]} \left| \sum_{k=1}^{d_2} B_{i,k} \left( \left(Q_R\right)_{k,j} - \left(P_R\right)_{k,j} \right) + \sum_{k=1}^{d_2} \left( B_{i,k} - A_{i,k} \right) \left(P_R\right)_{k,j} \right| \nonumber \\
	&< d_2 \epsilon (K + \delta) + d_2 \delta.
\end{align}
Note that the matrices $D,F \in {\mathbf D}(d_1 \times d_2)$, we can write $D = \left( D_{d_1}, 0_{d_1 \times l_0} \right)$ and $F = \left( F_{d_1}, 0_{d_1 \times l_0} \right)$, where $D_{d_1}, F_{d_1} \in {\mathbf D}(d_1)$, and $0_{d_1 \times l_0} \in \bR^{d_1 \times l_0}$ is a zero matrix. Then we have $P_LD = \left( P_LD_{d_1}, 0_{d_1 \times l_0} \right)$ and $Q_LF = \left( Q_LF_{d_1}, 0_{d_1 \times l_0} \right)$. As $D_{d_1}$ and $F_{d_1}$ are diagonal square matrices with strictly positive diagonal entries, they are invertible. Hence, by triangle inequality, \eqref{eq-approximate of D} and \eqref{eq-4.8},
\begin{align} \label{eq-approxiamte L-4.8}
	& \max_{i,j \in [d_1]} \left| \left(Q_L\right)_{i,j} - \left(P_L\right)_{i,j} \right|
	= \max_{i,j \in [d_1]} \left| \left(Q_LF_{d_1} F_{d_1}^{-1}\right)_{i,j} - \left(P_LD_{d_1} D_{d_1}^{-1}\right)_{i,j} \right| \nonumber \\
	&\le \max_{i,j \in [d_1]} \left| \left(Q_LF_{d_1}\right)_{i,j} \left( F_{d_1}^{-1} \right)_{j,j} - \left(P_LD_{d_1}\right)_{i,j} \left( D_{d_1}^{-1} \right)_{j,j} \right| \nonumber \\
	&\le \max_{i,j \in [d_1]} \left| \left( \left(Q_LF_{d_1}\right)_{i,j} - \left(P_LD_{d_1}\right)_{i,j} \right) \left( F_{d_1}^{-1} \right)_{j,j} + \left(P_LD_{d_1}\right)_{i,j} \left( \left( F_{d_1}^{-1} \right)_{j,j} - \left( D_{d_1}^{-1} \right)_{j,j} \right) \right| \nonumber \\
	&\le \max_{i,j \in [d_1]} \left| \left( \left(Q_LF_{d_1}\right)_{i,j} - \left(P_LD_{d_1}\right)_{i,j} \right) \left( F_{d_1}^{-1} \right)_{j,j} + \left(P_L\right)_{i,j} \left( \left( D_{d_1} \right)_{j,j} - \left( F_{d_1} \right)_{j,j} \right) \left( F_{d_1}^{-1} \right)_{j,j} \right| \nonumber \\
	&< \left( d_2 \epsilon (K + \delta) + d_2 \delta \right) K + \dfrac{K^2 \epsilon}{3}.
\end{align}
Therefore, the proof is concluded from \eqref{eq-approximate left matrix}, \eqref{eq-approximate of D} and \eqref{eq-approxiamte L-4.8}.
\end{proof}

The next lemma implies that the upper bounded for ${\mathbf M}_{\bR}(d_1 \times d_2; l_1, \ldots, l_r)$ given by Lemma \ref{Lem:dim-M_R} is optimal.

\begin{lemma} \label{Lem:dim-Wishart-bound}
For  $x_0 \in {\mathbf M}_{\bR}(d_1 \times d_2; l_1, \ldots, l_r)$ with $0 \notin \spec(x_0x_0^*)$ and $|\spec(x_0x_0^*)| = l$, there exists $\delta_0 > 0$ such that ${\mathbf M}_{\bR}(d_1 \times d_2; l_1, \ldots, l_r) \cap \mathfrak B_{\delta_0}(x_0)$ is a $(d_1d_2 - \frac{1}{2} \sum_{j=1}^r (l_j-1)(l_j+2))$-dimensional manifold.
In particular, ${\mathbf M}_{\bR}(d_1 \times d_2; l_1, \ldots, l_r) \cap \mathfrak B_{\delta_0}(x_0)$ has positive $(d_1d_2 - \frac{1}{2} \sum_{j=1}^r (l_j-1)(l_j+2))$-dimensional Lebesgue measure.
\end{lemma}

\begin{proof}
By Lemma \ref{Lem:dim-M_R}, there exist $u_0 \in \bR_+^l$, $v_0 \in \bR^{\frac12[d_2(d_2-1) + d_1(d_1-1) - \sum_{j=0}^r l_j(l_j-1)]}$, such that $x_0 = G(u_0,v_0)$, where the function $G$ is given by \eqref{eq-def-G-singular}. Moreover, by the construction, without loss of generality, we may assume that $x_0 = G^{(1)} (u_0,v_0) = \Gamma_L^{(1)}(v_0) \Delta(u_0) \Gamma_R^{(1)}(v_0)$, with $v_0 \in I_{\epsilon_1}$.

To show that the manifold ${\mathbf M}_{\bR}(d_1 \times d_2; l_1, \ldots, l_r) \cap \mathfrak B_{\delta_0}(x_0)$ has dimension $d_1d_2 - \frac{1}{2} \sum_{j=1}^r (l_j-1)(l_j+2)$ and has positive $(d_1d_2 - \frac{1}{2} \sum_{j=1}^r (l_j-1)(l_j+2))$-dimensional Lebesgue measure, it is sufficient to show that there exist open sets $U \subseteq \bR_+^l \times I_{\epsilon_1}$ and $V \subseteq {\mathbf M}_{\bR}(d_1 \times d_2; l_1, \ldots, l_r) \cap \mathfrak B_{\delta_0}(x_0)$, such that the map $\left. G^{(1)} \right|_U: U \rightarrow V$ is a homeomorphism.

Let 
\begin{align*}
	r_0 = \dfrac{1}{2} \min_{\substack{\mu , \lambda\in \spec(x_0x_0^*)\\ \mu \not= \lambda}} |\sqrt{\mu} - \sqrt{\lambda}|.
\end{align*}
For $\gamma_0 \in (0, r_0)$ that is small enough, by Lemma \ref{Lem:approx-singular}, there exists $\delta_0$ such that for each $x \in {\mathbf M}_{\bR}(d_1 \times d_2; l_1, \ldots, l_r) \cap \mathfrak B_{\delta_0}(x_0)$,  it has the singular value decomposition $x = Q_L E Q_R^*$ with $Q_L \in {\mathbf O}(d_1) \cap \mathfrak B_{\gamma_0/2} \left( \Gamma_L^{(1)}(v_0) \right)$, $Q_R \in {\mathbf O}(d_2) \cap \mathfrak B_{\gamma_0/2} \left( \Gamma_R^{(1)}(v_0) \right)$ and $E \in {\mathbf D}(d) \cap \mathfrak B_{\gamma_0}(D)$, such that $|\spec(xx^*)|=l$ and $0 \notin \spec(xx^*)$. Then by the definition of $r_0$, we have
\begin{align*}
    E_{1,1} < E_{2,2} < \cdots < E_{L_1-1,L_1-1} < E_{L_1,L_1} = \cdots = E_{L_2-1,L_2-1} < E_{L_2,L_2} = \cdots < E_{L_r,L_r} = \cdots = E_{d,d}.
\end{align*}
Denote $u = (E_{1,1}, \dots, E_{L_1-1,L_1-1}, E_{L_1,L_1}, E_{L_2,L_2}, \ldots, E_{L_r,L_r}) \in \bR_+^l$. 
Recall that $\pi$ is continuous and $\phi^{(1)}$ is a diffeomorphism. Since $(Q_L, Q_R) \in \left( {\mathbf O}(d_1) \times {\mathbf O}(d_2) \right) \cap \mathfrak B_{\gamma_0} \left( \left( \Gamma_L^{(1)}(v_0), \Gamma_R^{(1)}(v_0) \right) \right)$ with sufficiently small $\gamma_0$, and $\pi \left( \Gamma_L^{(1)} (v_0), \Gamma_R^{(1)} (v_0) \right) = \phi^{(1)} (v_0)$, we have $\pi(Q_L, Q_R) \in \phi^{(1)} \left( I_{\left( \epsilon_1 + \|v_0\| \right)/2} \right)$. Thus, we choose $v = \left( \phi^{(1)} \right)^{-1} (\pi(Q_L,Q_R)) \in I_{(\epsilon_1 + \|v_0\|)/2}$, then $\left( \Gamma_L^{(1)}(v), \Gamma_R^{(1)}(v) \right) = \sigma_{\pi}^{(1)} \left( \pi(Q_L, Q_R) \right)$ is in the same orbit space as $(Q_L,Q_R)$ with respect to the right action of the Lie group $O(d_2;l_1, \ldots, l_r, l_0)$. Hence, there exists $g = \rd(g_{d_1}, g_{l_0}) \in O(d_2;l_1, \ldots, l_r,l_0)$ such that $\sigma_{\pi}^{(1)} \circ \pi (Q_L, Q_R) = (Q_L, Q_R) \cdot g = (Q_L \cdot g_{d_1}, Q_R \cdot g)$. Therefore,
\begin{align*}
	x = Q_LEQ_R^*
	= Q_L (g_{d_1} E g^*) Q_R^*
	= (Q_L \cdot g_{d_1}) E (Q_R \cdot g)^*
	= \Gamma_L^{(1)}(v) \Delta (u) \Gamma_R^{(1)}(v)
	= G^{(1)}(u,v).
\end{align*}

Now we choose $V = {\mathbf M}_{\bR}(d_1 \times d_2; l_1, \ldots, l_r) \cap \mathfrak B_{\delta_0}(x_0)$, then $V \subset G^{(1)} \left( \bR_+^l \times I_{\epsilon_1} \right)$. We choose $U = \left( G^{(1)} \right)^{-1}(V)$, then $\left. G^{(1)} \right|_U$ is surjective. The continuity of the mapping $G^{(1)}$ implies that $U$ is open in $\bR^l \times I_{\epsilon_1}$. To show $\left. G^{(1)} \right|_U$ is a homeomorphism, it suffices to show the following conditions are satisfied:
\begin{enumerate}
	\item[(a1)] $\left. G^{(1)} \right|_U$ is injective;
	\item[(a2)] $\left( \left. G^{(1)} \right|_U \right)^{-1}$ is continuous over $V$.
\end{enumerate}

To verify (a1), we assume that $x \in {\mathbf M}_{\bR}(d_1 \times d_2; l_1, \ldots, l_r) \cap \mathfrak B_{\delta_0}(x_0)$ has the following singular value decompositions,
\begin{align*}
	x = \Gamma_L^{(1)}(v) \Delta(u) \Gamma_R^{(1)}(v)^*
	= \Gamma_L^{(1)}(v') \Delta(u') \Gamma_R^{(1)}(v')^*.
\end{align*}
If $u \not= u'$, we may assume that the $i$th entries of $u$ and $u'$ are different, then the $i$th columns of $\Gamma_L^{(1)}(v)$ and $\Gamma_L^{(1)}(v')$ belong to different eigenspaces of $xx^*$. Following the proof of Lemma \ref{Lem:dim-bound}, we will get a contradiction, which implies that $u = u'$. We can also deduce that $v=v'$ as in the proof of Lemma \ref{Lem:dim-bound}.

To verify (a2),  Consider any sequence $\{x_n\}_{n \in \bN} \subseteq V$, such that $\lim_{n \rightarrow \infty} x_n = x \in V$. Let $(u_n, v_n) = \left( \left. G^{(1)} \right|_U \right)^{-1}(x_n) \in \bR_+^l \times I_{\epsilon_1}$ and $(u, v) = \left( \left. G^{(1)} \right|_U \right)^{-1}(x) \in \bR_+^l \times I_{\epsilon_1}$, then
\begin{align*}
	x_n = \Gamma_L^{(1)}(v_n) \Delta(u_n) \Gamma_R^{(1)}(v_n)^*, \quad
	x = \Gamma_L^{(1)}(v) \Delta (u) \Gamma_R^{(1)}(v)^*.
\end{align*}
As in the proof of Lemma \ref{Lem:dim-bound}, we can deduce that
\begin{align*}
	\lim_{n \rightarrow \infty} u_n = u.
\end{align*}
Note that $v_n \in \overline{I_{(\epsilon_1+\|v_0\|)/2}}$, which is compact. Let $v' \in \overline{I_{(\epsilon_1+\|v_0\|)/2}} \subset I_{\epsilon_1}$ be a limit point of the sequence $\{v_n\}_{n \in \bN}$, then there exists a subsequence $\{v_{m_n}\}_{n \in \bN}$, such that $\lim_{n \rightarrow \infty} v_{m_n} = v'$. By the continuity of the mappings $\Gamma_L^{(1)}$, $\Gamma_R^{(1)}$ and $\Delta$, we have
\begin{align*}
	x = \lim_{n \rightarrow \infty} x_{m_n}
	= \lim_{n \rightarrow \infty} \Gamma_L^{(1)}(v_{m_n}) \Delta(u_{m_n}) \Gamma_R^{(1)}(v_{m_n})^*
	= \Gamma_L^{(1)}(v') \Delta(u) \Gamma_R^{(1)}(v')^*.
\end{align*}
Thus, the columns of $\Gamma_L^{(1)}(v)$ and $\Gamma_L^{(1)}(v')$ are the eigenvectors of the same eigenvalues of the matrix $xx^*$, which implies that there exists $g_{d_1} \in \mathbf O(d_1;l_1, \ldots, l_r)$, such that $\Gamma_L^{(1)}(v') = \Gamma_L^{(1)}(v) g_{d_1}$. Similarly, by considering $x^*x$, one can deduce that there exists $g \in \mathbf O(d_2;l_1, \ldots, l_r, l_0)$, such that $\Gamma_R^{(1)}(v') = \Gamma_R^{(1)}(v) g$. Hence, we have
\begin{align*}
	\Delta(u) = \Gamma_L^{(1)}(v')^* x \Gamma_R^{(1)}(v')
	= g_{d_1}^* \Gamma_L^{(1)}(v)^* x \Gamma_R^{(1)}(v) g
	= g_{d_1}^* \Delta(u) g.
\end{align*}
By the definition of $\Delta$ and the assumption that $0 \notin \spec(x_0x_0^*)$, if we write $g = \rd(g_{d_1}', g_{l_0})$, then we have $g_{d_1}' = g_{d_1}$, which implies that $\left( \Gamma_L^{(1)}(v'), \Gamma_R^{(1)}(v') \right)$ and $\left( \Gamma_L^{(1)}(v), \Gamma_R^{(1)}(v) \right)$ are in the same orbit space of $\pi$, i.e. $\pi \left( \Gamma_L^{(1)}(v), \Gamma_R^{(1)}(v) \right) = \pi \left( \Gamma_L^{(1)}(v'), \Gamma_R^{(1)}(v') \right)$. Recall that $\pi \circ \sigma_{\pi}^{(1)}$ is identity and $\phi^{(1)}$ is diffeomorphism on $I_{\epsilon_1}$, we have $v=v'$, which implies that
\begin{equation*}
	\lim_{n\to\infty}  v_n =v,
\end{equation*}
since any limit point $v'$ has to be $v$. Thus, we obtain the continuity of the map $\left(G^{(1)}\right)^{-1}$ and (a2) is proved. The proof is concluded.
\end{proof}

Now we are ready to prove Theorem \ref{Thm-hitting prob-real-Wishart}.
\begin{proof}[Proof of Theorem \ref{Thm-hitting prob-real-Wishart}]
The proof is similar to that of Theorem \ref{Thm-hitting prob-real}, which is sketched below.

We first prove Part (i). By Lemma \ref{Lem:dim-M_R}, there exists a smooth map
\begin{align*}
    G: \bR_{\ge 0}^l \times \bR^{\frac12[d_2(d_2-1) + d_1(d_1-1) - \sum_{j=0}^r l_j(l_j-1)]} \rightarrow \bR^{d_1 \times d_2},
\end{align*}
such that ${\mathbf M}_{\bR}(d_1 \times d_2; l_1, \ldots, l_r) \subseteq \mathrm{Im}(G)$. In view of Lemma \ref{Lem:dim-Wishart-bound}, we obtain
\begin{align} \label{eq-dimh-G-SVD}
    \dimh \left( \mathrm{Im}(G) \right) = d_1d_2 - \frac{1}{2} \sum_{j=1}^r (l_j-1)(l_j+2).
\end{align}

If
\begin{align*}
    \sum_{j=1}^N \frac{1}{H_j} < \frac{1}{2} \sum_{j=1}^r (l_j-1)(l_j+2),
\end{align*}
then by Lemma \ref{Lemma-Xiao-Hitting prob}, the translation invariance of Hausdorff measure, \eqref{eq-dimh-G-SVD}, Corollary \ref{Coro-trans} and Lemma \ref{Lem-T-Wigner}, we have
\begin{align*}
	&\bP \left( \exists t \in I, \ \mathrm{s.t. \ } \left| s_{J_j}^{\beta}(t) \right| = 1 \ \mathrm{for \ some \ disjoint \ } J_j \in [d_1] \mathrm{\ with \ } |J_j| =  l_j, \forall 1 \le j \le r \right) \\
	&= \bP \left( Z^{\beta}(t) \in {\mathbf M}_{\bR}(d_1 \times d_2; l_1, \ldots, l_r) \mathrm{\ for \ some \ } t \in I \right) \\
	&\le \bP \left( W^{\beta}(I) \cap T_\beta^{-1} \left( \mathrm{Im}(G) - B^{\beta} \right) \tilde T_\beta^{-1} \not= \emptyset \right) \\
	&\le c_6 {\mathbf H}_{d_1d_2-\sum_{j=1}^N \frac1{H_j}} \left( T_\beta^{-1} \left( \mathrm{Im}(G) - B^{\beta} \right) \tilde T_\beta^{-1} \right) = 0.
\end{align*}
The remaining case is
\begin{align*}
    \sum_{j=1}^N \frac{1}{H_j} = \frac{1}{2} \sum_{j=1}^r (l_j-1)(l_j+2).
\end{align*}
To handle this case, for $K\ge 1$, we choose a compact interval $F_K = [0,K]^l \times F'$, where $F'$ is a compact interval of $\bR^{\frac12[d_2(d_2-1) + d_1(d_1-1) - \sum_{j=0}^r l_j(l_j-1)]}$ containing the support of $\Gamma_L$ and $\Gamma_R$ given in Lemma \ref{Lem:dim-M_R}.
Then by Lemma \ref{Lemma-condition-verify}, we have
\begin{align*}
    \lambda_{d_1d_2} \left( \left( \mathrm{Im}(G|_{F_K}) - A^{\beta} \right)^{(r)} \right)
    \le C r^{\frac{1}{2} \sum_{j=1}^r (l_j-1)(l_j+2)}
    = C r^{\sum_{j=1}^N \frac{1}{H_j}},
\end{align*}
for all $r>0$ small, and $C$ is a constant the depends only on $G,d,H_1,\ldots,H_N$. Here, $G|_{G_K}$ is the smooth restriction of $G$ on $F_K$, that is $G|_{F_K}$ coincide with $G$ on $F_K$ and $G|_{F_{K+1}^\complement} = 0$. By a similar argument as \eqref{Eq:Hit-up}, for $K$ large, we obtain
\begin{align*}
    &\bP \bigg( \left\{ \exists t \in I, \ \mathrm{s.t. \ } \left| s_{J_j}^{\beta}(t) \right| = 1 \ \mathrm{for \ some \ disjoint \ } J_j \in [d_1] \mathrm{\ with \ } |J_j| =  l_j, \forall 1 \le j \le r \right\} \\
    & \quad\quad \quad\quad \quad\quad \quad\quad \cap \left\{ \spec \left( Z^\beta \left( Z^\beta \right)^* \right) \subset \left[ 0,K^2 \right] \right\} \bigg) \\
    & \le \bP \left( \left\{ W^{\beta}(t) \in T_\beta^{-1} \left( \mathrm{Im}(G) - B^\beta \right) \tilde T_\beta^{-1} \mathrm{\ for \ some \ } t \in I \right\} \cap \left\{ \spec \left( Z^\beta \left( Z^\beta \right)^* \right) \subset \left[ 0,K^2 \right] \right\} \right) \\
    & \le \bP \left( W^{\beta}(I) \cap T_\beta^{-1} \left( \mathrm{Im}(G|_{F_K}) - B^\beta \right) \tilde T_\beta^{-1} \not= \emptyset \right).
\end{align*}
We deduce from Lemma \ref{Lemma-Hitting prb} and Corollary \ref{Coro-trans-critical} that
\begin{align*}
    &\bP \bigg( \left\{ \exists t \in I, \ \mathrm{s.t. \ } \left| s_{J_j}^{\beta}(t) \right| = 1 \ \mathrm{for \ some \ disjoint \ } J_j \in [d_1] \mathrm{\ with \ } |J_j| =  l_j, \forall 1 \le j \le r \right\} \\
    & \quad\quad \quad\quad \quad\quad \quad\quad \cap \left\{ \spec \left( Z^\beta \left( Z^\beta \right)^* \right) \subset \left[ 0,K^2 \right] \right\} \bigg) =0.
\end{align*}
Letting $K \to \infty$, we obtain
\begin{align*}
    \bP \left( \exists t \in I, \ \mathrm{s.t. \ } \left| s_{J_j}^{\beta}(t) \right| = 1 \ \mathrm{for \ some \ disjoint \ } J_j \in [d_1] \mathrm{\ with \ } |J_j| =  l_j, \forall 1 \le j \le r \right) =0.
\end{align*}

Next, we prove (ii). We choose $x_0 \in {\mathbf M}_{\bR}(d_1 \times d_2; l_1, \ldots, l_r)$ satisfying $0 \notin \spec(x_0x_0^*)$ and $|\spec(x_0x_0^*)| = l$. By Lemma \ref{Lem:dim-Wishart-bound}, there exists $\delta_0 > 0$, such that ${\mathbf M}_{\bR}(d_1 \times d_2; l_1, \ldots, l_r) \cap \mathfrak B_{\delta_0}(x_0)$ is an $(d_1d_2 - \frac{1}{2} \sum_{j=1}^r (l_j-1)(l_j+2))$-dimensional manifold. Thus, by an argument that is similar to \eqref{Eq:Hit-up}, we have
\begin{align*}
	&\bP \left( \exists t \in I, \ \mathrm{s.t. \ } \left| s_{J_j}^{\beta}(t) \right| = 1 \ \mathrm{for \ some \ disjoint \ } J_j \in [d_1] \mathrm{\ with \ } |J_j| =  l_j, \forall 1 \le j \le r \right) \\
	&= \bP \left( Z^{\beta}(t) \in {\mathbf M}_{\bR}(d_1 \times d_2; l_1, \ldots, l_r) \mathrm{\ for \ some \ } t \in I \right) \\
	&\ge \bP \left( W^{\beta}(I) \cap T_\beta^{-1} \left( {\mathbf M}_{\bR}(d_1 \times d_2; l_1, \ldots, l_r) \cap \mathfrak B_{\delta_0}(x_0) - B^{\beta} \right) \tilde T_\beta^{-1} \not= \emptyset \right) \\
	&\ge c_5 \cC_{d_1d_2-\sum_{j=1}^N \frac{1}{H_j}} \left( T_\beta^{-1} \left( {\mathbf M}_{\bR}(d_1 \times d_2; l_1, \ldots, l_r) \cap \mathfrak B_{\delta_0}(x_0) - B^{\beta} \right) \tilde T_\beta^{-1} \right) > 0.
\end{align*}

The proof of (iii) is to apply Lemma \ref{Lemma-Xiao-Hitting prob} as the proof of Theorem \ref{Thm-hitting prob-real}, and we omit the details. This finishes the proof.
\end{proof}

\subsection{The complex case: proof of Theorem \ref{Thm-hitting prob-compex-Wishart}} \label{sec:complex case-singular value}

In this section, we consider the matrix \eqref{def-Z} for the case $\beta = 2$.

The first lemma is the complex version of Lemma \ref{Lem:dim-M_R}, which shows that $2d_1d_2 - \sum_{j=1}^r (l_j-1)(l_j+1)$ is an upper bound for the dimension of ${\mathbf M}_{\bC}(d_1 \times d_2;l_1, \ldots, l_r)$.  

\begin{lemma} \label{Lem-complex-singular-cover}
Let $\Delta: \bR_{\ge 0}^l \rightarrow {\mathbf D}(d)$ be the function that maps the vector $u = (u_1, \ldots, u_l) \in \bR_{\ge 0}^l$ to the diagonal matrix $\Delta(u) = \{\Delta_{i,j}(u):1 \le i, j \le d\}$ given by \eqref{def-Lamda mapping}. Then there exist compactly supported smooth functions $\widehat{\Gamma}_L: \bR^{d_1^2 + d_2^2-d_2 - \sum_{j=0}^r l_j(l_j-1)} \rightarrow \mathbf U(d_1)$ and $\widehat{\Gamma}_R: \bR^{d_1^2 + d_2^2-d_2 - \sum_{j=0}^r l_j(l_j-1)} \rightarrow \mathbf U(d_2)$, such that the mapping
\begin{align*}
	\widehat{G}: \bR_{\ge 0}^l \times \bR^{d_1^2 + d_2^2-d_2 - \sum_{j=0}^r l_j(l_j-1)} \rightarrow \bC^{d_1 \times d_2}
\end{align*}
given by
\begin{align*}
	\widehat{G}(u, v) = \widehat{\Gamma}_L(v) \Delta(u) \widehat{\Gamma}_R(v)^*, ~~ 
	(u,v) \in \bR_{\ge 0}^l \times \bR^{d_1^2 + d_2^2-d_2 - \sum_{j=0}^r l_j(l_j-1)}
\end{align*}
satisfies
\begin{align*}
	{\mathbf M}_{\bC}(d_1 \times d_2;l_1, \ldots, l_r) \subseteq \left\{ x \in \bR^{d^2}: \ x \in \mathrm{Im}(\widehat{G}) \right\},
\end{align*}
where $\mathrm{Im}(\widehat{G})$ is the image of the mapping $\widehat{G}$.
\end{lemma}

\begin{proof}
The proof is similar to the proof of Lemma \ref{Lem:dim-M_R}, and is sketched below.

Since $\mathbf U(d_1)$ and $\mathbf U(d_2)$ are Lie groups of dimension $d_1^2$ and $d_2^2$ respectively, their direct product $\mathbf U(d_1) \times \mathbf U(d_2)$ is again a Lie group with dimension $d_1^2 + d_2^2$. Recall that $\mathbf U(d_2;l_1, \ldots, l_r, l_0)$ is a Lie group of dimension $d_2 + \sum_{j=0}^r l_j(l_j-1)$. For any $g \in \mathbf U(d_2;l_1, \ldots, l_r, l_0)$, one can write $g = \rd (g_{d_1}, g_{l_0})$, where $g_{d_1} \in \mathbf U(d_1;l_1, \ldots, l_r)$ and $g_{l_0} \in \mathbf U(l_0)$. In case that $l_0 = 0$, we use the convention $\mathbf U(d_2;l_1, \ldots, l_r, l_0) = \mathbf U(d_1;l_1, \ldots, l_r)$ and $g = g_{d_1}$. Then one can verify that the mapping $\left( \mathbf U(d_1) \times \mathbf U(d_2) \right) \times \mathbf U(d_2;l_1, \ldots, l_r, l_0) \to \mathbf U(d_1) \times \mathbf U(d_2)$ given by
\begin{align*}
	((p_1,p_2),g) \mapsto (p_1,p_2) \cdot g
	= (p_1 g_{d_1}, p_2 g)
\end{align*}
is the right action of the Lie group $\mathbf U(d_2;l_1, \ldots, l_r, l_0)$ on the manifold $\mathbf U(d_1) \times \mathbf U(d_2)$. Moreover, the right action is smooth, free and proper. Therefore, by Lemma \ref{Lemma-Quotient}, the orbit space $\left( \mathbf U(d_1) \times \mathbf U(d_2) \right) / \mathbf U(d_2;l_1, \ldots, l_r, l_0)$ is a smooth manifold of dimension $d_1^2 + d_2^2 - d_2 - \sum_{j=0}^r l_j(l_j-1)$, and the quotient map $\pi: \mathbf U(d_1) \times \mathbf U(d_2) \longrightarrow \left( \mathbf U(d_1) \times \mathbf U(d_2) \right) / \mathbf U(d_2;l_1, \ldots, l_r, l_0)$ is a smooth submersion. Furthermore, by Lemma \ref{Lemma-open}, $\pi$ is an open map.

For any $p \in \mathbf U(d_1) \times \mathbf U(d_2)$, since $\pi$ is a submersion, by Lemma \ref{Lemma:local section}, there exists a smooth local section $\sigma_{\pi}$ from an open nieghbourhood of $\pi(p)$ to $\mathbf U(d_1) \times \mathbf U(d_2)$ that maps $\pi(p)$ to $p$ and that $\pi \circ \sigma_{\pi}$ is identity. Moreover, by the definition of chart (see e.g. \cite{Tu2011}*{Definition 5.1}), there exist $\delta > 0$ and a smoothly diffeomorphism
\begin{align*}
	\phi: I_{\epsilon} \to \phi(I_{\epsilon}) \subset \left( \mathbf U(d_1) \times \mathbf U(d_2) \right) / \mathbf U(d_2;l_1, \ldots, l_r, l_0) \cap \pi(\mathfrak B_{\delta}(p))
\end{align*}
satisfying $\phi(0) = \pi(p)$, where
\begin{align*}
	I_{\epsilon} = (-\epsilon, \epsilon)^{d_1^2 + d_2^2-d_2 - \sum_{j=0}^r l_j(l_j-1)}.
\end{align*}
Then we can construct smooth function
\begin{align*}
    \Gamma = \left( \Gamma_L, \Gamma_R \right): \bR^{d_2^2 + d_1^2 - d_2 - \sum_{j=0}^r l_j(l_j-1)} \rightarrow {\mathbf U}(d_1) \times {\mathbf U}(d_2)
\end{align*}
that is supported on $I_{2\epsilon}$ and satisfies
\begin{align*}
	\left( \Gamma_L(v), \Gamma_R(v) \right)
    = \Gamma(v)
	= \sigma_{\pi} \circ \phi(v), \forall v \in I_{\epsilon}.
\end{align*}

The rest of the proof is similar to that of Lemma \ref{Lem:dim-M_R} and hence is omitted. 
\end{proof}

The following result is the complex version of Lemma \ref{Lem:approx-singular}. The proof is similar to that of Lemma \ref{Lem:approx-singular} 
and hence is omitted.

\begin{lemma}
Let $A \in {\mathbf M}_{\bC}(d_1 \times d_2; l_1, \ldots, l_r)$ with singular value decomposition $A = P_L D P_R^*$ for some $P_L \in {\mathbf U}(d_1)$, $P_R \in {\mathbf U}(d_2)$ and $D \in {\mathbf D}(d_1 \times d_2)$, such that $0 \notin \spec(AA^*)$ and $|\spec(AA^*)| = l$. Then for any $\epsilon > 0$, there exists $\delta > 0$, such that for all $B \in {\mathbf M}_{\bC}(d_1 \times d_2; l_1, \ldots, l_r)$ satisfying
\begin{align*}
	\max_{1 \le i, j \le d} |A_{i,j} - B_{i,j}| < \delta,
\end{align*}
we have $0 \notin \spec(BB^*)$ and $|\spec(BB^*)|=l$. Moreover, there exists a singular value decomposition  $B = Q_L FQ_R^*$, where $Q_L \in {\mathbf U}(d_1)$, $Q_R \in {\mathbf U}(d_2)$ and $F \in {\mathbf D}(d_1 \times d_2)$ satisfy
\begin{align*}
	\max_{1 \le i \le d_1} |D_{i,i} - F_{i,i}| < \epsilon, \quad
	\max_{1 \le i, j \le d_1} \left| \left( Q_L \right)_{i,j} - \left( P_L \right)_{i,j} \right| < \epsilon, \quad
	\max_{1 \le i, j \le d_2} \left| \left( Q_R \right)_{i,j} - \left( P_R \right)_{i,j} \right| < \epsilon.
\end{align*}
\end{lemma}

The following lemma is the complex version of Lemma \ref{Lem:dim-Wishart-bound}, which indicates that the upper bound $2d_1d_2 - \sum_{j=1}^r (l_j-1)(l_j+1)$ for the dimension of ${\mathbf M}_{\bC}(d_1 \times d_2; l_1, \ldots, l_r)$ is optimal. The proof is similar to the proof of Lemma \ref{Lem:dim-Wishart-bound}, and hence is omitted.

\begin{lemma} \label{Lem:dim-complex-Wishart-bound}
For  $x_0 \in {\mathbf M}_{\bC}(d_1 \times d_2; l_1, \ldots, l_r)$ with $0 \notin \spec(x_0x_0^*)$ and $|\spec(x_0x_0^*)| = l$, there exists $\delta_0 > 0$ such that ${\mathbf M}_{\bC}(d_1 \times d_2; l_1, \ldots, l_r) \cap \mathfrak B_{\delta_0}(x_0)$ is a $(2d_1d_2 - \sum_{j=1}^r (l_j-1)(l_j+1))$-dimensional manifold.
In particular, ${\mathbf M}_{\bC}(d_1 \times d_2; l_1, \ldots, l_r) \cap \mathfrak B_{\delta_0}(x_0)$ has positive $(2d_1d_2 - \sum_{j=1}^r (l_j-1)(l_j+1))$-dimensional Lebesgue measure.
\end{lemma}

Now we are ready to prove Theorem \ref{Thm-hitting prob-compex-Wishart}.

\begin{proof}[Proof of Theorem \ref{Thm-hitting prob-compex-Wishart}]
The proof is similar to that of Theorem \ref{Thm-hitting prob-real}, which is sketched below.
	
We first prove Part (i). By Lemma \ref{Lem-complex-singular-cover}, there exists a smooth map
\begin{align*}
    \widehat{G}: \bR_{\ge 0}^l \times \bR^{d_1^2 + d_2^2-d_2 - \sum_{j=0}^r l_j(l_j-1)} \rightarrow \bC^{d_1 \times d_2},
\end{align*}
such that ${\mathbf M}_{\bC}(d_1 \times d_2;l_1, \ldots, l_r) \subseteq \mathrm{Im}(\widehat{G})$. Together with Lemma \ref{Lem:dim-complex-Wishart-bound}, we derive that
\begin{align} \label{eq-dimh-complex-im G}
    \dimh \left( \mathrm{Im}(\widehat{G}) \right) = 2d_1d_2 - \sum_{j=1}^r (l_j-1)(l_j+1).
\end{align}
If
\begin{align*}
    \sum_{j=1}^N \frac1{H_j} < \sum_{j=1}^r (l_j-1)(l_j+1),
\end{align*}
then by Lemma \ref{Lemma-Xiao-Hitting prob}, we obtain
\begin{align*}
	&\bP \left( \exists t \in I, \ \mathrm{s.t. \ } \left| s_{J_j}^{\beta}(t) \right| = 1 \ \mathrm{for \ some \ disjoint \ } J_j \in [d_1] \mathrm{\ with \ } |J_j| =  l_j, \forall 1 \le j \le r \right) \\
	&= \bP \left( Z^{\beta}(t) \in {\mathbf M}_{\bC}(d_1 \times d_2; l_1, \ldots, l_r) \mathrm{\ for \ some \ } t \in I \right) \\
	&\le \bP \left( W^{\beta}(I) \cap T_\beta^{-1} \left( \mathrm{Im}(\widehat G) - B^{\beta} \right) \tilde T_\beta^{-1} \not= \emptyset \right) \\
	&\le c_6 {\mathbf H}_{2d_1d_2-\sum_{j=1}^N \frac1{H_j}} \left( T_\beta^{-1} \left( \mathrm{Im}(\widehat G) - B^{\beta} \right) \tilde T_\beta^{-1} \right) = 0,
\end{align*}
where we use the translation invariance of Hausdorff measure, \eqref{eq-dimh-complex-im G}, Corollary \ref{Coro-trans} and Lemma \ref{Lem-T-Wigner}.
If
\begin{align*}
    \sum_{j=1}^N \frac1{H_j} = \sum_{j=1}^r (l_j-1)(l_j+1),
\end{align*}
then by Lemma \ref{Lemma-condition-verify}, we can compute the Lebesgue measure
\begin{align*}
    \lambda_{2d_1d_2} \left( \left( \mathrm{Im} \left( \widehat G\big|_{F_K} \right) - B^\beta \right)^{(r)} \right)
    \le C r^{\sum_{j=1}^N \frac1{H_j}}
\end{align*}
for all $r>0$ small, and $C$ is a constant the depends only on $G,d,H_1,\ldots,H_N$. Here, $F_K = [0,K]^l \times F'$ is a compact interval with $F'$ being a compact interval of $\bR^{d_1^2 + d_2^2-d_2 - \sum_{j=0}^r l_j(l_j-1)}$ containing the support of both $\widehat \Gamma_L$ and $\widehat \Gamma_R$.
By a similar argument as \eqref{Eq:Hit-up}, for $K$ large, we obtain
\begin{align*}
    &\bP \bigg( \left\{ \exists t \in I, \ \mathrm{s.t. \ } \left| s_{J_j}^{\beta}(t) \right| = 1 \ \mathrm{for \ some \ disjoint \ } J_j \in [d_1] \mathrm{\ with \ } |J_j| =  l_j, \forall 1 \le j \le r \right\} \\
    & \quad\quad \quad\quad \quad\quad \quad\quad \cap \left\{ \spec \left( Z^\beta \left( Z^\beta \right)^* \right) \subset \left[0,K^2\right] \right\} \bigg) \\
    & \le \bP \left( \left\{ Z^{\beta}(t) \in T_\beta^{-1} \left( \mathrm{Im} \left( \widehat G \right) - B^\beta \right) \tilde T_\beta^{-1} \mathrm{\ for \ some \ } t \in I \right\} \cap \left\{ \spec \left( Z^\beta \left( Z^\beta \right)^* \right) \subset \left[0,K^2\right] \right\} \right) \\
    & \le \bP \left( W^{\beta}(I) \cap T_\beta^{-1} \left( \mathrm{Im} \left( \widehat G \big|_{F_K} \right) - B^\beta \right) \tilde T_\beta^{-1} \not= \emptyset \right).
\end{align*}
We deduce from Lemma \ref{Lemma-Hitting prb} and Corollary \ref{Coro-trans-critical} that
\begin{align*}
    &\bP \bigg( \left\{ \exists t \in I, \ \mathrm{s.t. \ } \left| s_{J_j}^{\beta}(t) \right| = 1 \ \mathrm{for \ some \ disjoint \ } J_j \in [d_1] \mathrm{\ with \ } |J_j| =  l_j, \forall 1 \le j \le r \right\} \\
    & \quad\quad \quad\quad \quad\quad \quad\quad \cap \left\{ \spec \left( Z^\beta \left( Z^\beta \right)^* \right) \subset \left[0,K^2\right] \right\} \bigg) =0.
\end{align*}
Letting $K \to \infty$, we obtain
\begin{align*}
    \bP \left( \exists t \in I, \ \mathrm{s.t. \ } \left| s_{J_j}^{\beta}(t) \right| = 1 \ \mathrm{for \ some \ disjoint \ } J_j \in [d_1] \mathrm{\ with \ } |J_j| =  l_j, \forall 1 \le j \le r \right) =0.
\end{align*}

Next, we prove (ii). We choose $x_0 \in {\mathbf M}_{\bC}(d_1 \times d_2; l_1, \ldots, l_r)$ satisfying $0 \notin \spec(x_0x_0^*)$ and $|\spec(x_0x_0^*)| = l$. By Lemma \ref{Lem:dim-complex-Wishart-bound}, there exists $\delta_0 > 0$, such that ${\mathbf M}_{\bR}(d_1 \times d_2; l_1, \ldots, l_r) \cap \mathfrak B_{\delta_0}(x_0)$ is an $(2d_1d_2 - \sum_{j=1}^r (l_j-1)(l_j+1))$-dimensional manifold. Thus, by Lemma \ref{Lemma-Xiao-Hitting prob} together with Corollary \ref{Coro-trans} and Lemma \ref{Lem-T-Wigner}, when
\begin{align*}
    \sum_{j=1}^N \frac{1}{H_j} > \sum_{j=1}^r (l_j-1)(l_j+1),
\end{align*}
we have
\begin{align*}
	&\bP \left( \exists t \in I, \ \mathrm{s.t. \ } \left| s_{J_j}^{\beta}(t) \right| = 1 \ \mathrm{for \ some \ disjoint \ } J_j \in [d_1] \mathrm{\ with \ } |J_j| =  l_j, 1 \le j \le r \right) \\
	&= \bP \left( Z^{\beta}(t) \in {\mathbf M}_{\bC}(d_1 \times d_2; l_1, \ldots, l_r) \mathrm{\ for \ some \ } t \in I \right) \\
	&\ge \bP \left( W^{\beta}(I) \in T_\beta^{-1} \left( {\mathbf M}_{\bC}(d_1 \times d_2; l_1, \ldots, l_r) \cap \mathfrak B_{\delta_0}(x_0) - B^{\beta} \right) \tilde T_\beta^{-1} \not= \emptyset \right) \\
	&\ge c_5 \cC_{2d_1d_2-\sum_{j=1}^N \frac{1}{H_j}} \left( T_\beta^{-1} \left( {\mathbf M}_{\bC}(d_1 \times d_2; l_1, \ldots, l_r) \cap \mathfrak B_{\delta_0}(x_0) - B^{\beta} \right) \tilde T_\beta^{-1} \right) > 0.
\end{align*}

The proof of (iii) is similar to that in the proof of Theorem  \ref{Thm-hitting prob-real} and is omitted. This finishes the proof.
\end{proof}

\section*{Acknowledgments} 

The author gratefully acknowledges the financial support of ERC Consolidator Grant 815703 ”STAMFORD: Statistical Methods for High Dimensional Diffusions”. Besides, the author would like to acknowledge Jian Song and Jianfeng Yao for the discussion and helpful suggestions.

\bibliographystyle{plain}
\bibliography{Collision_of_singularvalue.bib}

\begin{bibdiv}
\begin{biblist}

\bib{anderson2010}{book}{
      author={Anderson, Greg~W.},
      author={Guionnet, Alice},
      author={Zeitouni, Ofer},
       title={An introduction to random matrices},
      series={Cambridge Studies in Advanced Mathematics},
   publisher={Cambridge University Press, Cambridge},
        date={2010},
      volume={118},
        ISBN={978-0-521-19452-5},
      review={\MR{2760897}},
}

\bib{Xiao2009}{article}{
      author={Bierm\'{e}, Hermine},
      author={Lacaux, C\'{e}line},
      author={Xiao, Yimin},
       title={Hitting probabilities and the {H}ausdorff dimension of the inverse images of anisotropic {G}aussian random fields},
        date={2009},
        ISSN={0024-6093},
     journal={Bull. Lond. Math. Soc.},
      volume={41},
      number={2},
       pages={253\ndash 273},
         url={https://doi.org/10.1112/blms/bdn122},
      review={\MR{2496502}},
}

\bib{Bru1989}{article}{
      author={Bru, Marie-France},
       title={Diffusions of perturbed principal component analysis},
        date={1989},
        ISSN={0047-259X},
     journal={J. Multivariate Anal.},
      volume={29},
      number={1},
       pages={127\ndash 136},
         url={https://doi.org/10.1016/0047-259X(89)90080-8},
      review={\MR{991060}},
}

\bib{dalang2017polarity}{article}{
      author={Dalang, Robert~C.},
      author={Mueller, Carl},
      author={Xiao, Yimin},
       title={Polarity of points for {G}aussian random fields},
        date={2017},
        ISSN={0091-1798},
     journal={Ann. Probab.},
      volume={45},
      number={6B},
       pages={4700\ndash 4751},
         url={https://doi.org/10.1214/17-AOP1176},
      review={\MR{3737922}},
}

\bib{Dyson1962}{article}{
      author={Dyson, Freeman~J.},
       title={A {B}rownian-motion model for the eigenvalues of a random matrix},
        date={1962},
        ISSN={0022-2488},
     journal={J. Mathematical Phys.},
      volume={3},
       pages={1191\ndash 1198},
         url={https://doi.org/10.1063/1.1703862},
      review={\MR{148397}},
}

\bib{Yau2017}{book}{
      author={Erd\H{o}s, L\'{a}szl\'{o}},
      author={Yau, Horng-Tzer},
       title={A dynamical approach to random matrix theory},
      series={Courant Lecture Notes in Mathematics},
   publisher={Courant Institute of Mathematical Sciences, New York; American Mathematical Society, Providence, RI},
        date={2017},
      volume={28},
        ISBN={978-1-4704-3648-3},
      review={\MR{3699468}},
}

\bib{Falconer2014}{book}{
      author={Falconer, Kenneth},
       title={Fractal geometry},
     edition={Third},
   publisher={John Wiley \& Sons, Ltd., Chichester},
        date={2014},
        ISBN={978-1-119-94239-9},
        note={Mathematical foundations and applications},
      review={\MR{3236784}},
}

\bib{Forrester2010}{book}{
      author={Forrester, P.~J.},
       title={Log-gases and random matrices},
      series={London Mathematical Society Monographs Series},
   publisher={Princeton University Press, Princeton, NJ},
        date={2010},
      volume={34},
        ISBN={978-0-691-12829-0},
         url={https://doi.org/10.1515/9781400835416},
      review={\MR{2641363}},
}

\bib{Graczyk2019}{article}{
      author={Graczyk, Piotr},
      author={Ma{\l}ecki, Jacek},
       title={On squared {B}essel particle systems},
        date={2019},
        ISSN={1350-7265,1573-9759},
     journal={Bernoulli},
      volume={25},
      number={2},
       pages={828\ndash 847},
         url={https://doi.org/10.3150/17-bej997},
      review={\MR{3920358}},
}

\bib{Jaramillo2018}{article}{
      author={Jaramillo, Arturo},
      author={Nualart, David},
       title={Collision of eigenvalues for matrix-valued processes},
        date={2020},
        ISSN={2010-3263},
     journal={Random Matrices Theory Appl.},
      volume={9},
      number={4},
       pages={2030001, 26},
         url={https://doi.org/10.1142/S2010326320300016},
      review={\MR{4133067}},
}

\bib{Lax2002}{book}{
      author={Lax, Peter~D.},
       title={Functional analysis},
      series={Pure and Applied Mathematics (New York)},
   publisher={Wiley-Interscience [John Wiley \& Sons], New York},
        date={2002},
        ISBN={0-471-55604-1},
      review={\MR{1892228}},
}

\bib{Lee2023}{article}{
      author={Lee, Cheuk~Yin},
      author={Song, Jian},
      author={Xiao, Yimin},
      author={Yuan, Wangjun},
       title={Hitting probabilities of {G}aussian random fields and collision of eigenvalues of random matrices},
        date={2023},
        ISSN={0002-9947,1088-6850},
     journal={Trans. Amer. Math. Soc.},
      volume={376},
      number={6},
       pages={4273\ndash 4299},
         url={https://doi.org/10.1090/tran/8895},
      review={\MR{4586811}},
}

\bib{Lee2013}{book}{
      author={Lee, John~M.},
       title={Introduction to smooth manifolds},
     edition={Second},
      series={Graduate Texts in Mathematics},
   publisher={Springer, New York},
        date={2013},
      volume={218},
        ISBN={978-1-4419-9981-8},
      review={\MR{2954043}},
}

\bib{mattila1999}{book}{
      author={Mattila, Pertti},
       title={Geometry of sets and measures in euclidean spaces: fractals and rectifiability},
   publisher={Cambridge university press},
        date={1999},
      number={44},
}

\bib{McKean1969}{book}{
      author={McKean, H.~P., Jr.},
       title={Stochastic integrals},
      series={Probability and Mathematical Statistics, No. 5},
   publisher={Academic Press, New York-London},
        date={1969},
      review={\MR{0247684}},
}

\bib{Mehta2004}{book}{
      author={Mehta, Madan~Lal},
       title={Random matrices},
     edition={Third},
      series={Pure and Applied Mathematics (Amsterdam)},
   publisher={Elsevier/Academic Press, Amsterdam},
        date={2004},
      volume={142},
        ISBN={0-12-088409-7},
      review={\MR{2129906}},
}

\bib{Nualart2014}{article}{
      author={Nualart, David},
      author={P\'{e}rez-Abreu, Victor},
       title={On the eigenvalue process of a matrix fractional {B}rownian motion},
        date={2014},
        ISSN={0304-4149},
     journal={Stochastic Process. Appl.},
      volume={124},
      number={12},
       pages={4266\ndash 4282},
         url={https://doi.org/10.1016/j.spa.2014.07.017},
      review={\MR{3264448}},
}

\bib{Pardo2017}{article}{
      author={Pardo, Juan~Carlos},
      author={P\'{e}rez, Jos\'{e}-Luis},
      author={P\'{e}rez-Abreu, Victor},
       title={On the non-commutative fractional {W}ishart process},
        date={2017},
        ISSN={0022-1236},
     journal={J. Funct. Anal.},
      volume={272},
      number={1},
       pages={339\ndash 362},
         url={https://doi.org/10.1016/j.jfa.2016.10.008},
      review={\MR{3567507}},
}

\bib{Xiao2021}{article}{
      author={Song, Jian},
      author={Xiao, Yimin},
      author={Yuan, Wangjun},
       title={On collision of multiple eigenvalues for matrix-valued {G}aussian processes},
        date={2021},
     journal={J. Math. Anal. Appl.},
      volume={502},
      number={2},
       pages={125261},
}

\bib{Yuan2023}{article}{
      author={Song, Jian},
      author={Xiao, Yimin},
      author={Yuan, Wangjun},
       title={On eigenvalues of the {B}rownian sheet matrix},
        date={2023},
        ISSN={0304-4149,1879-209X},
     journal={Stochastic Process. Appl.},
      volume={166},
       pages={Paper No. 104231, 38},
         url={https://doi.org/10.1016/j.spa.2023.104231},
      review={\MR{4654811}},
}

\bib{Tu2011}{book}{
      author={Tu, Loring~W.},
       title={An introduction to manifolds},
     edition={Second},
      series={Graduate Texts in Mathematics},
   publisher={Springer, New York},
        date={2011},
      volume={218},
        ISBN={978-1-4419-9981-8},
      review={\MR{2954043}},
}

\bib{xiao2009sample}{incollection}{
      author={Xiao, Yimin},
       title={Sample path properties of anisotropic {G}aussian random fields},
        date={2009},
   booktitle={A minicourse on stochastic partial differential equations},
      series={Lecture Notes in Math.},
      volume={1962},
   publisher={Springer, Berlin},
       pages={145\ndash 212},
         url={https://doi.org/10.1007/978-3-540-85994-9_5},
      review={\MR{2508776}},
}

\end{biblist}
\end{bibdiv}

\end{document}